\documentclass[12pt, reqno]{amsart}
\usepackage[colorlinks]{hyperref}
\usepackage{amsmath,amssymb,latexsym, amscd}
\usepackage{mathabx}
\usepackage{mathtools}
\usepackage{exscale, cite, eps fig, graphics}
\usepackage{float}
\usepackage{mathtools}
\usepackage{graphicx}
\usepackage{import, pst-plot, comment}
\usepackage{subcaption}

\usepackage[shortlabels]{enumitem}
\usepackage{lipsum}

\usepackage{graphicx}
\usepackage{soul, comment}
\usepackage[mathscr]{euscript}
\usepackage{longtable} 
\usepackage{tcolorbox}
\usepackage{array}
\usepackage{multirow}
\usepackage{mathrsfs}
\newcommand{\no}{\noindent}
\renewcommand{\Re}{\mbox{Re}}

\newcommand\MyBox[2]{
  \fbox{\lower0.75cm
    \vbox to 1.7cm{\vfil
      \hbox to 1.7cm{\hfil\parbox{1.4cm}{#1\\#2}\hfil}
      \vfil}%
  }%
}

\calclayout

\renewcommand{\geq}{\geqslant}
\renewcommand{\leq}{\leqslant}

\newcommand{\J}{\mathcal{J}}

\newcommand{\R}{\mathbb{R}}
\newcommand{\Z}{\mathbb{Z}}
\newcommand{\T}{\mathbb{T}}
\newcommand{\C}{\mathbb{C}}
\newcommand{\N}{\mathbb{N}}

\newtheorem{theorem}{Theorem}[section]
\newtheorem{lemma}[theorem]{Lemma}
\newtheorem{proposition}[theorem]{Proposition}

\newtheorem{definition}[theorem]{Definition}

\newtheorem*{main-theorem}{Main Theorem}
\newtheorem*{remark*}{Remark}

\newtheorem{hypothesis*}[theorem]{Hypothesis}
\numberwithin{equation}{section}

\newcommand{\beq}{\begin{equation}}
\newcommand{\eeq}{\end{equation}}

\title[The Benjamin-Feir instability in generalized KdV equations]{The Benjamin-Feir instability in KdV-like equations with general dispersion and monomial nonlinearity}
\author[Kaushik]{Bhavna~Kaushik$^\dagger$}
\author[Deconinck]{Bernard~Deconinck$^\dagger$}
\address{$^\dagger$Department of Applied Mathematics, University of Washington, Seattle, WA 98195-3925, USA
}
\email{bhavnak@uw.edu, deconinc@uw.edu
}
\date{\today}

\begin{document}

\begin{abstract}

Nonlinear waves in dispersive media 
can be succeptible to modulational instabilities. 
We examine a category of scalar equations, 
with general dispersion and monomial nonlinearity, including a large variety of KdV-like equations. For small-amplitude traveling wave solutions, we provide a complete characterization of the spectrum near the origin of the linear operator obtained from linearizing about periodic traveling waves. We prove rigorously that, when the modulational instability is present, the spectrum connected to the origin consists of curves that invariably form a closed figure-eight pattern.


\end{abstract}

\maketitle

\section{Introduction}

We study the stability of small-amplitude solutions of a large family of scalar partial differential equations of KdV type, 

\begin{equation}\label{e:gkdv}
    u_t+(\mathcal{J}u+\alpha u^{N})_x=0, \quad x,t\in\R,
\end{equation}

\no where $N\in\N\setminus\{1\}$ and $u=u(x,t)$ is real valued. The indices $x$ and $t$ denote partial differentiation. Here the operator $\J$ is defined by its Fourier symbol

\begin{equation}\label{e:M}
\widehat{\mathcal{J}f}(k)=\jmath(k)\hat{f}(k),  
\end{equation}

\no and $\hat f(k)$ is the Fourier transform of $f(x)$: 

\[
\hat f(k)=\int_{\mathbb{R}} e^{-ikx} f(x) dx. 
\]

\no This allows for the incorporation of a variety of linear dispersive phenomena, see below. 
The parameter $\alpha\in\{-1,1\}$ is related to the focusing or defocusing nature of the equation. For even $N$, $\alpha$ is chosen to be 1, without loss of generality through the use of a scaling transformation. For odd $N$, the sign of $\alpha$ matters as no real-valued scaling transformation switches between the two equations. 

Our study is motivated by the long history of considering different dispersion models and different nonlinearities, briefly reviewed next. The interplay between nonlinearity and dispersion in nonlinear dispersive evolution equations in scalar partial differential equations starts with the advected Korteweg-de Vries (KdV) equation,

\[
u_t+ c u_x-u_{xxx}+\left(u^2\right)_x=0,
\]

\no where $c$ is a real parameter. This equation is a generic model for the unidirectional propagation of weakly nonlinear, small-amplitude waves in the long-wavelength regime in a dispersive medium, {e.g.}, long waves in shallow water (see \cite{Ablowitz1981SolitonsTransform, Lannes2013TheAsymptotics} and references therein).

It is well known that while the KdV equation effectively describes long waves in dispersive media, such as traveling solitary and periodic waves, it does not capture high-frequency effects. These include wave breaking, which involves the formation of bounded solutions with infinite gradients, and peaking, which refers to the existence of bounded steady solutions featuring singular points such as peaks or cusps. In the case of water waves, this is manifested by the phase velocity derived from the linear part of the KdV equation poorly approximating the actual phase velocity of water waves outside the long-wavelength regime, for instance. For unidirectional water waves governed by the Euler water wave problem \cite{Lannes2013TheAsymptotics}, the phase speed can be expanded as

\[
v_p(k)=\sqrt{\frac{\tanh(k)}{k}} =1-\frac{1}{6}k^2+\mathcal{O}(k^4),~~|k|\ll 1,
\]

\no where $k$ represents the wavenumber. The first two terms on the right give rise to the linear KdV equation, but they provide a poor approximation for large values of $k$.   

To address water wave phenomena beyond the long-wavelength regime in the context of a simple scalar equation, Whitham proposed the model

\begin{equation}\label{e:whitham}
u_t+\mathcal{J}u_x+\left(u^2\right)_x=0,
\end{equation}

\no now referred to as the Whitham equation \cite[p. 477]{W74}. Here $\widehat{\mathcal{J}f}(k) = v_p(k)\hat{f}(k)$. The Whitham equation incorporates the phase speed of unidirectional Euler water waves along with the generic KdV shallow-water nonlinearity, providing a paradigm incorporating both wave breaking and peaking phenomena. Recent studies have confirmed its efficacy in modeling these effects (see \cite{Hur_breaking, EK_global, EW_WhithamConj}). 

Recently, following Whitham's lead, many models incorporating additional physical effects such as surface tension, constant vorticity, and bidirectional wave propagation with dispersion specified   through a Fourier symbol (see, for example,  \cite{HJ_SurfaceTension, CJ_num1, EJC, EJMR, JW, JTW, Carter_2018} and references therein) have been proposed and examined.
In general, such scalar models are given by \eqref{e:gkdv}
where the operator $\J$ takes on different functional forms:
%
%
$\jmath(k) = 1-|k|^\beta$ $(\beta > 1)$ in \eqref{e:whitham} yields the Fractional KdV (fKdV) equation \cite{Johnson2013StabilityEquations}, $\jmath(k) = 1-|k|$ corresponds to the Benjamin-Ono (BO) equation~\cite{BO}, $\jmath(k)=k\coth k$ gives the Intermediate Long wave (ILW) equation \cite{ILW} and $\jmath(k)=\sqrt{\tanh k/k}$ reproduces the original Whitham equation \eqref{e:whitham} \cite{W74}. The Kawahara equation \cite{kawahara1972oscillatory}, which models capillary–gravity waves and other higher-order dispersive phenomena, arises when $\jmath(k) = a k^2 + b k^4$ for $a \in \mathbb{R}$ and $b > 0$.

Similarly, one may consider more general nonlinearities, as in the case of the modified KdV equation \cite{Ablowitz1981SolitonsTransform} or the so-called generalized KdV equation, see~\cite{aud24}, for instance. This leads to \eqref{e:M} as a scalar model incorporating general dispersion and general monomial nonlinearity. 

The modulational instability of periodic traveling waves is a phenomenon that describes the instability of these waves in dispersive media when subjected to long-wave perturbations. This instability leads to the breakdown of wave trains and has been a topic of study since the pioneering works of Benjamin~\cite{B67} and Whitham~\cite{W67}. In the context of surface water waves, the instability is referred to as the Benjamin-Feir instability, recognizing the influential experimental work presented in~\cite{BF}. The instability is found in plasma physics, nonlinear optics and many other applications, as discussed in the review article by Zakharov \& Ostrovsky~\cite{OZ09}.

The modulational instability manifests itself by the presence of an $X$-shaped cross at the origin in the spectrum of the linear operator obtained by linearizing about the periodic wave, in a frame of reference moving with the wave's speed. In order to determine which modes are the most unstable ({i.e.}, correspond to the spectral element with the largest real part), it is necessary to determine the whole spectral component connected to the cross. For surface water waves, numerical work~\cite{DO} showed the cross to be the central part of a figure-8 curve. This was proven in the works of Berti, Maspero \& Ventura~\cite{Berti2022Benjamin-FeirDepth,Berti2021FullWater}. For simpler models like the Nonlinear Schr\"odinger (NLS) equation, the presence of a figure 8 at the origin of the spectral plane was known for constant amplitude solutions. For non-constant solutions, this was shown in~\cite{DS, DU}. Recently, using the methods of \cite{Berti2022Benjamin-FeirDepth,Berti2021FullWater}, Maspero \& Radakovic~\cite{MasperoRadakovic2024} considered scalar equations with general dispersion and quadratic nonlinearity $u u_x$. They show that if a modulational instability is present, it is part of a figure 8 at the origin. Their work is most closely related to ours, as we generalize their results, using similar techniques, to consider equations with general monomial nonlinearity.

For the modulational instability, the use of this method provides explicit expressions for the spectrum near the origin as real analytic functions of the Floquet exponent and the amplitude of the periodic waves. 

The multiplier $\jmath(k)$ in \eqref{e:M} is assumed to satisfy the following hypotheses. 

\phantomsection\label{h:j}

\vspace*{0.1in}

\begin{enumerate}
    \item[(H1)] The dispersion function $\jmath(k)$ is real valued, even, and, without loss of generality, $\jmath(0)=1$.
    \item[(H2)] There exists constants $C_1,C_2 > 0$ and $\sigma\geq -1$ such that for $ k \gg 1$,
    \[
    C_1 k^\sigma \leq \jmath(k) \leq C_2 k^\sigma.
    \]

\item[(H3)] For each fixed $n\in \mathbb{N}, n\neq 1$, we have 
\[
\jmath(k n)-\jmath(k)\neq0, ~~~ k>0. 
\]
\end{enumerate}

\vspace*{0.1in}

\no The hypotheses (H1)-(H3) are crucial for proving the existence of small-amplitude periodic traveling waves of \eqref{e:gkdv}; see Section \ref{s:2}. It should be noted that $\jmath(0)=1$ in Hypothesis~H1 can always be achieved using a scaling transformation on the independent and dependent variables, perhaps altering other coefficients in the equation if it is not scale invariant. Hypothesis (H3) excludes the possibility of a resonance between the fundamental Fourier mode of the solution and its higher harmonics. For more details, see the discussion directly preceding the statement of Theorem~\ref{t:1}.

The problem is mathematically formulated as follows. We consider periodic traveling waves of the gKdV equation \eqref{e:gkdv} with amplitude $0<a\ll1$. A standard linearization leads to a spectral problem $\lambda \vartheta = \mathcal{Q} \vartheta$, where  the linear operator $\mathcal{Q}$ has periodic coefficients determined by the traveling wave and its derivatives. Thus it depends on $a$. Using a Floquet decomposition~\cite{DK, KP}, this leads to a one-parameter family of spectral problems  $\lambda \varphi = \mathcal{Q}^{(\mu)} \varphi$, parameterized by the Floquet parameter $\mu$. Here $\vartheta(x) = e^{i\mu x}\varphi(x)$, where $\varphi(x)$ is a $2\pi$-periodic function. If spectral elements $\lambda$ of $\vartheta^{(\mu)}$ exist for any $\mu$ that have a positive real part, the perturbation  of the traveling wave grows exponentially in time, establishing spectral instability.

We state our main results in colloquial language here. Precise technical statements are found in Section~\ref{sec:results}. All results are for waves of sufficiently small amplitude: using an expansion in terms of the amplitude (a so-called Stokes expansion), only the leading correction to the wave profile satisfying the linearized equation is taking into account. 

\begin{itemize}

\item For a given equation \eqref{e:gkdv}, i.e., for a given $N$, $\alpha$ and $\jmath(k)$, we construct its Whitham-Benjamin coefficient $\Delta(k)$. If $\Delta(k)>0$ then traveling waves of period $2\pi/k$ of \eqref{e:gkdv} are unstable with respect to the Benjamin-Feir or modulational instability. For $N=2$, our results agree with those obtained by Maspero \& Radakovic \cite{MasperoRadakovic2024}. 

\item If $N$ is even, the modulational stability criterion is quite involved, see Section~\ref{sec:results}. On the other hand, if $N$ is odd, the modulational instability criterion can be written down succinctly: $2\pi/k$-periodic traveling wave solutions of \eqref{e:gkdv} are unstable with respect to the Benjamin-Feir or modulational instability if 

\[
\alpha\left(k j'(k)+\dfrac{1}{2}k^2\jmath^{\prime\prime}(k)\right)<0.
\] 

\no It is noteworthy that for odd $N$, this criterion does not depend on $N$. 

\item If traveling wave solutions of \eqref{e:gkdv} are unstable with respect to the Benjamin-Feir or modulational instability, then the unstable spectrum of the linear operator $\mathcal{Q}$ in a neighborhood of the origin lies on a figure-eight curve centered at the origin whose size (height and width) decreases exponentially with $N$. 

\end{itemize}

A comparison of the analytical result for the unstable spectrum near the origin with numerical computations using the FFHM method \cite{DK} linearizing about a high-order correction Stokes expansion shows excellent agreement for waves of sufficiently small amplitude, see Figs.~\ref{fig:mkdvcomp} and \ref{fig:whithamcomp}. 
In Figs. \ref{fig:mkdvcomp} and \ref{fig:whithamcomp} ($N=3$), a gap is visible in the spectrum on the imaginary axis. This is a consequence of restricting the interval for the Floquet exponent $\mu$ so as to get better numerical coverage of the figure eight.

\begin{figure}[tb]
\begin{center}
\makebox[\textwidth][c]{\def\svgwidth{5in}\hspace*{1.0in}
\begingroup%
  \makeatletter%
  \providecommand\color[2][]{%
    \errmessage{(Inkscape) Color is used for the text in Inkscape, but the package 'color.sty' is not loaded}%
    \renewcommand\color[2][]{}%
  }%
  \providecommand\transparent[1]{%
    \errmessage{(Inkscape) Transparency is used (non-zero) for the text in Inkscape, but the package 'transparent.sty' is not loaded}%
    \renewcommand\transparent[1]{}%
  }%
  \providecommand\rotatebox[2]{#2}%
  \newcommand*\fsize{\dimexpr\f@size pt\relax}%
  \newcommand*\lineheight[1]{\fontsize{\fsize}{#1\fsize}\selectfont}%
  \ifx\svgwidth\undefined%
    \setlength{\unitlength}{2992.85087045bp}%
    \ifx\svgscale\undefined%
      \relax%
    \else%
      \setlength{\unitlength}{\unitlength * \real{\svgscale}}%
    \fi%
  \else%
    \setlength{\unitlength}{\svgwidth}%
  \fi%
  \global\let\svgwidth\undefined%
  \global\let\svgscale\undefined%
  \makeatother%
  \begin{picture}(1,0.50246693)%
    \lineheight{1}%
    \setlength\tabcolsep{0pt}%
    \put(0,0){\includegraphics[width=\unitlength,page=1]{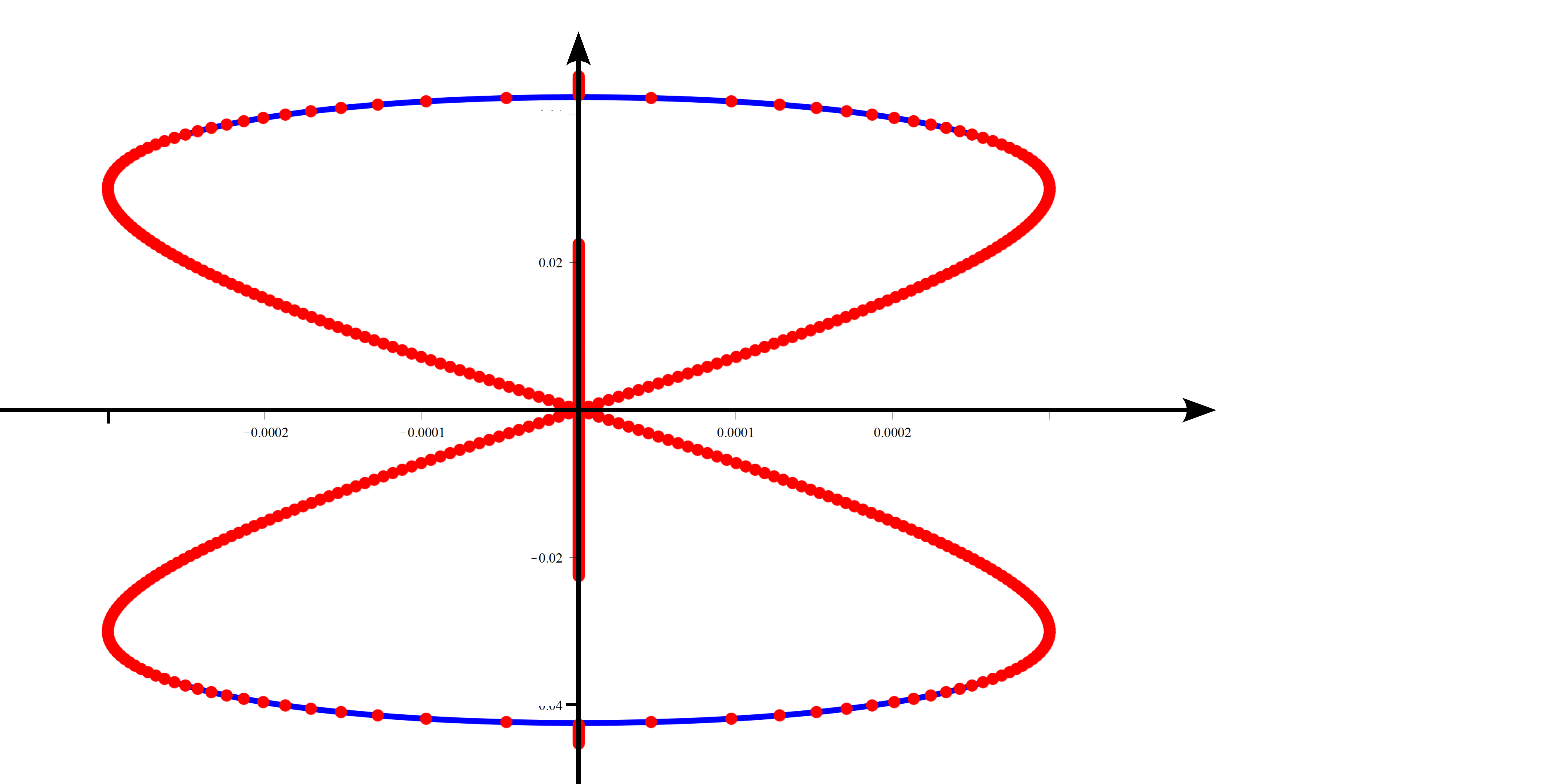}}%
    \put(0.65022136,0.21135526){\makebox(0,0)[lt]{\lineheight{1.25}\smash{\begin{tabular}[t]{l}$0.0003$\end{tabular}}}}%
    \put(0.38198449,0.47200373){\makebox(0,0)[lt]{\lineheight{1.25}\smash{\begin{tabular}[t]{l}Im$\lambda$\end{tabular}}}}%
    \put(0.73783249,0.25649014){\makebox(0,0)[lt]{\lineheight{1.25}\smash{\begin{tabular}[t]{l}Re$\lambda$\end{tabular}}}}%
    \put(0,0){\includegraphics[width=\unitlength,page=2]{mkdv.pdf}}%
    \put(0.30289367,0.41382997){\makebox(0,0)[lt]{\lineheight{1.25}\smash{\begin{tabular}[t]{l}$0.04$\end{tabular}}}}%
    \put(0,0){\includegraphics[width=\unitlength,page=3]{mkdv.pdf}}%
    \put(0.27773228,0.05046402){\makebox(0,0)[lt]{\lineheight{1.25}\smash{\begin{tabular}[t]{l}$-0.04$\end{tabular}}}}%
    \put(0,0){\includegraphics[width=\unitlength,page=4]{mkdv.pdf}}%
    \put(0.05011942,0.21050166){\makebox(0,0)[lt]{\lineheight{1.25}\smash{\begin{tabular}[t]{l}$-0.0003$\end{tabular}}}}%
  \end{picture}%
\endgroup%
}
\caption{Comparison of the analytical result (blue) for the unstable spectrum near the origin for the modified KdV equation with $\alpha=-1$ (thus $N=3$, $\jmath(k)=1+k^2$), with numerical results (red) using a 9-th order Stokes expansion, with $a=0.02$, $\rho=1.5$. Hill's method \cite{DK} uses $5$ Fourier modes, $11\times 11$ matrices, and $201$ equally spaced Floquet exponents $\mu\in [-0.01, 0.01]$.}
\label{fig:mkdvcomp}
\end{center}
\end{figure}
\begin{figure}[tb]
\centering

\begin{tabular}{cc}
\def\svgwidth{3.5in}
\hspace*{0.0in}
\begingroup%
  \makeatletter%
  \providecommand\color[2][]{%
    \errmessage{(Inkscape) Color is used for the text in Inkscape, but the package 'color.sty' is not loaded}%
    \renewcommand\color[2][]{}%
  }%
  \providecommand\transparent[1]{%
    \errmessage{(Inkscape) Transparency is used (non-zero) for the text in Inkscape, but the package 'transparent.sty' is not loaded}%
    \renewcommand\transparent[1]{}%
  }%
  \providecommand\rotatebox[2]{#2}%
  \newcommand*\fsize{\dimexpr\f@size pt\relax}%
  \newcommand*\lineheight[1]{\fontsize{\fsize}{#1\fsize}\selectfont}%
  \ifx\svgwidth\undefined%
    \setlength{\unitlength}{2996.60075511bp}%
    \ifx\svgscale\undefined%
      \relax%
    \else%
      \setlength{\unitlength}{\unitlength * \real{\svgscale}}%
    \fi%
  \else%
    \setlength{\unitlength}{\svgwidth}%
  \fi%
  \global\let\svgwidth\undefined%
  \global\let\svgscale\undefined%
  \makeatother%
  \begin{picture}(1,0.57449588)%
    \lineheight{1}%
    \setlength\tabcolsep{0pt}%
    \put(0,0){\includegraphics[width=\unitlength,page=1]{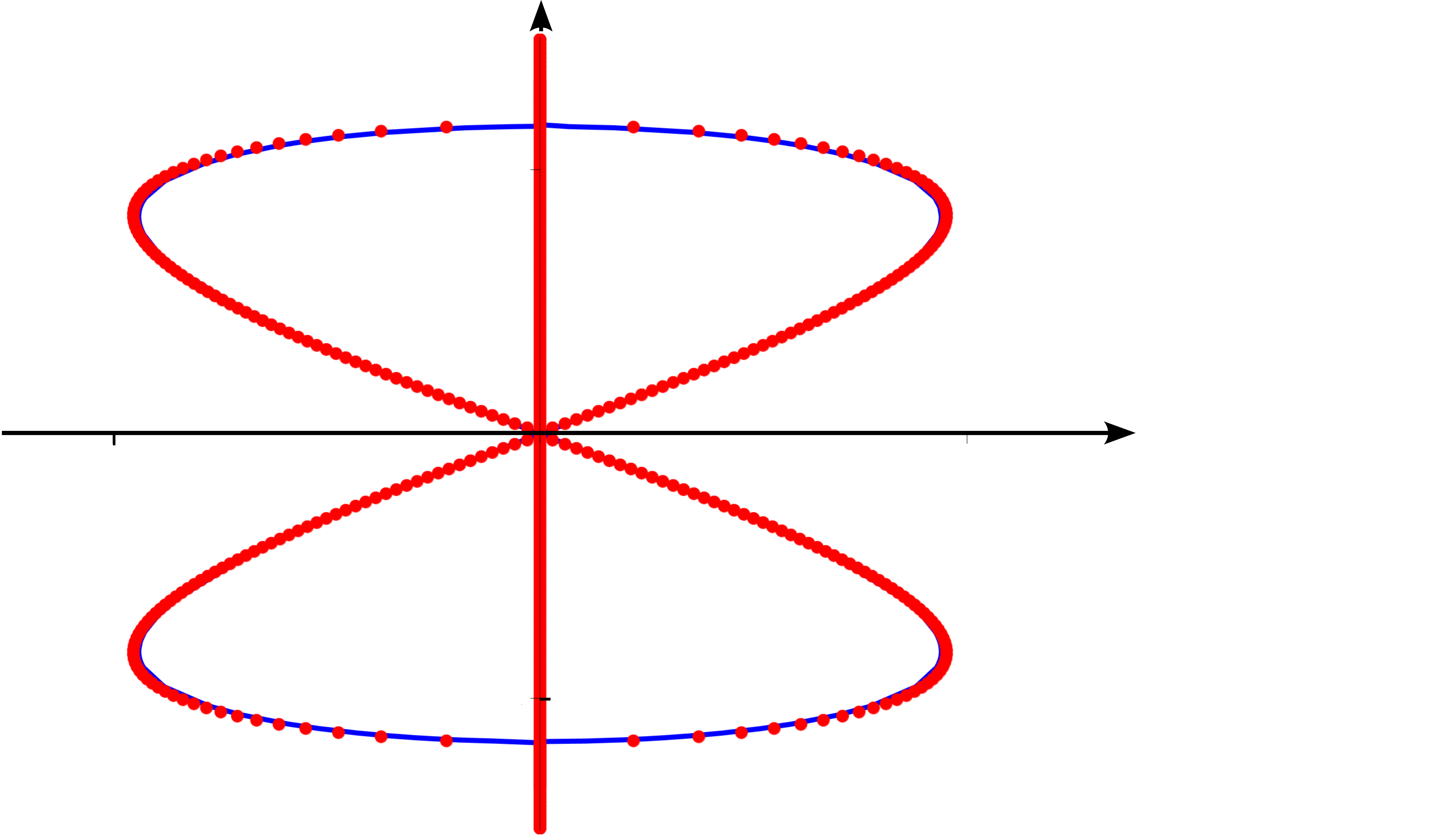}}%
    \put(0.38095409,0.09183672){\makebox(0,0)[lt]{\lineheight{1.25}\smash{\begin{tabular}[t]{l}$-0.01$\end{tabular}}}}%
    \put(0.38183289,0.4489347){\makebox(0,0)[lt]{\lineheight{1.25}\smash{\begin{tabular}[t]{l}$0.01$\end{tabular}}}}%
    \put(0.0613858,0.23785288){\makebox(0,0)[lt]{\lineheight{1.25}\smash{\begin{tabular}[t]{l}$-0.0002$\end{tabular}}}}%
    \put(0.63314765,0.23681615){\makebox(0,0)[lt]{\lineheight{1.25}\smash{\begin{tabular}[t]{l}$0.0002$\end{tabular}}}}%
    \put(0.38168242,0.53859695){\makebox(0,0)[lt]{\lineheight{1.25}\smash{\begin{tabular}[t]{l}Im$\lambda$\end{tabular}}}}%
    \put(0.73816056,0.29419892){\makebox(0,0)[lt]{\lineheight{1.25}\smash{\begin{tabular}[t]{l}Re$\lambda$\end{tabular}}}}%
    \put(0,0){\includegraphics[width=\unitlength,page=2]{Whitham2.pdf}}%
  \end{picture}%
\endgroup%
 &
\hspace*{-0.5in}\def\svgwidth{3.5in}
\begingroup%
  \makeatletter%
  \providecommand\color[2][]{%
    \errmessage{(Inkscape) Color is used for the text in Inkscape, but the package 'color.sty' is not loaded}%
    \renewcommand\color[2][]{}%
  }%
  \providecommand\transparent[1]{%
    \errmessage{(Inkscape) Transparency is used (non-zero) for the text in Inkscape, but the package 'transparent.sty' is not loaded}%
    \renewcommand\transparent[1]{}%
  }%
  \providecommand\rotatebox[2]{#2}%
  \newcommand*\fsize{\dimexpr\f@size pt\relax}%
  \newcommand*\lineheight[1]{\fontsize{\fsize}{#1\fsize}\selectfont}%
  \ifx\svgwidth\undefined%
    \setlength{\unitlength}{2782.85092812bp}%
    \ifx\svgscale\undefined%
      \relax%
    \else%
      \setlength{\unitlength}{\unitlength * \real{\svgscale}}%
    \fi%
  \else%
    \setlength{\unitlength}{\svgwidth}%
  \fi%
  \global\let\svgwidth\undefined%
  \global\let\svgscale\undefined%
  \makeatother%
  \begin{picture}(1,0.59776367)%
    \lineheight{1}%
    \setlength\tabcolsep{0pt}%
    \put(0,0){\includegraphics[width=\unitlength,page=1]{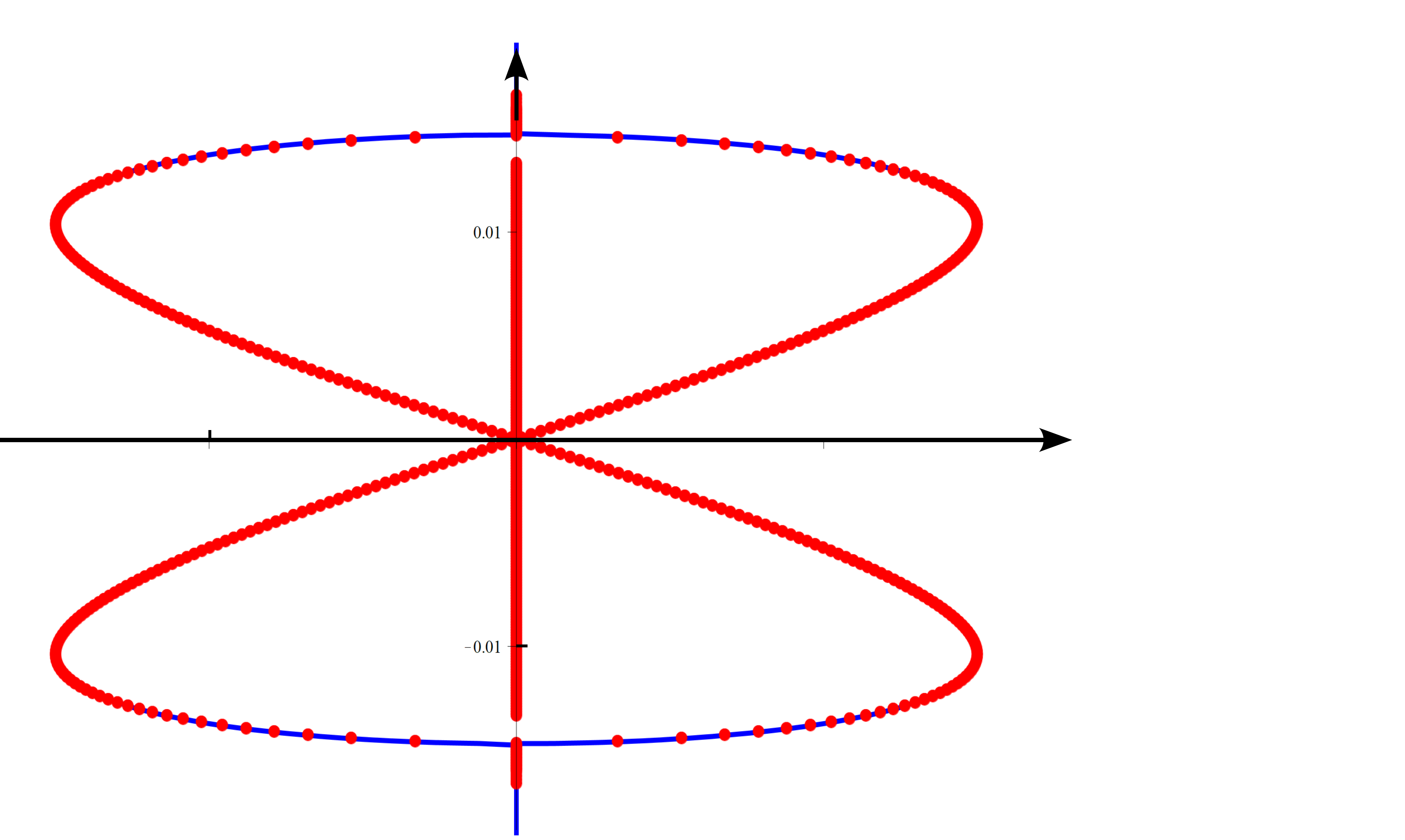}}%
    \put(0.37860245,0.12972361){\makebox(0,0)[lt]{\lineheight{1.25}\smash{\begin{tabular}[t]{l}$-0.01$\end{tabular}}}}%
    \put(0.38084784,0.42637444){\makebox(0,0)[lt]{\lineheight{1.25}\smash{\begin{tabular}[t]{l}$0.01$\end{tabular}}}}%
    \put(0.06720662,0.29621642){\makebox(0,0)[lt]{\lineheight{1.25}\smash{\begin{tabular}[t]{l}$-0.0002$\end{tabular}}}}%
    \put(0.56594966,0.24470091){\makebox(0,0)[lt]{\lineheight{1.25}\smash{\begin{tabular}[t]{l}$0.0002$\end{tabular}}}}%
    \put(0.37956577,0.56500166){\makebox(0,0)[lt]{\lineheight{1.25}\smash{\begin{tabular}[t]{l}Im$\lambda$\end{tabular}}}}%
    \put(0.71804875,0.29610319){\makebox(0,0)[lt]{\lineheight{1.25}\smash{\begin{tabular}[t]{l}Re$\lambda$\end{tabular}}}}%
    \put(0,0){\includegraphics[width=\unitlength,page=2]{Whitham3.pdf}}%
  \end{picture}%
\endgroup%
 \\
\hspace*{-0.9in}$N=2$ & \hspace*{-1.4in}$N=3$ \\[10pt]

\def\svgwidth{3.5in}
\hspace*{0.0in}
\begingroup%
  \makeatletter%
  \providecommand\color[2][]{%
    \errmessage{(Inkscape) Color is used for the text in Inkscape, but the package 'color.sty' is not loaded}%
    \renewcommand\color[2][]{}%
  }%
  \providecommand\transparent[1]{%
    \errmessage{(Inkscape) Transparency is used (non-zero) for the text in Inkscape, but the package 'transparent.sty' is not loaded}%
    \renewcommand\transparent[1]{}%
  }%
  \providecommand\rotatebox[2]{#2}%
  \newcommand*\fsize{\dimexpr\f@size pt\relax}%
  \newcommand*\lineheight[1]{\fontsize{\fsize}{#1\fsize}\selectfont}%
  \ifx\svgwidth\undefined%
    \setlength{\unitlength}{2441.28546767bp}%
    \ifx\svgscale\undefined%
      \relax%
    \else%
      \setlength{\unitlength}{\unitlength * \real{\svgscale}}%
    \fi%
  \else%
    \setlength{\unitlength}{\svgwidth}%
  \fi%
  \global\let\svgwidth\undefined%
  \global\let\svgscale\undefined%
  \makeatother%
  \begin{picture}(1,0.58776834)%
    \lineheight{1}%
    \setlength\tabcolsep{0pt}%
    \put(0,0){\includegraphics[width=\unitlength,page=1]{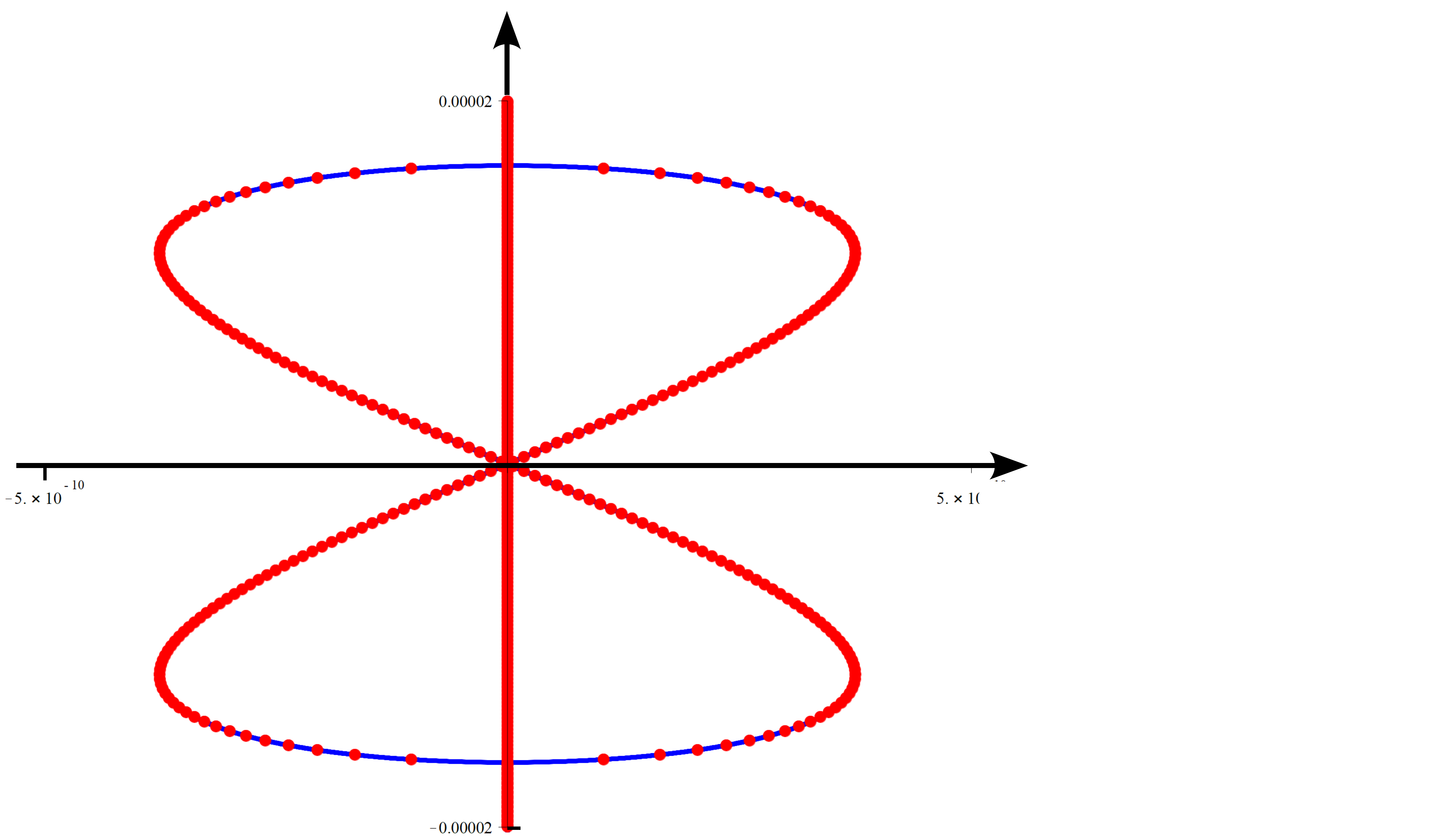}}%
    \put(0.36406763,0.0067669){\makebox(0,0)[lt]{\lineheight{1.25}\smash{\begin{tabular}[t]{l}$-0.00002$\end{tabular}}}}%
    \put(0.37001671,0.50870853){\makebox(0,0)[lt]{\lineheight{1.25}\smash{\begin{tabular}[t]{l}$0.00002$\end{tabular}}}}%
    \put(0.36558611,0.55042252){\makebox(0,0)[lt]{\lineheight{1.25}\smash{\begin{tabular}[t]{l}Im$\lambda$\end{tabular}}}}%
    \put(0.67860041,0.28268237){\makebox(0,0)[lt]{\lineheight{1.25}\smash{\begin{tabular}[t]{l}Re$\lambda$\end{tabular}}}}%
    \put(0,0){\includegraphics[width=\unitlength,page=2]{Whitham4.pdf}}%
    \put(0.6092658,0.22163434){\makebox(0,0)[lt]{\lineheight{1.25}\smash{\begin{tabular}[t]{l}$5 E-10$\end{tabular}}}}%
    \put(0,0){\includegraphics[width=\unitlength,page=3]{Whitham4.pdf}}%
    \put(0.00869429,0.22031568){\makebox(0,0)[lt]{\lineheight{1.25}\smash{\begin{tabular}[t]{l}$-5 E-10$\end{tabular}}}}%
  \end{picture}%
\endgroup%
 &
\def\svgwidth{3.5in}
\hspace*{-0.5in}
\begingroup%
  \makeatletter%
  \providecommand\color[2][]{%
    \errmessage{(Inkscape) Color is used for the text in Inkscape, but the package 'color.sty' is not loaded}%
    \renewcommand\color[2][]{}%
  }%
  \providecommand\transparent[1]{%
    \errmessage{(Inkscape) Transparency is used (non-zero) for the text in Inkscape, but the package 'transparent.sty' is not loaded}%
    \renewcommand\transparent[1]{}%
  }%
  \providecommand\rotatebox[2]{#2}%
  \newcommand*\fsize{\dimexpr\f@size pt\relax}%
  \newcommand*\lineheight[1]{\fontsize{\fsize}{#1\fsize}\selectfont}%
  \ifx\svgwidth\undefined%
    \setlength{\unitlength}{2596.36328509bp}%
    \ifx\svgscale\undefined%
      \relax%
    \else%
      \setlength{\unitlength}{\unitlength * \real{\svgscale}}%
    \fi%
  \else%
    \setlength{\unitlength}{\svgwidth}%
  \fi%
  \global\let\svgwidth\undefined%
  \global\let\svgscale\undefined%
  \makeatother%
  \begin{picture}(1,0.64541074)%
    \lineheight{1}%
    \setlength\tabcolsep{0pt}%
    \put(0,0){\includegraphics[width=\unitlength,page=1]{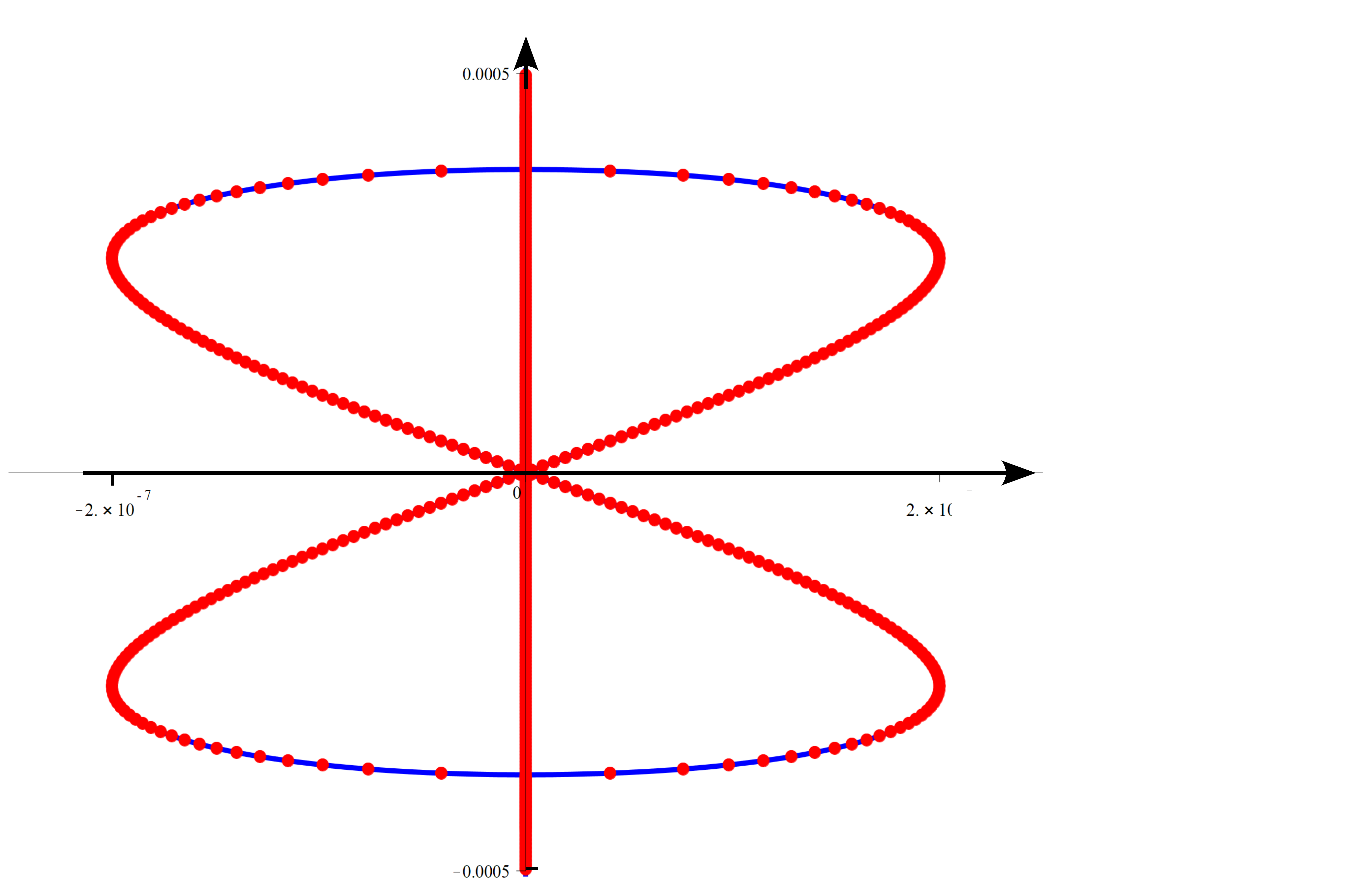}}%
    \put(0.39963423,0.01511679){\makebox(0,0)[lt]{\lineheight{1.25}\smash{\begin{tabular}[t]{l}$-0.0005$\end{tabular}}}}%
    \put(0.39515787,0.5636108){\makebox(0,0)[lt]{\lineheight{1.25}\smash{\begin{tabular}[t]{l}$0.0005$\end{tabular}}}}%
    \put(0.39340873,0.61029553){\makebox(0,0)[lt]{\lineheight{1.25}\smash{\begin{tabular}[t]{l}Im$\lambda$\end{tabular}}}}%
    \put(0.69779716,0.32168022){\makebox(0,0)[lt]{\lineheight{1.25}\smash{\begin{tabular}[t]{l}Re$\lambda$\end{tabular}}}}%
    \put(0,0){\includegraphics[width=\unitlength,page=2]{Whitham5.pdf}}%
    \put(0.05293091,0.2521678){\makebox(0,0)[lt]{\lineheight{1.25}\smash{\begin{tabular}[t]{l}$-2 E-7$\end{tabular}}}}%
    \put(0.64146551,0.25296796){\makebox(0,0)[lt]{\lineheight{1.25}\smash{\begin{tabular}[t]{l}$2 E-7$\end{tabular}}}}%
  \end{picture}%
\endgroup%
 \\
\hspace*{-0.9in}$N=4$ & \hspace*{-1.4in}$N=5$
\end{tabular}

\caption{Comparison of the analytical result (blue) for the unstable spectrum near the origin for 4 different Whitham-like equations with increasing degree of nonlinearity $N$, with $\alpha=1$, $\jmath(k)=\sqrt{\tanh k/k}$, with numerical results (red) using a 9-th order Stokes expansion, with $a=0.02$, $\rho=1.5$. Hill's method \cite{DK} is used with $5$ Fourier modes, $11\times 11$ matrices, $210$ equally spaced Floquet exponents. For $N=2$, $\mu\in [-1/2, 1/2]$, for $N=3$, $\mu\in [-0.06, 0.06]$, for $N=4$, $\mu\in [-0.0001,0.0001]$, for $N=5$, $\mu\in [-0.002, 0.002]$. The relative sizes of these Whitham figure-eight curves is noticeable, see also Fig.~\ref{fig:combo}.}
\label{fig:whithamcomp}
\end{figure}

\section{Small amplitude periodic traveling waves}\label{s:2}

We are interested in the stability of \emph{periodic traveling wave} solutions of \eqref{e:gkdv}. In this section, we discuss such solutions as needed to proceed in what follows. Throughout, we assume that the symbol $\jmath(\cdot)$ associated with the Fourier multiplier $\mathcal{J}$ satisfies hypotheses~(H1)-(H3). 

 A traveling wave solution of \eqref{e:gkdv} is given by $ u(x, t) = U(x - ct) $, where $ c > 0 $ is the wave speed. Substituting this into \eqref{e:gkdv} gives  

\begin{equation}\label{e:qq}
    \mathcal{J} U - c U + \alpha U^N+b = 0,
\end{equation} 

\no where $b$ is an arbitrary constant. We limit ourselves to the case $b=0$, which corresponds to periodic traveling waves bifurcating from the zero solution. 
We seek a solution of the form $ U(x) = \eta(z) $ with $ z = \rho x $, where $ \rho > 0 $ plays the role of a scaling, chosen so that $\eta(z)$ is $ 2\pi $-periodic and satisfies  
\begin{equation}\label{e:q}
    \mathcal{J}_\rho \eta - c \eta + \alpha \eta^N = 0,
\end{equation}

\no where $\J_\rho$ is a Fourier multiplier satisfying

\[
\J_\rho e^{inz} = \jmath(\rho  n)e^{inz},~~~~n\in\mathbb{Z},
\]

\no {i.e.}, if $\J$ is a polynomial in $\partial_x$, $\J_\rho$ has every $\partial_x$ replaced by $\rho \partial_z$. Since $\jmath(k)$ is real valued and even, \eqref{e:q} is invariant under translation $z\mapsto z+z_0$ and reflection $z\mapsto -z$, allowing us to assume that $\eta$ is even. Furthermore, since $\jmath$ is not necessarily a homogeneous function (see, for instance, the case of the Whitham equation), \eqref{e:q} may lack scaling invariance\footnote{The Whitham equation in physical variables is scale invariant, as any scaling affects the physical constants involved. If, as is commonly done, the physical constants depth and acceleration of gravity are equated to 1, the scaling symmetry is broken.}, and we cannot initially assume $\rho=1$.

\newcommand{\vs}{\vspace*{0.1in}}

\vs

{\bf Example.} {\em For the modified KdV equation with $N=3$, $\jmath(k)=1+k^2$, 

\[
u_t+(u-u_{xx}+\alpha u^3)_x=0,
\]

\no the traveling wave solutions bifurcating from the zero solution satisfy 

\[
-c \eta+\eta-\rho^2 \eta''+\alpha \eta^3=0, 
\]

\no where $\eta'$ denotes the derivative of $\eta$ with respect to $z$.

}

\vs

{\bf Example.} {\em For the Whitham equation $u_t+\mathcal{J} u_x+ 2 u u_x=0$ with $N=2$ and 

\[
\widehat{\mathcal{J}f}(k) = \sqrt{\frac{\tanh(k)}{k}}\hat{f}(k), 
\]

\no the traveling wave solutions bifurcating from the zero solution satisfy 

\[
-c \eta+{\mathcal J}_\rho \eta+ \eta^2=0.
\]
}

\vs

First, we establish the smoothness of solutions of \eqref{e:q}. Let $\mathbb{T}$ denote the one-dimensional torus. For $s\ge 0$, 
$H^{s}(\mathbb{T})$ is the standard $L^2$–based Sobolev space on $\mathbb{T}$, where 
$L^2(\mathbb{T})$ is the space of $2\pi$-periodic, square-integrable functions over $\mathbb{R}$. The subspace of even functions in $H^s(\mathbb{T})$ is

\[
H^{s}_{\mathrm{even}}(\mathbb{T}) = \{ f \in H^s(\mathbb{T}) : f(x) = f(-x) \text{ for all } x \in \mathbb{T} \}.
\]

\no We also define

\[
H^{\infty}_{\mathrm{even}}(\mathbb{T}) = \bigcap_{m \ge 0} H^m_{\mathrm{even}}(\mathbb{T}),
\]

\no the space of even functions that are smooth, that is, possess derivatives of all orders in $L^2(\mathbb{T})$.

Finally, $L^\infty(\mathbb{T})$ denotes the space of essentially bounded functions on $\mathbb{T}$, with the norm

\[
\|f\|_{L^\infty} = \operatorname{ess\,sup}_{x \in \mathbb{T}} |f(x)|.
\]

\begin{lemma}[Regularity]\label{l:1}
    If $\eta\in H^s_{\mathrm{even}}(\mathbb{T})$, with $s=\max(1,1+\sigma)$ and $\sigma$ from Hypothesis (H2) satisfies \eqref{e:q} for some $c\in\R$, $\rho>0$, and if $c-N|\alpha|\|\eta^{N-1}\|_{L^\infty}\geq\epsilon>0$ for some $\epsilon$, then $\eta\in H^\infty_{\rm even}(\T)$.
\end{lemma}

\begin{proof}
    Differentiating \eqref{e:q}, we find 
 \begin{equation}
     \eta^\prime=\dfrac{\J_\rho\eta^\prime}{c-N\alpha\eta^{N-1}}. 
 \end{equation}
 
\noindent There exists a positive constant $\nu_{\epsilon}$ such that $c-N|\alpha|\|\eta^{N-1}\|_{L^\infty}\geq\epsilon>0$ whenever $\|\eta\|_\infty<\nu_{\epsilon}$. Note that by Hypothesis (H2), $\eta\in H^\sigma_{\rm even}(\T)$. Using a bootstrap argument, it follows that solutions $\eta\in H^{\sigma}_{\rm even}(\mathbb{T})$
of \eqref{e:q} necessarily satisfy $\eta\in H^\infty_{\rm even}(\mathbb{T})$.  
\end{proof}

With $\sigma$ as in (H2), define the operator $\mathcal{G}:H_{\rm even}^{\sigma}(\mathbb{T}) \times \R^+ \to L^2(\mathbb{T})$ as

\begin{align*}
\mathcal{G}(\eta,c)= \J_\rho \eta-c\eta+\alpha \eta^N. 
\end{align*}

\noindent A standard Sobolev embedding argument shows that $\mathcal{G}$ is well defined.
A non-trivial $2\pi$-periodic solution $\eta$ of \eqref{e:q} in $H_{\rm even}^{\sigma}(\mathbb{T})$ with $c \in \R$ and $\rho>0$ satisfies
\begin{equation}\label{E:F1}
    \mathcal{G}(\eta,c)=0. 
\end{equation}
By virtue of Lemma \ref{l:1}, this provides a non-trivial {\em smooth} $2\pi$-periodic solution of \eqref{E:F1}. 
To study \eqref{E:F1}, note that $\mathcal{G}(0,c)=0$ for $c\in\mathbb{R}$ and $\rho>0$. Also 
\[
\partial_\eta \mathcal{G}(0,c)=\J_\rho-c,
\]
so that for $c\in\mathbb{R}$, $\rho>0$ and $n\in\mathbb{N}$, 
\[
\partial_\eta \mathcal{G}(0,c)\cos(nz) = \left(\jmath(\rho n)-c\right)\cos(nz).
\]
It follows that 

\begin{equation}\label{e:ker}
\ker\left(\partial_\eta \mathcal{G}(0,c_0)\right) = {\rm span}\left(\cos z \right),
\end{equation}

\noindent provided that

\begin{equation}\label{e:c0}
c=c_0=\jmath(\rho),
\end{equation}

\no and that $\rho>0$ is such that
\begin{equation}\label{e:k_cond}
\jmath(\rho n)-\jmath(\rho)\neq0~~~~{\rm for}~~n\in\mathbb{N},~~n\geq 2,
\end{equation}
to ensure the kernel is simple and thus Hypothesis (H3) is satisfied. Assuming $\rho>0$ meets the non-resonance requirement \eqref{e:k_cond}, a one-parameter family of non-trivial, even solutions $(\eta(a)(\cdot),c(a))$ of $\mathcal{G}(\eta,c)$ is obtained, bifurcating
from $\eta\equiv 0$ and $c=c_0$ and defined for $|a|\ll 1$. This existence proof is straightforward and follows those in \cite{EK_local, Johnson2013StabilityEquations} and is therefore not included here. A key characteristic of the solutions $\eta(a)(\cdot)$ and $c(a))$ is their analytic dependence on the parameter $a$ for $|a|\ll 1$. Using this, the following result further establishes analytic expansions for these solutions that are valid for $|a|\ll 1$.

\begin{theorem}[Existence]\label{t:1}
    Assume the symbol $\jmath(\cdot)$ in equation \eqref{e:M} meets hypotheses (H1)-(H2). For any fixed scaling $\rho>0$ satisfying hypothesis (H3) there exists a one-parameter family of solutions of \eqref{e:q}
 given by $u(x,t)=\eta(\rho(x-c t))$, where $c$ and $\eta$ depend on $a$ and $\rho$, for sufficiently small $|a|$, $a \in \R$; $\eta(a)(\cdot)$ is $2\pi$-periodic, 
 even and smooth in its argument, with $c(a)$ even in $a$, and $\eta(a)(\cdot)$ and $c(a)$ depend analytically on $a$ and $\rho$. Moreover, as $a\to 0$, 
\begin{align}\label{E:w_ansatz}
    \eta(a)(z)=a\cos z  + a^N \eta_N + \mathcal{O}(a^{N+1}),
\end{align}
and
\begin{align}\label{e:c}
    c(a)=c_0+a^\tau c_\tau+\mathcal{O}(a^{\tau+1}),~~c_0=\jmath(\rho). 
\end{align}

\noindent For even-power nonlinearity in \eqref{e:gkdv}, $N$ even, we have\footnote{The $L^2(\mathbb{T})$-inner product for $f,g \in L^2(\mathbb{T})$ is 
$\langle f, g \rangle := \frac{1}{\pi} \int_0^{2\pi} f(z) \overline{g(z)}\,dz
= \sum_{n \in \mathbb{Z}} \widehat{f}_n \overline{\widehat{g}_n}$, 
where $\widehat{f}_n$ denotes the $n$-th Fourier coefficient of $f$.} $\tau=2N-2$, $c_\tau= N \langle \eta_N \cos^{N}{z} ,1\rangle$, and 

\begin{equation}\label{e:expand1even} 
\eta_N= (\jmath(\rho)-\J_\rho)^{-1}\cos^N z.
\end{equation}

\noindent For odd-power nonlinearity in \eqref{e:gkdv}, $N$ odd, we have $\tau=N-1$, $c_\tau=\alpha \langle \cos^{N} z,\cos z\rangle$, and 

\begin{equation}\label{e:expand1odd} 
\eta_N=\alpha (\jmath(\rho)-\J_\rho)^{-1}(\cos^N z-\langle \cos^N z,\cos z \rangle \cos z).
\end{equation}

  
\end{theorem}
\begin{proof}
For any constant $\rho>0$, the existence of the solutions $(\eta(a)(\cdot),c(a))$ follows from a simple Lyapunov-Schmidt argument; for related arguments, see \cite{EK_local}, for instance. Consequently, the solutions $(\eta(a)(\cdot),c(a))$ are analytic in $a$ for $|a|\ll 1$. Thus they can be expanded as follows for a fixed $\rho>0$:
\begin{equation}\label{e:expand1} 
\left\{\begin{aligned}
\eta(a)(z)&=a\cos(z)+a^2\eta_2(\rho)(z)+\dots+a^N \eta_N+\mathcal{O}(a^{N+1}),\\
c(a)&=c_0+c_2(\rho)a^2+ \dots+c_\tau(\rho) a^\tau+\mathcal{O}(a^{\tau+1}).
\end{aligned}\right.
\end{equation}
The specifics of the Lyapunov-Schmidt argument lead to the relation $c(a)=c(-a)$.
For further analysis, we define the operator
\begin{equation}
    \mathcal{D}_\rho=\jmath(\rho)\mathcal{I}-\J_\rho,
\end{equation}

\no where $\mathcal{I}$ is the identity operator. The operator 
$\mathcal{D}_\rho$ is a Fourier multiplier that operates through

\begin{equation}\label{doperation}
    \mathcal{D}_\rho e^{inz}=(\jmath(\rho)-\jmath(\rho n))e^{inz},\quad\text{for}\quad n\in\Z\setminus\{1\}.
\end{equation}

\no The operator $\mathcal{D}_\rho$ is self adjoint (inherited from the operator $\J_\rho$) and invertible (from hypothesis (H3)). A hierarchy of compatibility criteria indexed by the order of the small parameter $a$ is obtained by substituting the expansions \eqref{e:expand1} into the profile equation \eqref{e:q}:

\begin{align}\label{e:solve2}
\partial_\eta \mathcal{G}(0,c;\rho)\eta= \alpha \sum_{m=0}^{N} \sum_{n=0}^{N-m} \binom{N}{m} \binom{N-m}{n} \frac{a^{N + m(N-1)} \eta_N^m}{2^{N-m}} \cos{(N-m-2n)z} .
\end{align}

\no The $\mathcal{O}(a^0)$ equation is trivially satisfied while the $\mathcal{O}(a^1)$ equation holds using \eqref{e:c0}. For any $j\leq N-1$, it is straightforward to see that $\eta_j\in\ker{(\mathcal{D}_\rho)}$ from which it follows that $c_{j-1}=0$. The Fredholm alternative implies that the right-hand side of \eqref{e:solve2} must be orthogonal to the kernel of the symmetric operator $\partial_\eta\mathcal{G}(0,c;\rho)$,  leading to the stated equation for $c_\tau=c_\tau(\rho)$. To satisfy this condition, it is essential to determine the lowest power of $a$ for which the right-hand side of \eqref{e:solve2} is proportional to $\cos(z)$. This occurs when

\[
N-m-2n=\pm1.
\] 

\no Consider $N-m-2n=1$. Then $N-1=m+2 n$.
The smallest values of $m$ and $n$ need to be checked. For $N$ odd and $(m,n)=(0,(N-1)/2)$, the smallest power of $a$ ensuring that the right-hand side of the expression \eqref{e:solve2} is proportional to $\cos(z)$ is $a^N$. For $N$ even and $(m,n)=(1,(N-2)/2)$, it is $a^{2N-1}$.
Upon comparing coefficients in \eqref{e:solve2}, the first non-trivial equation emerges at $\mathcal{O}(a^N)$, expressed as
\begin{equation}\label{e:solve1}
\mathcal{D}_\rho\eta_N= \alpha  \cos^N{z}-c_{N-1}\cos{z}.
\end{equation}

\no Taking the inner product with $\cos{z}$, we get

\begin{equation}\label{e:c3}
\langle\mathcal{D}_\rho\eta_N,\cos{z}\rangle=\alpha \langle\cos^N{z},\cos{z} \rangle-c_{N-1}\langle\cos{z},\cos{z} \rangle.
\end{equation}

\no Using the self adjointness of the operator $\mathcal{D}_\rho$, 
\begin{equation}
\langle\mathcal{D}_\rho\eta_N,\cos{z}\rangle=\langle\eta_N,\mathcal{D}_\rho\cos{z}\rangle=\langle\eta_N,0 \rangle=0, 
\end{equation}

\no since $\cos{z}\in\ker{\mathcal{D}_\rho}$, see \eqref{e:ker}. 
Consequently, \eqref{e:c3} yields 
 \begin{equation}\label{e:c}
c_{N-1}=\alpha \langle\cos^N{z},\cos{z} \rangle.
\end{equation}
 
\no Using the invertibility of the operator $\mathcal{D}_\rho$, and solving \eqref{e:solve1} for $\eta_N$ gives
\begin{equation}\label{e:eta}
\eta_N=\alpha \mathcal{D}_\rho^{-1} ( \cos^N{z}-c_{N-1}\cos{z}). 
\end{equation}

\no Proceeding similarly for even $N$, it follows that $c_{N-1}=0$. For this case, the first non-trivial equation emerges at $\mathcal{O}(a^{2N-1})$ and reads
\[
\mathcal{D}_\rho\eta_{2N-1}=\alpha N(\cos^{N-1}{z})\eta_N-c_{2N-2}\cos{z},
\] 

\no which leads to
\begin{equation}\label{e:c4}
c_{2N-2}=\alpha \langle N(\cos^{N-1}{z})\eta_N,\cos{z}\rangle. 
\end{equation}

For even $N$, $\alpha$ is chosen to be $1$, without loss of generality. This concludes the proof. \end{proof}

For a given power of nonlinearity $N$, we can calculate the leading-order corrections $\eta_N$ and $c_\tau$ from \eqref{e:c}, \eqref{e:eta} and \eqref{e:c4}. For $N=2, \ldots, 7$, these results are given in Tables \ref{t:2} (odd $N$) and \ref{t:1} (even $N$).


\begin{table}[H]
  \centering
{\renewcommand{\arraystretch}{2}
\begin{tabular}{|c|c|c|}
  \hline
  $N$ & $\eta_N$& $c_{N-1}$ \\
  \hline
  $3$ &$\eta_3=\dfrac{\alpha}{4(\jmath(\rho)-\jmath(3\rho))}\cos{(3z)}$ &$c_2=\dfrac{3\alpha}{4}$  \\[5pt]
  \hline
  $5$ & $\eta_5=\dfrac{5\alpha}{16(\jmath(\rho)-\jmath(3\rho))}\cos{(3z)}+\dfrac{\alpha}{16(\jmath(\rho)-\jmath(5\rho))}\cos{(5z)}$ & $c_4=\dfrac{5\alpha}{8}$ \\[5pt]
  \hline
  $7$&$\eta_7=\dfrac{21\alpha}{64(\jmath(\rho)-\jmath(3\rho))}\cos{(3z)}+\dfrac{7\alpha}{64(\jmath(\rho)-\jmath(5\rho))}\cos{(5z)}+$&$c_{6}=\dfrac{35\alpha}{64}$ \\ [5pt]
  &$\dfrac{\alpha}{64(\jmath(\rho)-\jmath(7\rho))}\cos{(7z)}$&\\[5pt]
  \hline
\end{tabular}
}
\caption{Expressions of $\eta_N$ and $c_{N-1}$ for odd values of $N$}\label{t:2}
\end{table}
 
 \begin{table}[H]
  \centering
  {\renewcommand{\arraystretch}{2}
\begin{tabular}{|c|c|c|}
  \hline
  $N$ & $\eta_N$& $c_{2N-2}$ \\
  \hline
  $2$ &$\eta_2=\dfrac{1}{2(\jmath(\rho)-1)}+\dfrac{1}{2(\jmath(\rho)-\jmath(2\rho))}\cos{(2z)}$ &$c_2=\dfrac{1}{(\jmath(\rho)-1)}+\dfrac{1}{2(\jmath(\rho)-\jmath(2\rho))} $  \\[5pt]
  \hline
  $4$ & $\eta_4=\dfrac{3}{8(\jmath(\rho)-1)}+\dfrac{1}{2(\jmath(\rho)-\jmath(2\rho))}\cos{(2z)}+$ & $c_6=\dfrac{9}{8(\jmath(\rho)-1)}+\dfrac{1}{(\jmath(\rho)-\jmath(2\rho))}+$ \\[5pt]
  &$\dfrac{1}{8(\jmath(\rho)-\jmath(4\rho))}\cos{(4z)}$&$\dfrac{1}{16(\jmath(\rho)-\jmath(4\rho))}$\\
  \hline
  $6$&$\eta_6=\dfrac{5}{16(\jmath(\rho)-1)}+\dfrac{15}{32(\jmath(\rho)-\jmath(2\rho))}\cos{(2z)}+$&$c_{10}=\dfrac{3}{16}\left(\dfrac{25}{4(\jmath(\rho)-1)}+\dfrac{225}{32(\jmath(\rho)-\jmath(2\rho))}+\right.$ \\[5pt] 
  &$\dfrac{3}{16(\jmath(\rho)-\jmath(4\rho))}\cos{(4z)}+\dfrac{1}{32(\jmath(\rho)-\jmath(6\rho))}\cos{(6z)}$&$\left.\dfrac{9}{8(\jmath(\rho)-\jmath(4\rho))}+\dfrac{1}{32(\jmath(\rho)-\jmath(6\rho))}\right)$\\[5pt]
  \hline
\end{tabular}}
\caption{Expressions of $\eta_N$ and $c_{2N-2}$ for even values of $N$, with $\alpha=1$, without loss of generality}\label{t:1}
\end{table}

\vs 

{\bf Example.} {\em For the modified KdV equation, $N=3$, $\jmath(k)=1+k^2$, the Stokes waves to leading order are

\[
u= a \cos z -\frac{\alpha a^3}{32\rho^2}\cos 3z+\mathcal{O}(a^4), \quad \quad c=1+\rho^2+\frac{3\alpha a^2}{4}+\mathcal{O}(a^4). 
\]
}

\vs 

{\bf Example.} {\em For the Whitham equation, the Stokes waves to leading order are

\begin{align*}
u&=a\cos z+a^2\left(\frac{1}{2(\jmath(\rho)-1}+\frac{1}{2(\jmath(\rho)-\jmath(2\rho))}\cos 2z\right)+\mathcal(O)(a^3),\\
c&=\jmath(\rho)+a^2 \left(\frac{1}{\jmath(\rho)-1}+\frac{1}{2(\jmath(\rho)-\jmath(2\rho))}\right)+\mathcal(O)(a^4).
\end{align*}

\no where $\jmath(k)=\sqrt{\tanh(k)/k}$. }

\section{The complete Benjamin-Feir spectrum}
Throughout this section, $\eta=\eta(a;\rho)(z)$ denotes a small-amplitude $2\pi$ - periodic traveling wave solution of \eqref{e:gkdv} with wave speed $c=c(a;\rho)$, with $\rho>0$ and $|a|\ll1$. The existence of such a solution is guaranteed by Theorem \ref{t:1}. 

Linearizing \eqref{e:gkdv} around $\eta$ in the spatial frame of reference $z=\rho(x-ct)$, and considering perturbations to $\eta$ of the form $\eta+\epsilon\Tilde{v}+\mathcal{O}(\epsilon^2)$ for $0<|\epsilon|\ll 1$, we obtain the linear evolution equation 

\begin{equation}
    \Tilde{v}_t=\rho\partial_z(c-\mathcal{J}_\rho-N\alpha \eta^{N-1})\Tilde{v}.
\end{equation}

\no We seek a solution of the form $\Tilde{v}(z,t)=\psi(z)\exp(\lambda \rho t)$, $\lambda\in \mathbb{C}$ and $\psi \in L^2(\R)$, the space of square-integrable functions on $\R$, to arrive at the equation

\begin{equation}
\label{e:l}\lambda \psi=\mathcal{Q}_a\psi,
\end{equation}

\no where

\begin{equation}\label{e:op}
    \mathcal{Q}_a=\partial_z \circ (c-\mathcal{J}_\rho-N\alpha \eta^{N-1}).
\end{equation}

\vs

{\bf Example.} {\em For the focusing modified KdV equation with $\alpha=-1$, $N=3$, $\jmath(k)=1+k^2$, the linear operator \eqref{e:op} is 

\[
\mathcal{Q}_a=\partial_z \circ (c-1+3\eta^2+\rho^2 \partial_z^2).
\]
}

\vs

{\bf Example.} {\em For the Whitham equation, the linear operator \eqref{e:op} is

\[
\mathcal{Q}_a=\partial_z \circ (c-{\mathcal J}_\rho-2\eta).
\]
}

\vs

Since the coefficients of the operator $\mathcal{Q}_a$ are $2\pi$-periodic, we can use Floquet theory and write  all solutions of \eqref{e:l} in $L^2(\R)$  as superpositions of solutions of the form $\psi(z)=e^{i\mu z}\Tilde{\psi}(z)$, with Floquet exponent $\mu\in\left(-1/2,1/2\right]$, where $\Tilde{\psi}$ is a $2\pi$-periodic function. We have

\begin{equation*}
    \lambda\Tilde{\psi}=e^{-i\mu z}\mathcal{Q}_ae^{i\mu z}\Tilde{\psi}=\mathcal{Q}_{a,\mu}\Tilde{\psi},
\end{equation*}

\no where $\mathcal{Q}_{a,\mu}: H^s(\T)\to L^2(\T)$ is given by

\begin{equation}\label{e:op1}
    \mathcal{Q}_{a,\mu}=(\partial_z+i\mu)\circ (c-e^{-i\mu z}\mathcal{J}_\rho e^{i\mu z}-N\alpha\eta^{N-1}),
\end{equation}

\no with $s=\max(1,1+\sigma)$ and $\sigma$ as in Hypothesis (H2). Consequently,

\begin{equation}\label{e:spe}
\operatorname{spec}_{L^2(\R)}(\mathcal{Q}_a)= \bigcup_{\mu\in(-1/2,1/2]}\operatorname{spec}(\mathcal{Q}(a,\mu)). 
\end{equation}

\no For each $\mu\in\left(-1/2,1/2\right]$, the $L^2(\T)$ spectrum of $\mathcal{Q}_{a,\mu}$ consists of discrete eigenvalues. Thus, the essential $L^2(\R)$ spectrum of $\mathcal{Q}_{a}$ can be described by the one-parameter family of point spectra of the associated Bloch operators $\mathcal{Q}_{a,\mu}$.
We proceed to examine the $L^2(\T)$ spectrum of $\mathcal{Q}_{a,\mu}$ as a function of the Floquet exponent $\mu\in\left(-1/2,1/2\right]$, for $|a|$ sufficiently small. Lemma~\ref{l:2} allows us to consider $\mu\in[0,1/2]$ only. 

We begin by considering the case $a=0$, which corresponds to the trivial solution $\eta=0$. Since $\operatorname{spec}(\mathcal Q_{0,\mu})$ is associated with a differential operator having periodic coefficients, its spectrum can be explicitly determined through straightforward Fourier analysis:

\begin{align}\label{E:spec}
    \mathcal Q_{0,\mu}e^{inz} = i\Omega_{n,\mu}e^{inz},\quad \text{for}\quad n \in \mathbb{Z},
\end{align}

\no where 

\begin{align}\label{E:oomega}
    \Omega_{n,\mu} =(n+\mu)(\jmath(\rho)-\jmath(\rho(n+\mu)). 
\end{align}

\no It follows that the $L^2(\T)$-spectrum of $\mathcal Q_{0,\mu}$ is 

\begin{equation}\label{e:spec}
    \operatorname{spec}_{L^2(\mathbb{T})}(\mathcal Q_{0,\mu})=\{i\Omega_{n,\mu}, n \in \Z\},
\end{equation}

\no which implies that $\operatorname{spec}_{L^2(\mathbb{T})}(\mathcal Q_{0,\mu})$ consists of purely imaginary eigenvalues with finite algebraic multiplicity, as expected, since $a=0$ represents the trivially-stable zero solution.
Moreover, Lemma \ref{l:2} implies that $\operatorname{spec}_{L^2(\mathbb{T})}(\mathcal Q_{a,\mu})$ is symmetric about the imaginary axis. Therefore, as $\mu$ and $a$ vary, eigenvalues of $\mathcal{Q}_{a,\mu}$ may leave the imaginary axis only through collisions with other purely imaginary eigenvalues, leading to instability. Therefore, we analyze the non-simple eigenvalues of the operator. In particular, zero is an eigenvalue of $\mathcal{Q}_{0,0}$ with algebraic multiplicity three. 

Furthermore, $\operatorname{spec}(\mathcal Q_{0,0})$
decomposes in two separated parts,

\begin{equation*}
           \operatorname{spec}(\mathcal Q_{0,0})= \operatorname{spec}_0(\mathcal Q_{0,0})\cup\operatorname{spec}_1(\mathcal Q_{0,0}),
\end{equation*}

\no where $\operatorname{spec}_0(\mathcal Q_{0,0})=\{0\}$ and

\begin{equation*}
\operatorname{spec}_1(\mathcal Q_{0,0})\subset\{\lambda\in i\R:|\lambda|\geq |\jmath(\rho)-\jmath(2\rho)|\}.
\end{equation*}

\no Next, we examine the stability of solutions of small amplitude $a\neq 0$. 

\begin{lemma}\label{l:31}
    For $a$ and $\mu$ sufficiently small, the following properties hold. 
    
\begin{enumerate}[(a)]

\item The point spectrum of $\mathcal{Q}_{a,\mu}$ decomposes as 

\begin{equation*}
           \operatorname{spec}(\mathcal Q_{a,\mu})= \operatorname{spec}_0(\mathcal Q_{a,\mu})\cup\operatorname{spec}_1(\mathcal Q_{a,\mu}),
\end{equation*}

\no with

\begin{equation*}
           \operatorname{spec}_0(\mathcal Q_{a,\mu})\subset B(0;R/3),\quad \operatorname{spec}_1(\mathcal Q_{a,\mu})\subset \C\setminus \overline{B(0;R/2)},
\end{equation*}

\no where $R=|\jmath(\rho)-\jmath(2\rho)|$. In other words, for fixed $a$ and $\mu$, there is a separation between the point spectrum near the origin and the rest of the spectrum. The size of this separation is independent of $a$ or $\mu$.



        
\item The spectral projection $\Pi_a(\mu)$ associated with $\operatorname{spec}_0(\mathcal Q_{a,\mu})$, 

\begin{equation}\label{e:pr}
\Pi_{a,\mu}=\dfrac{1}{2\pi i}\oint_{\partial B(0;R/3)}(\lambda-\mathcal{Q}_{a,\mu})^{-1}d\lambda, 
\end{equation}

\no satisfies $\|\Pi_{a,\mu}-\Pi_{0,0}\|_{L^2(\T)\to L^2(\T)}=\mathcal{O}(|\mu|+|a|)$. The operators $\Pi_{a,\mu}$ are well-defined projectors commuting with $\mathcal{Q}_{a,\mu}$:

\begin{equation}\label{e:pp}
\Pi^2_{a,\mu}=\Pi_{a,\mu},\quad\Pi_{a,\mu}\mathcal{Q}_{a,\mu}=\mathcal{Q}_{a,\mu}\Pi_{a,\mu}.
\end{equation}   
        
\item The projectors $\Pi_{a,\mu}$ are similar to each other: 

\beq\label{e:u2}
\mathcal{U}_{a,\mu}\Pi_{0,0}\mathcal{U}^{-1}_{a,\mu}=\Pi_{a,\mu},\quad \mathcal{U}^{-1}_{a,\mu}\Pi_{a,\mu}\mathcal{U}_{a,\mu}=\Pi_{0,0},
\eeq

\no with transformation operators

\beq\label{e:u}
\mathcal{U}_{a,\mu}=(\mathcal{I}-(\Pi_{a,\mu}-\Pi_{0,0})^2)^{-1/2}[\Pi_{a,\mu}\Pi_{0,0}+(\mathcal{I}-\Pi_{a,\mu})(\mathcal{I}-\Pi_{0,0})],
\eeq 

\no which are bounded and invertible in $H^s(\T)$ and in $L^2(\T)$. Their inverse is
        
\begin{equation}\label{e:u1}
\mathcal{U}^{-1}_{a,\mu}=[\Pi_{0,0}\Pi_{a,\mu}+(\mathcal{I}-\Pi_{0,0})(\mathcal{I}-\Pi_{a,\mu})](\mathcal{I}-(\Pi_{a,\mu}-\Pi_{0,0})^2)^{-1/2}.
\end{equation}

\item The subspaces $\mathscr{F}_{a,\mu}=\rm{Rg}(\Pi_{a,\mu})$ are isomorphic to each other: $\mathscr{F}_{a,\mu}=\mathcal{U}_{a,\mu}\mathscr{F}_{0,0}$. In particular, $\dim{\mathscr{F}_{a,\mu}}=\dim{\mathscr{F}_{0,0}}=3$.

\item $\operatorname{spec}_1(\mathcal Q_{a,\mu})\subset i \R$

\end{enumerate}
\end{lemma}
\begin{proof}
The proof is provided in Lemma~\ref{l:31a} of Appendix~\ref{a:2}.
\end{proof}

\sloppypar Lemma~\ref{l:31} shows that the operator $\mathcal{U}_{a,\mu}$ is an isomorphism between $\mathscr{F}_{0,0}$ and $\mathscr{F}_{a,\mu}$. Lemma~\ref{l:b2} entails that this isomorphism maintains the symplectic and reversible properties of the basis.

Consider the decomposition of the spectrum of $\mathcal Q_{a,\mu}$ in Lemma \ref{l:31}. The eigenvalues in
$\operatorname{spec}_0(\mathcal Q_{a,\mu})$ are the eigenvalues of 
the restriction of $\mathcal{Q}_{a,\mu}$ to the three-dimensional subspace $\mathscr F_{a,\mu}$. We determine the location of these
eigenvalues by successively computing 
a basis of  $\mathscr F_{a,\mu}$, 
the $3\times3$ matrix representing the action of $\mathcal Q_{a,\mu}$
on this basis, and the eigenvalues of this matrix. Note that for $a=0$, $\mathscr{F}_{0,\mu}$ is spanned by $\{\cos z, \sin z,1/\sqrt{2}\}$. Let $\Sigma(a, \mu)=\{\psi_1^+(a, \mu),\psi_1^-(a, \mu),\psi_0(a, \mu)\}$ be the reversible and symplectic basis of $\mathscr{F}_{a,\mu}$. Thus $\Sigma(0,\mu)=\{\psi_1^+(0, \mu),\psi_1^-(0, \mu),\psi_0(0, \mu)\}=\{\cos z, \sin z,1/\sqrt{2}\}$. We use the transformation operators $\mathcal{U}_{a,\mu}$ derived in Lemma \ref{l:31} to construct the symplectic and reversible basis for $\mathscr{F}_{a,\mu}$ by

\begin{equation}\label{e:teig}
    \psi_i^{\sigma}(a,\mu)=\mathcal{U}_{a,\mu}\psi_i^{\sigma}(0,\mu).
\end{equation}



\begin{lemma}[Expansion of the basis $\Sigma$]\label{l:b1} For $a$ and $\mu$ sufficiently small, the expansion for the symplectic and reversible basis for $\mathscr{F}_{a,\mu}$ is

\begin{equation}\label{e:bf1}
    \begin{aligned}
        \psi_1^+(a,\mu)=&\cos{z}+\alpha N a^{N-1}\mathcal{D}_\rho^{-1}\left(\cos^N z-\psi^{(1)}_{N}\right)+\mathcal{O}(a^N,a^{N-1}\mu),\\
        \psi_1^-(a,\mu)=&\sin{z}+\alpha N a^{N-1} \mathcal{D}_\rho^{-1}\left(\cos^{N-1} z\sin{z}-\psi^{(2)}_{N}\right)+\mathcal{O}(a^N,a^{N-1}\mu),\\
        \psi_0(a,\mu)=&\dfrac{1}{\sqrt{2}}+\dfrac{1}{\sqrt{2}}\alpha N a^{N-1}\mathcal{D}_\rho^{-1}\left(\cos^{N-1} z-\psi^{(3)}_{N}\right)+\mathcal{O}(a^N,a^{N-1}\mu),
        \end{aligned}
\end{equation}

\no where

\begin{equation}\label{e:expand31} \psi^{(1)}_N=
\left\{\begin{aligned}
\frac{1}{2}\left\langle\cos^N{z},1\right\rangle,\quad &N~~\text{even},\\
\cos{z}\left\langle\cos^N{z},\cos{z}\right\rangle,\quad &N~~\text{odd},
\end{aligned}\right.
\end{equation}

\begin{equation}\label{e:expand32} \psi^{(2)}_N=
\left\{\begin{aligned}
0,\quad &N~~\text{even},\\
\sin{z}\left\langle\cos^{N-1}{z}\sin{z},\sin{z}\right\rangle,\quad &N~~\text{odd},
\end{aligned}\right.
\end{equation}

\begin{equation}\label{e:expand4} \psi^{(3)}_N=
\left\{\begin{aligned}
\cos{z}\left\langle\cos^{N-1}{z},\cos{z}\right\rangle,\quad &N~~\text{even},\\
\frac{1}{2}\left\langle\cos^{N-1}{z},1\right\rangle,\quad &N~~\text{odd}.
\end{aligned}\right.
\end{equation}


\end{lemma}

\begin{proof}
    The proof is provided in Appendix \ref{a:3}.
\end{proof}

\vs

{\bf Example.} {\em For the focusing modified KdV equation with $\alpha=-1$, $N=3$, $\jmath(k)=1+k^2$, the expansion for the symplectic and reversible basis is
\begin{align*}
    \psi_1^+(a,\mu)&=\cos z +\dfrac{3a^2}{32\rho^2}\cos{3z}+\mathcal{O}(a^3,a^2\mu),\\
    \psi_1^-(a,\mu)&=\sin z +\dfrac{3a^2}{32\rho^2}\sin{3z}+\mathcal{O}(a^3,a^2\mu),\\
    \psi_0(a,\mu)&=\dfrac{1}{\sqrt{2}}+\dfrac{a^2}{2\sqrt{2}\rho^2}\cos{2z}+\mathcal{O}(a^3,a^2\mu).
\end{align*}
}
\vs

{\bf Example.} {\em For the Whitham equation, the expansion for the symplectic and reversible basis is
\begin{align*}
    \psi_1^+(a,\mu)&=\cos z +\dfrac{a}{\jmath(\rho)-\jmath(2\rho)}\cos{2z}+\mathcal{O}(a^2,a\mu),\\
    \psi_1^-(a,\mu)&=\sin z +\dfrac{a}{\jmath(\rho)-\jmath(2\rho)}\sin{2z}+\mathcal{O}(a^2,a\mu),\\
    \psi_0(a,\mu)&=\dfrac{1}{\sqrt{2}}+\mathcal{O}(a^2,a\mu), 
\end{align*}

\no where $\jmath(k)=\sqrt{\tanh(k)/k}$. In fact, not specifying $\jmath(k)$, these expressions give the sympletic and reversible basis for any equation with $N=2$. 

}

\vs

First, we introduce the notions of complex Hamiltonian and complex reversible operators.

\begin{definition}\label{d:1} 
An operator $\mathcal{Q}$ is said to be

\begin{enumerate}[(i)]

\item \textbf{Complex Hamiltonian} if $\mathcal{Q}=\mathcal{A}\mathcal{T}$, where $\mathcal{A}$ is skew adjoint and $\mathcal{T}$ is self adjoint, that is,
\[
\mathcal{A}^\ast = -\mathcal{A} 
\quad \text{and} \quad 
\mathcal{T}^\ast = \mathcal{T}.
\]
In this manuscript, we take $\mathcal{A}=\partial_x$.

\item \textbf{Complex reversible} if
\[
\mathcal{Q} \circ \overline{\mathcal{P}} 
= - \overline{\mathcal{P}} \circ \mathcal{Q},
\]
where $\overline{\mathcal{P}}$ is the complex involution defined by

\begin{equation}\label{e:in}
[\overline{\mathcal{P}}\eta](x) = \overline{\eta}(-x).
\end{equation}

\end{enumerate}
\end{definition}

The operator $\mathcal{Q}_{a,\mu}$ assumes the form of a complex Hamiltonian and reversible operator,

 \begin{equation}\label{e1:h}
\mathcal{Q}_{a,\mu}=\mathcal{A}_\mu\mathcal{T}_{a,\mu},
 \end{equation}
 
\no with $\mathcal{T}_{a,\mu}$ self adjoint and $\mathcal{A}_\mu$ skew adjoint. Here,
\[
\mathcal{A}_\mu = \partial_x + i\mu .
\]

\no For $\mu \neq 0$ the operator $\mathcal{A}_\mu$ is invertible on the space of periodic functions, that is, on $ L^2(\mathbb{T}) $, and we write

\[
\mathcal{X}_\mu := \mathcal{A}_\mu^{-1}: L^2(\T) \to H^1(\T),
\]

\no which acts as a Fourier multiplier with symbol $1/i(\mu+k)$.
Since $\mathcal{A}_0=\partial_x$ is not invertible on $L^2(\T)$, we restrict $\mathcal{X}_0$ to the mean‑zero subspace
\[
L_0^2(\T)=\Bigl\{f\in L^2(\T)\,:\,\int_\T f(z)\,dz=0\Bigr\}.
\]

We now define $\mu$-symplectic and reversible bases. 

\begin{definition}\label{d:2} 
A linearly independent set $\{\Omega_1^+,\Omega_1^-,\Omega_0\}$ is called
\begin{enumerate}[(i)]
  
      \item 
      \begin{enumerate}
      
      \item {\bf $\mu$-{symplectic}, $\mu\neq0$}, if 
      
      \[
      \langle\mathcal{X}_\mu\Omega_1^+,\Omega_1^+\rangle=\langle\mathcal{X}_\mu\Omega_1^-,\Omega_1^-\rangle=\dfrac{i\mu}{1-\mu^2},\quad \langle\mathcal{X}_\mu\Omega_0,\Omega_0\rangle=-\dfrac{i}{\mu},
      \]
      
      \[
      \langle\mathcal{X}_\mu\Omega_1^+,\Omega_1^-\rangle=\dfrac{1}{1-\mu^2}, \quad\langle\mathcal{X}_\mu\Omega_0,\Omega_1^+\rangle=\langle\mathcal{X}_\mu\Omega_0,\Omega_1^-\rangle=0, 
      \]
  
      \item {\bf $0$-{symplectic}} if both $\Omega_1^+,\Omega_1^-\in L^2_0(\T)$ and 
      
      \[
      \langle\mathcal{X}_0\Omega_1^+,\Omega_1^-\rangle=1,\quad \langle\mathcal{X}_0\Omega_1^+,\Omega_1^+\rangle=0=\langle\mathcal{X}_0\Omega_1^-,\Omega_1^-\rangle;
      \]
      
      \end{enumerate}
      
      \item {\bf reversible} if 
      
      \[
      \overline{\mathcal{P}}\Omega_1^+=\Omega_1^+, \quad\overline{\mathcal{P}}\Omega_1^-=-\Omega_1^-,\quad \overline{\mathcal{P}}\Omega_0=\Omega_0.
      \]
  
\end{enumerate}

\end{definition}

Alternatively, \eqref{e1:h} can be written as
 
\begin{equation}\label{e:h1}
\mathcal{X}_\mu\mathcal{Q}_{a,\mu}=\mathcal{T}_{a,\mu}.
\end{equation}

\no Let $\mathcal{Q}_{ a,\mu}\psi=\gamma_1^+\psi_1^+(a,\mu)+\gamma_1^-\psi_1^-(a,\mu)+\gamma_0\psi_0(a,\mu)$, which implies $\mathcal{T}_{a,\mu}\psi=\mathcal{X}_\mu(\gamma_1^+\psi_1^+(a,\mu)+\gamma_1^-\psi_1^-(a,\mu)+\gamma_0\psi_0(a,\mu))$. For $\mu\neq0$, taking the inner product with every element of the basis gives the system

\begin{align*}
   &\begin{pmatrix}
                \langle \mathcal{T}_{a,\mu}\psi,\psi_1^+(a,\mu)\rangle \\ \langle \mathcal{T}_{a,\mu}\psi,\psi_1^-(a,\mu)\rangle \\ \langle \mathcal{T}_{a,\mu}\psi,\psi_0(a,\mu)\rangle
            \end{pmatrix}=\\
            &\begin{pmatrix}
                \langle \mathcal{X}_{\mu}\psi_1^+(a,\mu),\psi_1^+(a,\mu)\rangle & \langle \mathcal{X}_{\mu}\psi_1^-(a,\mu),\psi_1^+(a,\mu)\rangle & \langle \mathcal{X}_{\mu}\psi_0(a,\mu),\psi_1^+(a,\mu)\rangle\\\langle \mathcal{X}_{\mu}\psi_1^+(a,\mu),\psi_1^-(a,\mu)\rangle & \langle \mathcal{X}_{\mu}\psi_1^-(a,\mu),\psi_1^-(a,\mu)\rangle & \langle \mathcal{X}_{\mu}\psi_0(a,\mu),\psi_1^-(a,\mu)\rangle\\\langle \mathcal{X}_{\mu}\psi_1^+(a,\mu),\psi_0(a,\mu)\rangle & \langle \mathcal{X}_{\mu}\psi_1^-(a,\mu),\psi_0(a,\mu)\rangle & \langle \mathcal{X}_{\mu}\psi_0(a,\mu),\psi_0(a,\mu)\rangle
            \end{pmatrix}\begin{pmatrix}
                \gamma_1^+\\ \gamma_1^-\\ \gamma_0
            \end{pmatrix}.\end{align*}
            
\no Since the unperturbed basis $\{\cos{z}, \sin{z}, 1/\sqrt{2}\}$ satisfies the bracket identities of Definition \ref{d:2}, it is $\mu$-symplectic and reversible. By Lemma \ref{l:b2}, the perturbed basis $\{\psi_1^+(a,\mu), \psi_1^-(a,\mu), \psi_0(a,\mu)\}$, constructed via a smooth Kato-type transformation, inherits this $\mu$-symplectic and reversible structure for small values of $a$ and $\mu$. It follows that

\begin{align*}
\begin{pmatrix}
    \langle \mathcal{T}_{a,\mu}\psi,\psi_1^+(a,\mu)\rangle \\ \langle \mathcal{T}_{a,\mu}\psi,\psi_1^+(a,\mu)\rangle \\ \langle \mathcal{T}_{a,\mu}\psi,\psi_1^+(a,\mu)\rangle
\end{pmatrix}=\begin{pmatrix}
    i\dfrac{\mu}{1-\mu^2}&-\dfrac{1}{1-\mu^2}&0\\\dfrac{1}{1-\mu^2}&i\dfrac{\mu}{1-\mu^2}&0\\0&0&-\dfrac{i}{\mu}
\end{pmatrix}
\begin{pmatrix}
\gamma_1^+\\ \gamma_1^-\\ \gamma_0
\end{pmatrix}.
\end{align*}

\no Inverting this, the $3\times 3$ matrix $Q_{a,\mu}$ representing the Hamiltonian and reversible operator $\mathcal{Q}_{a,\mu}$ with respect to any $\mu$-symplectic and reversible basis of $\mathscr{F}_{a,\mu}$ for $\mu\neq0$ can be decomposed using \eqref{e:h} as  $Q_{a,\mu}=A_\mu T_{a,\mu}$, with

 

\begin{equation}\label{e:m}
A_\mu=
\begin{pmatrix}
i\mu&1&0\\-1&i\mu&0\\0&0&i\mu
\end{pmatrix},
\end{equation}

\begin{equation}\label{e:m111}
    T_{a,\mu}=
\begin{pmatrix}
\langle \mathcal{T}_{a,\mu}\psi_1^+(a,\mu),\psi_1^+(a,\mu)\rangle & \langle \mathcal{T}_{a,\mu}\psi_1^-(a,\mu),\psi_1^+(a,\mu)\rangle & \langle \mathcal{T}_{a,\mu}\psi_0(a,\mu),\psi_1^+(a,\mu)\rangle\\\langle \mathcal{T}_{a,\mu}\psi_1^+(a,\mu),\psi_1^-(a,\mu)\rangle & \langle \mathcal{T}_{a,\mu}\psi_1^-(a,\mu),\psi_1^-(a,\mu)\rangle & \langle \mathcal{T}_{a,\mu}\psi_0(a,\mu),\psi_1^-(a,\mu)\rangle\\\langle \mathcal{T}_{a,\mu}\psi_1^+(a,\mu),\psi_0(a,\mu)\rangle & \langle \mathcal{T}_{a,\mu}\psi_1^-(a,\mu),\psi_0(a,\mu)\rangle & \langle \mathcal{T}_{a,\mu}\psi_0(a,\mu),\psi_0(a,\mu)\rangle
\end{pmatrix}.
\end{equation} 

\no For $\mu=0$, $\psi_1^+,\psi_1^-\in L_0^2(\T)$ and $\psi_0\in\ker{\mathcal{A}_0}=\text{{\rm span}} \langle 1 \rangle$, i.e., $\psi_0$ is a constant function.
Define
 \begin{equation}
     \psi^\circ=\psi-\left\langle\psi,\psi_0\right\rangle\psi_0,
 \end{equation}
 
 \no the function $\psi$ with its projection on $\psi_0$ removed, leaving only the part of $\psi$ orthogonal to $\psi_0$. So, $\psi^\circ\in L^2_0(\T)$, orthogonal to $\psi_0$. We write $\psi^\circ$ as a linear combination of the two remaining basis elements:
 
 \begin{equation}
\psi^\circ=\gamma_1^+\psi_1^++\gamma_1^-\psi_1^-.
 \end{equation}
 
\no Taking the inner product of $ \mathcal{X}_0 \psi^\circ $ with $ \psi_1^+ $ and $ \psi_1^- $, and using the symplectic pairing identities from Definition~\ref{d:2} since $ \psi_1^+, \psi_1^- $ are a $\mu$-symplectic pair, we obtain

\begin{equation} \gamma_1^+=\langle\mathcal{X}_0\psi^\circ,\psi_1^-\rangle\quad\text{and}\quad \gamma_1^-=-\langle\mathcal{X}_0\psi^\circ,\psi_1^+\rangle.
\end{equation}
 
\no The following decomposition of $\psi$ preserves the symplectic structure and ensures consistency with the symplectic pairing between $\psi_1^+$ and $\psi_1^-$:

\begin{equation}\label{e:f}
\psi=\langle\mathcal{X}_0\psi^\circ,\psi_1^-\rangle\psi_1^+-\langle\mathcal{X}_0\psi^\circ,\psi_1^+\rangle\psi_1^-+\langle\psi,\psi_0\rangle\psi_0.
\end{equation}



\no We express $\psi$ in the $0$-symplectic basis ({i.e., $\mu=0$}) $ \{ \psi_1^+, \psi_1^-, \psi_0 \} $. To compute the matrix representation of $ \mathcal{Q}_{a,0} $ in this basis, we observe that 
\beq
\left\langle \mathcal{Q}_{a,0} \psi, \psi_0 \right\rangle 
= -\left\langle \mathcal{T}_{a,0} \psi, \mathcal{A}_0 \psi_0 \right\rangle,
\eeq
since $ \mathcal{Q}_{a,0} = \mathcal{A}_0 \mathcal{T}_{a,0} $. Because $ \psi_0 $ lies in the kernel of $ \mathcal{A}_0 $, this identity ensures the third component decouples and we focus on the action in the $ \{ \psi_1^+, \psi_1^- \} $ - subspace. Applying $ \mathcal{T}_{a,0} $ to $ \psi $, and using the symplectic identities in Definition~\ref{d:2}, we obtain the following coordinate representation of $\psi$ in the symplectic basis:
\beq
\begin{pmatrix}
\langle \mathcal{X}_0 \psi^\circ, \psi_1^- \rangle \\
\langle \mathcal{X}_0 \psi^\circ, \psi_1^+ \rangle
\end{pmatrix}
=
\begin{pmatrix}
0 & 1 \\
-1 & 0
\end{pmatrix}
\begin{pmatrix}
\langle \mathcal{T}_{a,0} \psi, \psi_1^+ \rangle \\
\langle \mathcal{T}_{a,0} \psi, \psi_1^- \rangle
\end{pmatrix}.
\eeq

This matrix expression captures the action of $ \mathcal{Q}_{a,0} $ in the 0-symplectic subspace, consistent with the canonical structure. Therefore, for all $\mu$ the matrix representing the action of $\mathcal{Q}_{a,\mu}:\mathscr{F}_{a,\mu}\to \mathscr{F}_{a,\mu}$ is given by \eqref{e:m}-\eqref{e:m111}. We observe that for any $\Psi_1,\Psi_2\in L^2(\T)$, $\langle\Psi_1,\Psi_2\rangle=\overline{\langle\overline{\mathcal{P}}\Psi_1,\overline{\mathcal{P}}\Psi_2\rangle}$, where $\overline{\mathcal{P}}$ is the complex involution defined in \eqref{e:in}. Using this, along with \eqref{e:tr} and the properties of the reversible basis, it follows that the entries of the matrix $T_{a,\mu}$ are alternating between real and imaginary.

The following definitions are used below.

 \begin{definition}\label{d:3} 
A $3 \times 3$ matrix $Q = AT$ is

\begin{enumerate}[(i)]

    \item {\bf Complex Hamiltonian}, if $T$ is self adjoint and $A$ is skew adjoint with respect to the standard scalar product of $\mathbb{C}^3$,
    
    \item {\bf reversible}, if $Q \circ \mathcal{P} = -\mathcal{P} \circ Q$, where
    
    \[
    \mathcal{P} =
    \begin{pmatrix}
    \varrho & 0 & 0 \\
    0 & -\varrho & 0 \\
    0 & 0 & \varrho
    \end{pmatrix},
    \]
    
    \no and $\varrho: z \mapsto \bar{z}$ denotes conjugation in the complex plane.

\end{enumerate}

Finally, an invertible matrix $Y$ is said to be $J$-symplectic if $Y J Y^* = J$.
\end{definition}

 \begin{lemma}\label{l:Asym}
     Let $T$ be a self-adjoint and reversible matrix and $A$ be a skew-adjoint and reversible matrix. Then for all $\tau\in\R$, $\exp{(\tau AT)}$ is $A$-symplectic and reversibility preserving.
 \end{lemma}
  \begin{proof}
      The proof is provided in \cite{MasperoRadakovic2024}[Lemma 3.9].
  \end{proof}
  
\begin{proposition}[Matrix representation of $\mathcal{Q}_{a,\mu}$ on $\mathscr{F}_{a,\mu}$]\label{p:1}
The action of the Hamiltonian and reversible operator $\mathcal{Q}_{a,\mu}$ on the symplectic and reversible basis 
$\psi_1^+(a,\mu),\psi_1^-(a,\mu),\psi_0(a,\mu)$ of $\mathscr{F}_{a,\mu}$ is represented by the $3\times 3$ reversible and Hamiltonian matrix

\begin{equation}
      Q_{a,\mu}=A_\mu T_{a,\mu},
\end{equation}

\no where $A_\mu$ is given in \eqref{e:m} and $T_{a,\mu}$ is the self-adjoint and reversibility-preserving matrix

\begin{equation}
      T_{a,\mu}= 
      \left( 
      \begin{array}{c|c}
\Lambda_1 & \Lambda_2 \\ 
\hline
\Lambda_2^\dagger & \Lambda_3
\end{array} 
\right).
\end{equation}

Here, $\Lambda_1$ is a $2\times 2$ self-adjoint matrix,

\begin{equation}
      \Lambda_1(N,a,\mu)=\begin{pmatrix}
            \mathcal{F}_1(N,a) -\mu^2\Lambda_{11}^{\mu}(N,\mu) & i\mu\Lambda_{12}(N,\mu) \\ 
           -i\mu\Lambda_{12}(N,\mu) & -\mu^2\Lambda_{11}^\mu(N,\mu)
      \end{pmatrix},
\end{equation}

\no and

\begin{equation}
      \Lambda_2(N,a,\mu)=\begin{pmatrix}
          \mathcal{F}_2(N,a)\\0
      \end{pmatrix}, 
      \qquad 
      \Lambda_3(N,a,\mu)=\jmath(\rho)-1+ \mathcal{F}_3(N,a)+\Lambda_{33}^\mu(N,\mu).
\end{equation}

The coefficients depend on the parity of $N$ as follows.

\medskip
\noindent
\textbf{(i) Even $N$.}  
For even $N$,
\begin{align*}
\mathcal{F
}_1(N,a)&=a^{2N-2}\Lambda_{11}^a(N)+\mathcal{O}(a^{2N-1}), 
\qquad 
\mathcal{F}_2(N,a)=a^{N-1}\Lambda_{13}(N)+\mathcal{O}(a^{2N-1}), 
\\
\mathcal{F}_3(N,a)&=a^{2N-2}\Lambda_{33}^a(N)+\mathcal{O}(a^{2N-1}),
\end{align*}
and the quantities $\Lambda_{ij}$ are given by
\begin{align}\nonumber
    \Lambda_{11}^a(N)=&2N(1-N)\left\langle\cos^Nz,\mathcal{D}_\rho^{-1}(\cos^Nz)\right\rangle +N^2\left\langle\cos^Nz,\dfrac{\langle\cos^Nz,1\rangle}{2(\jmath(\rho)-1)}\right\rangle,\\
    \Lambda_{11}^\mu(N,\mu)=&\sum_{m=1}^{N-1}\rho^{2m}\dfrac{\mu^{2m-2}}{2m!}\jmath^{(2m)}(\rho)+\mathcal{O}(\mu^{2N-3}),\\
    \Lambda_{12}(N,\mu)=& \rho \jmath^\prime(\rho)+\sum_{m=2}^{N-1}\rho^{2m-1}\dfrac{\mu^{2m-2}}{(2m-1)!}\jmath^{(2m-1)}(\rho)+\mathcal{O}(\mu^{2N-2}),\\
    \Lambda_{13}(N)=&-\dfrac{1}{\sqrt{2}}N\left<\cos^{N}{z},1\right>,\\
\Lambda_{33}^a(N)=&N\langle \cos^Nz,\mathcal{D}_\rho^{-1}(\cos^Nz)\rangle- \dfrac{1}{2}N(N-1)\langle\cos^{N-2}z,\mathcal{D}_\rho^{-1}(\cos^Nz)\rangle-\\&\dfrac{1}{2}N^2\langle\cos^{N-1}z,\mathcal{D}_\rho^{-1}(\cos^{N-1}z-\cos{z}\langle\cos^{N-1}{z},\cos{z}\rangle)\rangle,\\
\Lambda_{33}^\mu(N,\mu)=&-\sum_{m=1}^{N-1} \left(\frac{\rho^{2m}\mu^{2m}}{(2m)!}\jmath^{(2m)}(0)\right)+\mathcal{O}(\mu^{2N-1}).
\end{align}

\medskip
\noindent
\textbf{(ii) Odd $N$.}  
For odd $N$,
\[
\mathcal{F}_1(N,a)=-a^{N-1}\alpha\Lambda_{11}^a(N)+\mathcal{O}(a^N), 
\qquad 
\mathcal{F}_2(N,a)=\mathcal{O}(a^{N}), 
\qquad 
\mathcal{F}_3(N,a)=a^{N-1}\Lambda_{33}^a(N)+\mathcal{O}(a^N),
\]
and the quantities $\Lambda_{ij}$ are given by
\begin{align}\nonumber
    \Lambda_{11}^a(N)=&(N-1)\langle\cos^{N}z,\cos{z}\rangle,\\
    \Lambda_{11}^\mu(N,\mu)=&\sum_{m=1}^{(N-1)/2} \left(\frac{\rho^{2m}\mu^{2m-2}}{(2m)!}\jmath^{(2m)}(\rho)\right)+\mathcal{O}(\mu^N),\\
     \Lambda_{12}(N,\mu)=& \rho \jmath^\prime(\rho)+\sum_{m=1}^{(N-3)/2}\rho^{2m+1}\dfrac{\mu^{2m}}{(2m+1)!}\jmath^{(2m+1)}(\rho)+\mathcal{O}(\mu^N),\\
\Lambda_{33}^a(N)=&\alpha\langle\cos^Nz,\cos{z}\rangle-\dfrac{1}{2}\alpha N\langle\cos^{N-1}z,1\rangle,\\
     \Lambda^\mu_{33}(N,\mu)=&-\sum_{m=1}^{(N-1)/2} \left(\frac{\rho^{2m}\mu^{2m}}{(2m)!}\jmath^{(2m)}(0)\right)+\mathcal{O}(\mu^N).
\end{align}

\end{proposition}

\begin{proof}
    The proof is provided in Appendix \ref{a:4}.
\end{proof}

Next, we consider the block-diagonalization of the matrix $Q_{a,\mu}$.
If $N$ is odd, $Q_{a,\mu}$ is already block-diagonal.  However, for even $N$ this is not the case. We develop a block-diagonalization for $Q_{a,\mu}$ when $N$ is even.  Lengthy calculations give rise to  following theorem.

\begin{theorem}[Block-diagonal Matrix for even $N$]
For even $N$ and sufficiently small $a$ and $\mu$, the operator $\mathcal{Q}_{a,\mu}$ can be represented by a $3\times3$ block diagonal matrix $Q_{a,\mu}$ as

\beq
Q_{a,\mu}=
\begin{pmatrix}
\begin{array}{c|c}
        \Gamma_{1} & 0\\
        \hline
        0 & \Gamma_{2}
        \end{array}
    \end{pmatrix}.
\eeq
    
\no  Here, $\Gamma_1$ is a $2\times 2$ matrix given by

\beq \Gamma_1=\mu\begin{pmatrix}
    -i\Lambda_{12}(N,\mu)+\mathcal{O}(a^{2N-2},\mu^{2N-2})& \Lambda_{b}(N,\mu)+\mathcal{O}(a^{2N-2})\\\Lambda_f(N)a^{2N-2}-\mu^2\Lambda_b(N,\mu)+\mathcal{O}(a^{2N-1})& -i\Lambda_{12}(N,\mu)+\mathcal{O}(a^{2N-2},\mu^{2N-2})
\end{pmatrix},
\eeq

\no and 

\beq
\Gamma_2=i\mu\left(\Lambda_3(N,a,\mu)+\mathcal{O}( a^{2N-1})\right),
\eeq 

\no with

\beq
\Lambda_f(N)=-\Lambda_{11}^a(N)+\dfrac{\Lambda^2_{13}(N)}{\widetilde{\Lambda}_d},\quad \Lambda_b(N,\mu)=-\Lambda_{12}(N,\mu)-\Lambda_{11}^\mu(N,\mu),
\eeq

\beq\label{e:del}
\widetilde{\Lambda}_d=\rho \jmath^\prime(\rho)+\jmath(\rho)-1+ \mathcal{O}(a^2,a\mu,\mu^2).
\eeq

\no The expressions for $\Lambda_{11}^a$, $\Lambda_{11}^{\mu}$, $\Lambda_{12}$ and $\Lambda_{3}$ are provided above in Proposition \ref{p:1}.
\end{theorem}

\begin{proof}
    The proof is presented in Appendix \ref{a:5}.
\end{proof}

\vs

{\bf Example.} {\em For the focusing modified KdV equation with $\alpha=-1$, $N=3$, $\jmath(k)=1+k^2$, the $3\times 3$ reversible and Hamiltonian matrix $Q_{a,\mu}$ is given by

\[
Q_{a,\mu}=\begin{pmatrix}
    -2i\rho^2\mu&-3\rho^2\mu^2&0\\ -\dfrac{3 a^2}{2}+3\rho^2\mu^2&-2i\rho^2\mu&0\\ 0&0&i\rho^2\mu
\end{pmatrix}+\mathcal{O}(a^3,a^2\mu,\mu^3).
\]
}

\vs

{\bf Example.} {\em With $N=2$, including for the Whitham equation, the $3\times 3$ reversible and Hamiltonian matrix $Q_{a,\mu}$ is

\[
Q_{a,\mu}=
\mu\begin{pmatrix}
-i\rho\jmath'(\rho) &
-\rho\jmath'(\rho)-\rho^2\jmath''(\rho)/2 &
0 \\

g_{21}
&
-i\rho\jmath'(\rho)&
0 \\

0 & 0 & i(\jmath(\rho)-1)
\end{pmatrix}+\mathcal{O}(a^4,a^3\mu,\mu^4),
\]

\no where 

\[
g_{21}=
a^2\left(
\begin{aligned}
\frac{\rho\jmath'(\rho)+3\jmath(\rho)-2\jmath(2\rho)-1}
{(\jmath(\rho)-\jmath(2\rho))(\rho\jmath'(\rho)+\jmath(\rho)-1)}
\end{aligned}
\right)
+\mu^2\left(\rho\jmath'(\rho)+\frac{\rho^2}{2}\jmath''(\rho)\right).
\]

}
\vs

\no The block-diagonal representation of the matrix $Q_{a,\mu}$ is given by

\beq
Q_{a,\mu}=
\begin{pmatrix}
\begin{array}{c|c}
        \Gamma_{1} & 0\\
        \hline
        0 & \Gamma_{2}
        \end{array}
    \end{pmatrix},
\eeq

\no for $N$ both even and odd. This allows us to analyze the spectrum of the operator $\mathcal{Q}_{a,\mu}$ through the eigenvalues of the blocks $\Gamma_1$ and $\Gamma_2$. Since $\Gamma_2$ is a scalar multiple of $i\mu$, its eigenvalue

\beq
\lambda_0(a,\mu)=i\mu\Lambda_3(N,a,\mu), 
\eeq

\no is purely imaginary for sufficiently small $a$ and $\mu$, and therefore it does not lead to a spectral instability. Consequently, the modulational instability of periodic traveling waves of \eqref{e:gkdv} is determined by the eigenvalues of the $2\times2$ matrix $\Gamma_1$. Ignoring higher-order terms for clarity, $\Gamma_1$ has the structure
\[
\Gamma_1 \approx
\mu
\begin{pmatrix}
-i\Lambda_{12}(N,\mu) & \Lambda_b(N,\mu) \\
\Lambda_f(N)a^{2N-2}-\mu^2\Lambda_b(N,\mu) & -i\Lambda_{12}(N,\mu)
\end{pmatrix},
\]

\no for even $N$, while for odd $N$ it takes the form

\[
\Gamma_1 \approx
\begin{pmatrix}
-i\mu\Lambda_{12}(N,\mu) & \mu^2\Lambda_b(N,\mu) \\
-\alpha\Lambda_{11}^a(N)a^{N-1}-\mu^2\Lambda_b(N,\mu) & -i\mu\Lambda_{12}(N,\mu)
\end{pmatrix}.
\]

\no In both cases, the characteristic polynomial has the same leading-order structure, yielding the eigenvalues

\beq
\lambda^\pm_1(a,\mu)
=
-i\mu\Lambda_{12}(N,\mu)
\pm
\mu\sqrt{\lambda_N\Lambda_b(N,\mu)-\mu^2\Lambda_b^2(N,\mu)},
\eeq

\no where $\lambda_N=\Lambda_f(N)a^{2N-2}$ for even $N$ and $\lambda_N=-\alpha\Lambda_{11}^a(N)a^{N-1}$ for odd $N$. Instabilities originate from $\lambda^\pm_1(a,\mu)$ having a nonzero real part, i.e., when the radicand $\lambda_N\Lambda_b(N,\mu)-\mu^2\Lambda_b^2(N,\mu)$ is positive. 


\section{Main Results}\label{sec:results}
In this section, we state the main results concerning the modulational instability of the small-amplitude periodic traveling waves of the generalized KdV equation \eqref{e:gkdv}. As shown in the previous section, instability occurs when the radicand in the eigenvalue formula becomes positive. This observation leads to the introduction of the Whitham–Benjamin coefficient, whose sign determines the modulational stability or instability of the periodic waves.

\subsection{Results for $N$ even}

\no We define $\Delta_{\rm{even}}(N)$, the Whitham–Benjamin coefficient for the generalized KdV equation \eqref{e:gkdv} with even nonlinearity power $N$:

\beq\label{e:del1}
\Delta_{\rm{even}}(N)=-\left(\rho\jmath^\prime(\rho)+\dfrac{1}{2}\rho^2\jmath^{\prime\prime}(\rho)\right)\Lambda_f(N)=\widetilde{\Lambda}_b\Lambda_f(N), 
\eeq

\no which defines $\tilde\Lambda_b$, the leading-order part of $\Lambda_b$, omitting higher-order dependence on $\mu$. Our main result provides a criterion for the modulational instability of periodic traveling waves of \eqref{e:gkdv} in this setting.

\begin{theorem}
    For even $N$, a $2\pi/\rho$-periodic traveling wave of the generalized KdV equation \eqref{e:gkdv} of sufficiently small amplitude is susceptible to the modulational or Benjamin-Feir instability if 
    \beq\label{e:wb}
    \Delta_{\rm{even}}>0.
    \eeq
\end{theorem}

Indeed, if $\Delta_{\rm{even}}>0$, the first term in the radicand is positive, implying the existence of a range of $\mu$ values for which the whole radicand is positive. 

\vs 

{\bf Example.} {\em For $N=2$, 

\beq
\Delta_{\rm{even}}=\dfrac{\left(\rho\jmath^\prime(\rho)+\frac{1}{2}\rho^2\jmath^{\prime\prime}(\rho)\right)(\rho\jmath^\prime(\rho)+3\jmath(\rho)-2\jmath(2\rho)-1)}{(\jmath(2\rho)-\jmath(\rho))(\rho\jmath^\prime(\rho)+\jmath(\rho)-1)},
\eeq

\no which agrees with \cite{Hur2015ModulationalWaves,MasperoRadakovic2024}. For the Whitham equation, $\jmath(\rho)=\sqrt{\tanh(\rho)/\rho}$. 
Since $\jmath(\rho)<1$ and $\jmath'(\rho)<0$ for $\rho>0$, we have
\[
\rho \jmath'(\rho) + \jmath(\rho) - 1 < 0, \qquad 
\rho \jmath'(\rho) + \frac{1}{2} \rho^2 \jmath''(\rho) < 0.
\]

The remaining factor in $\Delta_{\rm even}$,
\[
\rho \jmath'(\rho) + 3\jmath(\rho) - 2\jmath(2\rho) - 1,
\]
vanishes at a unique critical value $\rho_c \approx 1.146\ldots$ and is negative for $\rho>\rho_c$. 
Hence, $\Delta_{\rm even}>0$ for $\rho>\rho_c$, indicating modulational instability, as observed numerically \cite{Hur2015ModulationalWaves}
.} 



\vs

Next, we discuss the topology of the unstable spectrum near the origin. 

\begin{theorem}[Benjamin-Feir unstable eigenvalues for even $N$]
   Assume $\Delta_{\rm{even}}>0$. For sufficiently small $a>0$, there exists a function $\mu_*(a)>0$ satisfying
   
\beq\label{picardeven}
\Delta_{\rm{even}}(N)a^{2N-2}-\mu_*^2\Lambda_b^2(N,\mu_*)=\mathcal{O}(a^{2N-1},a^{2N-1}\mu_*,\mu_*^{2N-1}),
\eeq
   
\no so that the Benjamin-Feir eigenvalues of the spectrum of the operator $\mathcal{Q}_{a,\mu}$ are given by

\begin{align}\label{evenlambda}
   \lambda^\pm_1(a,\mu)=-i\mu\Lambda_{12}(N,\mu)+\mathcal{O}(a^{2N-2}\mu,\mu^{2N-1})\pm\mu\sqrt{\Delta_{\rm{BF}}(N,a,\mu)}, 
   \end{align}

   \no for $\mu\in [0, \mu_*]$, with

   \beq
\Delta_{\rm{BF}}(N,a,\mu)=\Delta_{\rm{even}}(N)a^{2N-2}-\mu^2\Lambda_b^2(N,\mu)+\mathcal{O}(a^{2N-1},a^{2N-1}\mu,\mu^{2N-1}).
   \eeq

\no To leading order, $\mu_*(a)=\Delta_{\rm even}^{1/2}(N)a^{N-1}/{\tilde \Lambda_b}$. Equation \eqref{picardeven} may be used to determine $\mu_*(a)$ to higher order by Picard iteration.

\end{theorem}

From \eqref{e:spe}, the unstable spectrum of $\mathcal{Q}_a$ near the origin is fully characterized by the two curves $\mu \mapsto \lambda^\pm_1(a,\mu)$. For even $N$, as $\mu$ varies within the interval $[0, \mu_*]$, these curves form a closed figure resembling the shape of an ``8''. The upper part of this figure arises when $\Delta_{\rm{even}}-\mu^2\Lambda_b^2 > 0$ for $\mu \in [0, \mu_*)$, while the lower part appears for $\mu < 0$. 
Indeed, ignoring the $\mathcal{O}$ contributions in \eqref{evenlambda}, to leading order the unstable spectrum near the origin is given by a Lemniscate of Gerono or Huygens~\cite{lawrence}: with $\mu\in [-\mu_*, \mu_*]$, 

\begin{align*}
p&=\mu \sqrt{\Delta_{\rm{even}}a^{2N-2}-\mu^2\tilde \Lambda_b^2},\\
q&=-\mu \Lambda_{12},
\end{align*}

\no where $p=\mbox{Re} \lambda^\pm_1(a,\mu)$, $q=\mbox{Im} \lambda^\pm_1(a,\mu)$. Eliminating $\mu$, 

\[
p^2=\frac{q^2}{(\rho\jmath^\prime(\rho))^2}\left(\Delta_{\rm{even}}a^{2N-2}-\frac{q^2}{(\rho\jmath^\prime(\rho))^2}\widetilde{\Lambda}_b^2\right). 
\]

\no The top of this figure eight is given by 

\[
q_{\rm{max}}=\frac{|\rho\jmath^\prime(\rho)|a^{N-1}\Delta_{\rm{even}}^{1/2}}{|\widetilde{\Lambda}_b|}, 
\]

\no and its width

\[
p_{\rm{max}}-p_{\rm{min}}=\frac{\Delta_{\rm{even}}a^{2N-2}}{|\widetilde{\Lambda}_b|}. 
\]

\no The coefficient $\Delta_{\rm{even}}$ depend on $N$ as well, but not exponentially. Thus these results show that for solutions of small amplitude $a$, the figure-eight curves shrink in size exponentially fast with $N$, with the width shrinking quadratically faster than the height.  
The shrinking of the figure eights for even degree nonlinearity $N$ is illustrated for the case of the Whitham equation with $N=2$ and $N=4$ (red figures) in Fig.~\ref{fig:combo}. 

\subsection{Results for $N$ odd.}

The same logic as for the even case is followed for the odd case. We define $\Delta_{\rm{odd}}$, the Whitham–Benjamin coefficient for the generalized KdV equation \eqref{e:gkdv} with odd nonlinearity power $N$: 

\beq
\Delta_{\rm{odd}}=-\alpha\left(\rho j'(\rho)+\dfrac{1}{2}\rho^2\jmath^{\prime\prime}(\rho)\right)=\alpha\widetilde{\Lambda}_b.
\eeq

\no As above, we provide a criterion for the presence of the modulational instability. 

\begin{theorem}\label{t:odd}
    For odd $N$, a $2\pi/\rho$-periodic traveling wave of the generalized KdV equation \eqref{e:gkdv} of sufficiently small amplitude is susceptible to the modulational or Benjamin-Feir instability if 
    
    \beq\label{e:wb}
    \Delta_{\rm{odd}}>0.
    \eeq

\end{theorem}




\no This corresponds to the focusing behavior of equation \eqref{e:gkdv}. Linearizing \eqref{e:gkdv} about the trivial solution $u=0$ and seeking solutions of the form $u(x,t)=\exp({i\rho x-i\omega t)})$ yields the dispersion relation

\[
\omega(\rho)=\rho\jmath(\rho).
\]

\no and the group velocity is $\omega'(\rho)$ with dispersion curvature $\omega''(\rho)$. A standard multiple-scales reduction shows that the slowly modulated small-amplitude wavetrains satisfy, to leading order, a cubic NLS with dispersion coefficient proportional to $\omega''(\rho)$ and nonlinearity proportional to $\alpha$. Thus the generalized KdV equation \eqref{e:gkdv} is focusing if

\[
\omega''(\rho)\alpha < 0,
\]

\no and defocusing if

\[
\omega''(\rho)\alpha > 0, 
\]

\no consistent with Theorem~\ref{t:odd}, where $\Delta_{\rm{odd}}=-\alpha \rho \omega''(\rho)/2$. 

\vs

{\bf Example.} {\em For the focusing modified KdV equation with $\alpha=-1$, $N=3$, $\jmath(k)=1+k^2$, the criterion \eqref{e:wb} is 

\[
3\rho^2>0,
\]

\no and all small-amplitude solutions of the focusing modified KdV equation are succeptible to the modulation instability.} 

\vs

\begin{theorem}[Benjamin-Feir unstable eigenvalues for odd $N$]
   Assume $\Delta_{\rm{odd}}>0$. For sufficiently small $a>0$, there exists a function $\mu_*(a)$ that satisfies
   
\beq\label{picardodd}
\Delta_{\rm{odd}} a^{N-1}-\mu_*^2\Lambda_b^2(N,\mu_*)=\mathcal{O}(a^{N},a^{N-1}\mu_*,\mu_*^{N}),
\eeq
   
\no so that the Benjamin-Feir eigenvalues of the spectrum of the operator $\mathcal{Q}_{a,\mu}$ are given by

\begin{align}\label{oddlambda}
   \lambda^\pm_1(a,\mu)=-i\mu\Lambda_{12}(N,\mu)+\mathcal{O}(a^{N-1}\mu,\mu^N)\pm\mu\sqrt{\Delta_{\rm{BF}}(N,a,\mu)},
\end{align}

\no for $\mu\in[0,\mu_*]$, with

\beq
\Delta_{\rm{BF}}(N,a,\mu)=\Delta_{\rm{odd}}\Lambda_{11}^a(N) a^{N-1}-\mu^2\Lambda_b^2(N,\mu)+\mathcal{O}(a^N,a^{N-1}\mu,\mu^{N}).
\eeq

\no To leading order, $\mu_*(a)=\Delta_{\rm odd}^{1/2}a^{(N-1)/2}/{\tilde \Lambda_b}$. Equation \eqref{picardodd} may be used to determine $\mu_*(a)$ to higher order by Picard iteration.

\end{theorem}

As for the even power case, the unstable spectrum of $\mathcal{Q}_a$ near the origin is fully characterized by the two curves $\mu \mapsto \lambda^\pm_1(a,\mu)$. As $\mu$ varies within the interval $[0, \mu_*]$, these curves form a closed figure resembling the shape of an ``8''. The upper part of this figure arises when $\Delta_{\rm{odd}}-\mu^2\Lambda_b^2 > 0$ for $\mu \in [0, \mu_*)$, while the lower part appears for $\mu < 0$. 
This figure-eight is easily analyzed analogously as the even case: ignoring the $\mathcal{O}$ contributions in \eqref{oddlambda}, to leading order the unstable spectrum near the origin is given by the Gerono/Huygens Lemniscate again: with $\mu\in [-\mu_*, \mu_*]$, 

\begin{align*}
p&=\mu \sqrt{\Delta_{\rm{odd}}\Lambda_{11}^a(N)a^{N-1}-\mu^2\tilde \Lambda_b^2},\\
q&=-\mu \Lambda_{12},
\end{align*}

\no where $p=\mbox{Re} \lambda^\pm_1(a,\mu)$, $q=\mbox{Im} \lambda^\pm_1(a,\mu)$. Eliminating $\mu$, 

\[
p^2=\frac{q^2}{(\rho\jmath^\prime(\rho))^2}\left(\Delta_{\rm{odd}}\Lambda_{11}^a(N)a^{N-1}-\frac{q^2}{(\rho\jmath^\prime(\rho))^2}\widetilde{\Lambda}_b^2\right). 
\]

\no The top of this figure eight is given by 

\[
q_{\rm{max}}=\frac{|\rho\jmath^\prime(\rho)|a^{(N-1)/2}(\Delta_{\rm{odd}} \Lambda_{11}^a(N))^{1/2}}{|\widetilde{\Lambda}_b|}, 
\]

\no and its width

\[
p_{\rm{max}}-p_{\rm{min}}=\frac{\Delta_{\rm{odd}}\Lambda_{11}^a(N)a^{N-1}}{|\widetilde{\Lambda}_b|}. 
\]

\no These results show that also for odd $N$, for solutions of small amplitude $a$, the figure-eight curves shrink in size exponentially fast with $N$, with the width shrinking quadratically faster than the height, but at half the rates deduced for the even case. 

Regardless of the choice of the operator $\mathcal{J}$, the spectrum of $\mathcal{Q}_a$ near the origin is either purely imaginary or contains an unstable component that consistently exhibits a figure-eight shape. As for the even case, increasing $N$ results in a narrower figure-eight shape, see the case of the Whitham equation with $N=3$ and $N=5$ (blue figures) in Fig.~\ref{fig:combo}. It is observed there that the figure eight curves for odd powers are larger and wider (thus indicating more unstable waves) than those for lower even powers: the maximal growth rate for $N=3$ is larger than for $N=2$, that for $N=5$ is larger than for $N=4$, even though these exponents do not yet validate the asymptotic regime for large $N$ remarked above. 

Interesting shapes, different from a figure eight, are obtained by plotting the unstable spectrum for values of $a$ and $\rho$ where $a$ is not necessarily small. Numerical experiments show that these shapes do not represent the actual unstable spectrum near the origin, and the expressions obtained should only be used in the asymptotic regime relevant for their derivation. With $\rho$ fixed and decreasing values of $a$, figure-eight shapes are recovered, as required by our analytical results.

\begin{figure}[tb]
\centering
\def\svgwidth{6in}   
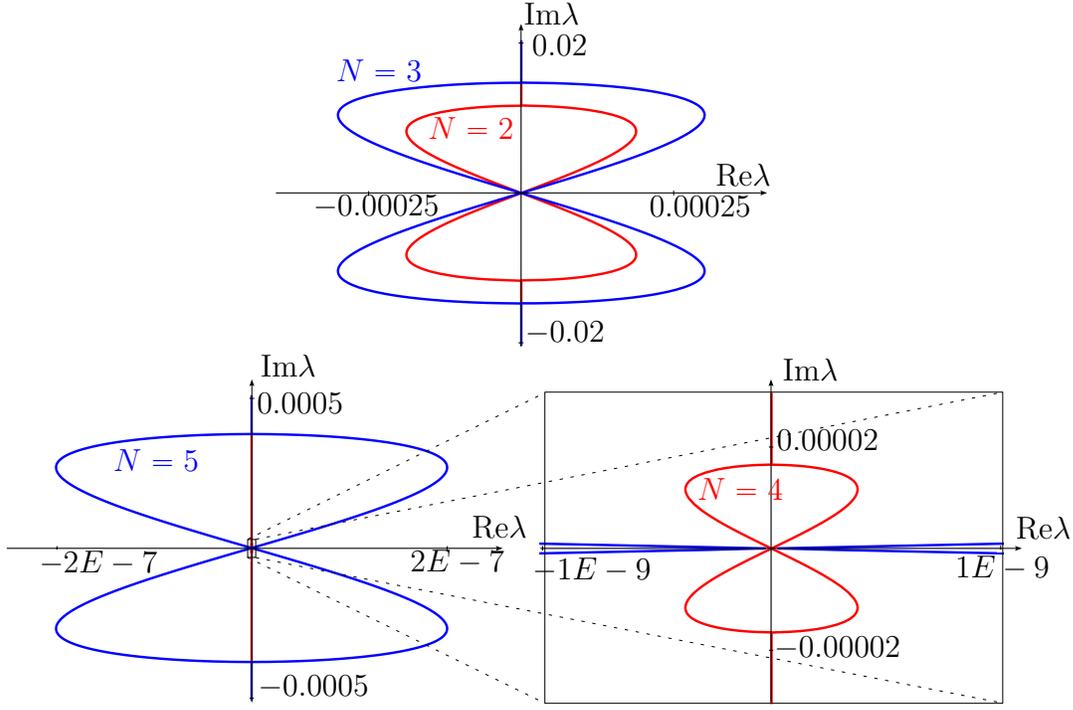
\caption{The unstable spectrum near the origin, for the Whitham equation ($N=2$) and its generalizations with higher-degree ($N=3,4,5$) monomial nonlinearity, with $a=0.02$ and $\rho=1.5$.}
\label{fig:combo}
\end{figure}

\section*{Acknowledgements.} Alberto Maspero, Ashish Pandey and Antonio Milosh Radakovic are thanked for useful conversations. 

\appendix

\section{Properties of the operator}\label{a:2}

Equation \eqref{e:gkdv} has the Hamiltonian formulation

\[
\partial_tu=\partial_x\nabla H(u),
\] 

\no where 

\begin{equation}\label{e:h}
      H(u)=-\dfrac{1}{2}\int_\R u(\mathcal{J}u)dx-\dfrac{1}{N+1}\int_\R u^{N+1}dx. 
\end{equation}

\no The equation is invariant under translation, as the vector field $Y(u)=-\partial_x\circ (\J u+\alpha u^N)$ satisfies 
\[
Y\circ \sigma_\varsigma=\sigma_\varsigma\circ Y,
\]

\no where $\sigma_\varsigma$ is the translation operator $[\sigma_\varsigma u](x)=u(x+\varsigma)$, $\varsigma\in\R$. Additionally, \eqref{e:gkdv} is reversible, with $Y(u)$ satisfying 

\[
Y\circ \mathcal{P}=-\mathcal{P} \circ Y,
\] 

\no where $\mathcal{P}$ is the involution $[\mathcal{P} u](x)=u(-x)$. The linearized operator $\mathcal{Q}_a$ in \eqref{e:op} preserves the system's structure. First, it is time reversible, and  second, it is linearly Hamiltonian, expressed as
\[
  \mathcal{Q}_a=\partial_x\circ \mathcal{T}_a,
\] 

\no where $\mathcal{T}_a$ is the symmetric operator

\[
\mathcal{T}_a=c-\mathcal{J}_\rho-N\alpha \eta^{N-1}. 
\] 

\no Consider the Floquet operator $\mathcal{Q}_{a,\mu}$ in \eqref{e:op1}. It assumes the form of a complex Hamiltonian and reversible operator:

\beq\label{e:hr}
\mathcal{Q}_{a,\mu}=(\partial_z+i\mu)\circ \mathcal{T}_{a,\mu},
\eeq

\no where 
the self-adjoint operator $\mathcal{T}_{a,\mu}$ is 
 
\beq\label{eq:tdef}
\mathcal{T}_{a,\mu}=c-e^{-i\mu z}\mathcal{J}_\rho e^{i\mu z}-N\alpha \eta^{N-1}.
\eeq

\no The operator $\mathcal{T}_{a,\mu}$ preserves reversibility, that is
 
\begin{equation}\label{e:tr}
\mathcal{T}_{a,\mu}\circ\overline{\mathcal{P}}=\overline{\mathcal{P}}\circ\mathcal{T}_{a,\mu}.
 \end{equation}

Note that $\operatorname{spec}_{L^2(\mathbb{T})}(\mathcal Q_{a,\mu})$ is not symmetric with respect to reflections about the real axis and the origin. Rather it exhibits the following properties. 

\begin{lemma}\label{l:2}

  Denote by $\mathcal{R}_r$, $\mathcal{R}_i$ and $\mathcal{R}_o$ reflections about the real axis, the imaginary axis, and the origin, respectively. Then
   
\begin{enumerate}

\item $\mathcal{R}_r(\mathcal{Q}_{a,\mu}\psi)(z)=(\mathcal{Q}_{a,-\mu}\mathcal{R}_r\psi)(z),$
     
\item $\mathcal{R}_i(\mathcal{Q}_{a,\mu}\psi)(z)=-(\mathcal{Q}_{a,\mu}\mathcal{R}_i\psi)(z),$

\item $\mathcal{R}_o(\mathcal{Q}_{a,\mu}\psi)(z)=-(\mathcal{Q}_{a,-\mu}\mathcal{R}_o\psi)(z),$
        
\item If $\lambda\in\operatorname{spec}(\mathcal{Q}_{a,\mu})$, then so is $-\Bar{\lambda}$. Furthermore, $-\Bar{\lambda}, -\lambda\in\operatorname{spec}(\mathcal{Q}_{a,-\mu})$.

\end{enumerate}

\end{lemma}

\begin{proof}

The operations $\mathcal{R}_r$, $\mathcal{R}_i$ and $\mathcal{R}_o$ are defined by

\begin{equation}\label{e:ref}
\mathcal{R}_r\psi(z)=\overline{\psi(z)},\quad   \mathcal{R}_i\psi(z)=\overline{\psi(-z)},\quad\text{and}\quad \mathcal{R}_o\psi(z)=\psi(-z).
\end{equation}

\no Using the explicit form of $\mathcal{Q}_{a,\mu}$, 

\begin{equation}\label{e:01}
    \mathcal{R}_r(\mathcal Q_{a,\mu}\psi)(z)=\overline{(\mathcal Q_{a,\mu}\psi)(z)}=\mathcal Q_{a,-\mu}\overline{\psi(z)}=(\mathcal Q_{a,-\mu}\mathcal{R}_r\psi)(z),
\end{equation}

\begin{equation}\label{e:2}
\mathcal{R}_i(\mathcal Q_{a,\mu}\psi)(z)=\overline{(\mathcal Q_{a,\mu}\psi)(-z)}=-\mathcal Q_{a,\mu}\overline{\psi(-z)}=-(\mathcal Q_{a,\mu}\mathcal{R}_i\psi)(z),
\end{equation}

\begin{equation}\label{e:3}
 \mathcal{R}_o(\mathcal Q_{a,\mu}\psi)(z)=(\mathcal Q_{a,\mu}\psi)(-z)=-\mathcal Q_{a,-\mu}\psi(-z)=-(\mathcal Q_{a,-\mu}\mathcal{R}_o\psi)(z).
\end{equation}

\no Assume $\lambda$ is an eigenvalue of $\mathcal Q_{a,\mu}$ with an associated eigenfunction $\varphi$:

\begin{equation}\label{e:eig1}
    \mathcal Q_{a,\mu}\varphi=\lambda\varphi.
\end{equation}

\no Using \eqref{e:eig1},

\begin{equation}
\mathcal Q_{a,-\mu}\mathcal{R}_r\varphi=-\mathcal{R}_r\mathcal Q_{a,\mu}\varphi=-\overline{\lambda}\mathcal{R}_r\varphi,
\end{equation}

\begin{equation}
\mathcal Q_{a,\mu}\mathcal{R}_i\varphi=-\mathcal{R}_i\mathcal Q_{a,\mu}\varphi=-\overline{\lambda}\mathcal{R}_i\varphi,
\end{equation}

\begin{equation}
\mathcal Q_{a,-\mu}\mathcal{R}_o\varphi=-\mathcal{R}_o\mathcal Q_{a,\mu}\varphi=-\lambda\mathcal{R}_o\varphi.
\end{equation}

\no It follows that if $\lambda$ is an eigenvalue of $\mathcal{Q}_{a,\mu}$ with eigenfunction $\varphi$, then $-\overline{\lambda}$ is an eigenvalue of $\mathcal{Q}_{a,\mu}$ with eigenfunction $\mathcal{R}_i\varphi$. Moreover, $-\overline{\lambda}$ and $-\lambda$ are eigenvalues of $\mathcal{Q}_{a,-\mu}$ with eigenfunctions $\mathcal{R}_r\varphi$ and $\mathcal{R}_o\varphi$, respectively.

\end{proof}

Given these properties, it suffices to consider $\mu\in(0,1/2]$. The results from Lemma~\ref{l:2} are not surprising: the 4-fold symmetry of the stability spectra of Hamiltonian problems is well known~\cite{KP}, but the problems considered here allow for the extra complication of nonlocal dispersion, which is why we include the statement and proof of the lemma. 

\begin{lemma}\label{l:31a}
    For $a$ and $\mu$ sufficiently small, the following properties hold. 
    
\begin{enumerate}[(a)]

\item The point spectrum of $\mathcal{Q}_{a,\mu}$ decomposes as 

\begin{equation*}
           \operatorname{spec}(\mathcal Q_{a,\mu})= \operatorname{spec}_0(\mathcal Q_{a,\mu})\cup\operatorname{spec}_1(\mathcal Q_{a,\mu}),
\end{equation*}

\no with

\begin{equation*}
           \operatorname{spec}_0(\mathcal Q_{a,\mu})\subset B(0;R/3),\quad \operatorname{spec}_1(\mathcal Q_{a,\mu})\subset \C\setminus \overline{B(0;R/2)},
\end{equation*}

\no where $R=|\jmath(\rho)-\jmath(2\rho)|$. In other words, for fixed $a$ and $\mu$, there is a separation between the point spectrum near the origin and the rest of the spectrum. The size of this separation is independent of $a$ or $\mu$. 
        
\item The spectral projection $\Pi_a(\mu)$ associated with $\operatorname{spec}_0(\mathcal Q_{a,\mu})$, 

\begin{equation}\label{e:pr}
\Pi_{a,\mu}=\dfrac{1}{2\pi i}\oint_{\partial B(0;R/3)}(\lambda-\mathcal{Q}_{a,\mu})^{-1}d\lambda, 
\end{equation}

\no satisfies $\|\Pi_{a,\mu}-\Pi_{0,0}\|=\mathcal{O}(|\mu|+|a|)$. The operators $\Pi_{a,\mu}$ are well-defined projectors commuting with $\mathcal{Q}_{a,\mu}$:

\begin{equation}\label{e:pp}
\Pi^2_{a,\mu}=\Pi_{a,\mu},\quad\Pi_{a,\mu}\mathcal{Q}_{a,\mu}=\mathcal{Q}_{a,\mu}\Pi_{a,\mu}.
\end{equation}   
        
\item The projectors $\Pi_{a,\mu}$ are similar to each other: 

\beq\label{e:u2}
\mathcal{U}_{a,\mu}\Pi_{0,0}\mathcal{U}^{-1}_{a,\mu}=\Pi_{a,\mu},\quad \mathcal{U}^{-1}_{a,\mu}\Pi_{a,\mu}\mathcal{U}_{a,\mu}=\Pi_{0,0},
\eeq

\no with transformation operators

\beq\label{e:u}
\mathcal{U}_{a,\mu}=(\mathcal{I}-(\Pi_{a,\mu}-\Pi_{0,0})^2)^{-1/2}[\Pi_{a,\mu}\Pi_{0,0}+(\mathcal{I}-\Pi_{a,\mu})(\mathcal{I}-\Pi_{0,0})],
\eeq 

\no which are bounded and invertible in $H^s(\T)$ and in $L^2(\T)$. Their inverse is
        
\begin{equation}\label{e:u1}
\mathcal{U}^{-1}_{a,\mu}=[\Pi_{0,0}\Pi_{a,\mu}+(\mathcal{I}-\Pi_{0,0})(\mathcal{I}-\Pi_{a,\mu})](\mathcal{I}-(\Pi_{a,\mu}-\Pi_{0,0})^2)^{-1/2}.
\end{equation}

\item The subspaces $\mathscr{F}_{a,\mu}=\rm{Rg}(\Pi_{a,\mu})$ are isomorphic to each other: $\mathscr{F}_{a,\mu}=\mathcal{U}_{a,\mu}\mathscr{F}_{0,0}$. In particular, $\dim{\mathscr{F}_{a,\mu}}=\dim{\mathscr{F}_{0,0}}=3$.

\item $\operatorname{spec}_1(\mathcal Q_{a,\mu})\subset i \R$.

\end{enumerate}
\end{lemma}

\vspace*{0.1in}

\begin{proof}~
\begin{enumerate}[(a)]
 
\item For $\lambda \in B(0; R/2) \setminus B(0; R/3) $, $ \lambda $ belongs to the resolvent set of $ \mathcal{Q}_{0,0} $, meaning $ \mathcal{Q}_{0,0} - \lambda $ is invertible and  $\|(\lambda-\mathcal{Q}_{0,0})^{-1}\|$ is bounded, say  

\beq
\|(\lambda-\mathcal{Q}_{0,0})^{-1}\|_{L^2(\T)\to H^s(\T)}\leq \nu_0.
\eeq

\no We write $\lambda - \mathcal{Q}_{a,\mu} $ as

\beq\label{def12}
\lambda-\mathcal{Q}_{a,\mu}=(\mathcal{I}-\tilde{\mathcal{Q}}_{a,\mu}(\lambda-\mathcal{Q}_{0,0})^{-1})(\lambda-\mathcal{Q}_{0,0})),
\eeq

\no with $ \tilde{\mathcal{Q}}_{a,\mu} = \mathcal{Q}_{a,\mu} - \mathcal{Q}_{0,0} $. To show that $\lambda-\mathcal{Q}_{a,\mu}$ is invertible, we check that $\mathcal{I} - \tilde{\mathcal{Q}}_{a,\mu} (\lambda - \mathcal{Q}_{0,0})^{-1} $ is invertible. First, we establish some bounds. We have

\begin{align*}
 \tilde{\mathcal{Q}}_{a,\mu}&=(\partial_z+i\mu)\circ\mathcal{T}_{a,\mu}-\partial_z\circ\mathcal{T}_{0,0}\\
 &=\partial_z\circ(\mathcal{T}_{a,\mu}-\mathcal{T}_{0,0})+i\mu\mathcal{T}_{a,\mu},
 \end{align*}

\begin{align*}
\mathcal{T}_{a,\mu}-\mathcal{T}_{0,0}&=(c-e^{-i\mu z}\mathcal{J}_\rho e^{i\mu z}-N\alpha \eta^{N-1})-(c-\mathcal{J}_\rho)\\
&=\mathcal{J}_\rho-e^{-i\mu z}\mathcal{J}_\rho e^{i\mu z}-N\alpha \eta^{N-1}.
\end{align*}

\no For sufficiently small $\mu$, a Taylor expansion of $e^{-i\mu z}\mathcal{J}_\rho e^{i\mu z}$ in terms of $\mu$ can be used, and 

\begin{align}
\|\mathcal{T}_{a,\mu}-\mathcal{T}_{0,0}\|_{H^s(\T)\to H^{s-\sigma}(\T)}\footnotemark
&\leq C_1|\mu|+C_2|a|,\\
\|\tilde{\mathcal{Q}}_{a,\mu}\|_{H^s(\T)\to L^2(\T)}
&\leq \|\partial_z\circ(\mathcal{T}_{a,\mu}-\mathcal{T}_{0,0})\|
   +|\mu|\|\mathcal{T}_{a,\mu}\|.
\end{align}

\footnotetext{We use the notation $\|\cdot\|_{X\to Y}$ to denote the operator norm of a bounded linear map from $X$ to $Y$.}

\no where $C_1$, $C_2$ are constants depending on $\mathcal{J}_\rho$ and $N \alpha \eta^{N-1}$. Note that the operator $\mathcal{T}_{a,\mu} - \mathcal{T}_{0,0}$ is a pseudodifferential operator of order $\sigma$ by assumption (H2), and hence maps $H^s(\mathbb{T})$ into $H^{s - \sigma}(\mathbb{T})$. Composing with the derivative $\partial_z$
yields a pseudodifferential operator of order $\sigma + 1$ that maps $H^s(\mathbb{T})$ into $H^{s - \sigma - 1}(\mathbb{T})$. With $s = \max(1, 1 + \sigma)$, we have $s \geq \sigma + 1$, and this composition is bounded into $L^2(\mathbb{T})$. We have

 \beq
\|\tilde{\mathcal{Q}}_{a,\mu}\|_{H^s(\T)\to L^2(\T)}\leq \nu_1(|a|+|\mu|).
 \eeq

\no Consequently, for $a$ and $\mu$ sufficiently small,  

 \beq\label{bound17}
 \|\tilde{\mathcal{Q}}_{a,\mu}(\lambda-\mathcal{Q}_{0,0})^{-1}\|_{L^2(\T)\to L^2(\T)}\leq \nu_0\nu_1(|a|+|\mu|)\leq 1.
 \eeq
 
\no The convergence of the Neumann series implies the invertibility of $\mathcal{I} - \tilde{\mathcal{Q}}_{a,\mu}(\lambda - \mathcal{Q}_{0,0})^{-1}$, so that $\lambda - \mathcal{Q}_{a,\mu}$ is invertible. Therefore, $ \lambda $ belongs to the resolvent set of $\mathcal{Q}_{a,\mu} $,and the decomposition of the spectrum follows.

\item The spectral projection $\Pi_{a,\mu}$ can be computed using the Dunford integral formula:
 
\beq
\Pi_{a,\mu}=\dfrac{1}{2\pi i}\oint_{\partial B(0;R/3)}(\lambda-\mathcal{Q}_{a,\mu})^{-1}d\lambda.
\eeq
 
\no It follows from (a) that the operator $\Pi_{a,\mu}$ is both well defined and bounded. Further, it is immediate that $\Pi_{a,\mu}$ commutes with $\mathcal{Q}_{a,\mu}$. The projection property $\Pi_{a,\mu}^2=\Pi_{a,\mu}$ follows from classical results in complex integration~\cite{Kato}. Next, using \eqref{def12},

 \begin{align}\nonumber
 (\lambda-\mathcal{Q}_{a,\mu})^{-1}&=(\lambda-\mathcal{Q}_{0,0})^{-1}(\mathcal{I}-\tilde{\mathcal{Q}}_{a,\mu}(\lambda-\mathcal{Q}_{0,0})^{-1})^{-1}\\ & =(\lambda-\mathcal{Q}_{0,0})^{-1}\sum_{n=0}^\infty\left(\tilde{\mathcal{Q}}_{a,\mu}(\lambda-\mathcal{Q}_{0,0})^{-1}\right)^n\nonumber \\&= (\lambda-\mathcal{Q}_{0,0})^{-1}+(\lambda-\mathcal{Q}_{0,0})^{-1}\sum_{n=1}^\infty\left(\tilde{\mathcal{Q}}_{a,\mu}(\lambda-\mathcal{Q}_{0,0})^{-1}\right)^n.
 \end{align}
 
\no Consequently, 

 \beq
 \Pi_{a,\mu}-\Pi_{0,0}=\dfrac{1}{2\pi i}\oint_{\partial B(0;R/3)}(\lambda-\mathcal{Q}_{0,0})^{-1}\sum_{n=1}^\infty\left(\tilde{\mathcal{Q}}_{a,\mu}(\lambda-\mathcal{Q}_{0,0})^{-1}\right)^n d\lambda.
 \eeq
 
\no Bounding the sum for small $a$ and $\mu$,
 
 \begin{align}\nonumber
 \|(\lambda-\mathcal{Q}_{0,0})^{-1}\sum_{n=1}^\infty\left(\tilde{\mathcal{Q}}_{a,\mu}(\lambda-\mathcal{Q}_{0,0})^{-1}\right)^n \|_{L^2(\T)\to H^s(\T)}&\leq \nu_0 \sum_{n=1}^\infty(\nu_0\nu_1(|a|+|\mu|))^n\\
 &\leq 2\nu_0^2\nu_1(|a|+|\mu|).
 \end{align}
 
\no Thus, for any $a$ and $\mu$ sufficiently small, 
 
 \beq\label{bound18}
 \|\Pi_{a,\mu}-\Pi_{0,0}\|_{L^2(\T)\to L^2(\T)}=\mathcal{O}(|a|+|\mu|).
 \eeq

 \item For any operator $\mathcal{L}$ satisfying $\|\mathcal{L}\|<1$, $(\mathcal{I}-\mathcal{L})^{-1/2}$ is well defined, . Therefore, using part (b), the operators $\mathcal{U}_{a,\mu}$ in \eqref{e:u} are well defined in both $H^s(\T)$ and $L^2(\T)$. The invertibility of $\mathcal{U}_{a,\mu}$ is proved in \cite[Chapter $\textsc{ii}$, $\mathsection 4$
]{Kato}, as is \eqref{e:u2}.
 
 \item This follows directly from the conjugation formula \eqref{e:u2}.

 \item To show that $\operatorname{spec}_1(\mathcal Q_{a,\mu})\subset i \R$, consider the restriction $\mathcal{Q}_{a,\mu}^\ddagger$ of $\mathcal{Q}_{a,\mu}$ to the spectral subspace $\mathcal{Y}_{a,\mu}=(\mathcal{I}-\Pi_{a,\mu})L^2(\T)$, so that
 
 \beq
\operatorname{spec}_1(\mathcal{Q}_{a,\mu})=\operatorname{spec}_1(\mathcal{Q}_{a,\mu}^\ddagger).
 \eeq
 
 \no Assume $\lambda$ is an eigenvalue of $\mathcal{Q}_{a,\mu}^\ddagger$ and $\psi_\lambda\neq0$ is its associated eigenvector. The vector $\psi_\lambda$ satisfies the condition
 
 \beq
 \psi_\lambda=(I-\Pi_{a,\mu})\psi_\lambda,
 \eeq
 
\no meaning it lies in the spectral subspace corresponding to $\operatorname{spec}_1(\mathcal{Q}_{a,\mu})$, and thus is orthogonal to the projection $\Pi_{a,\mu}$. Recall that $\mathcal{Q}_{a,\mu}=(\partial_z+i\mu)\mathcal{T}_{a,\mu}$ and 
 
\beq
\mathcal{T}_{a,\mu}=\mathcal{T}_{0,0}+\tilde{\mathcal{T}}_{a,\mu},
\eeq

 \no where, for $a$ and $\mu$ sufficiently small, $\tilde{\mathcal{T}}_{a,\mu}$ is a bounded perturbation with bound proportional to $|a|+|\mu|$:
 
 \beq\label{e:t}
 \|\tilde{\mathcal{T}}_{a,\mu}\|_{H^s(\T)\to H^{s-\sigma}(\T)}\leq \nu_2(|a|+|\mu|). 
 \eeq 
 
 \no From \eqref{bound18}, 
 
 \beq\label{e:pi}
 \|\tilde{\Pi}_{a,\mu}\|_{L^2(\T)\to L^2(\T)}=\|\Pi_{a,\mu}-\Pi_{0,0}\|_{L^2(\T)\to L^2(\T)}\leq \nu_3(|a|+|\mu|).
 \eeq
 
 \no We compute
 
 \begin{align}\nonumber
 \left\langle\mathcal{T}_{a,\mu}\psi_\lambda,\psi_\lambda \right\rangle&=\langle\mathcal{T}_{a,\mu}(I-\Pi_{a,\mu})\psi_\lambda,(I-\Pi_{a,\mu})\psi_\lambda\rangle\\&=\langle\mathcal{T}_{0,0}(I-\Pi_{a,\mu})\psi_\lambda,(I-\Pi_{a,\mu})\psi_\lambda\rangle+\langle\tilde{\mathcal{T}}_{a,\mu}(I-\Pi_{a,\mu})\psi_\lambda,(I-\Pi_{a,\mu})\psi_\lambda\rangle\nonumber\\&=\langle\mathcal{T}_{0,0}(I-\Pi_{0,0})\psi_\lambda,(I-\Pi_{0,0})\psi_\lambda\rangle-\langle\mathcal{T}_{0,0}\tilde{\Pi}_{a,\mu}\psi_\lambda,(I-\Pi_{0,0})\psi_\lambda\rangle\nonumber\\&\quad-\langle\mathcal{T}_{0,0}(I-\Pi_{a,\mu})\psi_\lambda,\tilde{\Pi}_{a,\mu}\psi_\lambda\rangle+\langle\tilde{\mathcal{T}}_{a,\mu}(I-\Pi_{a,\mu})\psi_\lambda,(I-\Pi_{a,\mu})\psi_\lambda\rangle.
 \end{align}
 
 \no Since the spectrum of the restriction of $\mathcal{Q}_{0,0}$ to $(I-\Pi_{0,0})L^2(\T)$ consists of the eigenvalues $
 \jmath(\rho)-\jmath(\rho(n+\mu)),\quad n\in\Z\setminus \{-1,0,1\},
$ using Hypothesis (H3) and taking into account the estimates \eqref{e:t} and \eqref{e:pi}, we obtain that 
 
\beq
\left\langle\mathcal{T}_{a,\mu}\psi_\lambda,\psi_\lambda \right\rangle\neq0.
\eeq

 \no Lemma 4.4 of \cite{Haragus2017TransverseModel}, applicable here, states that if $\Re(\lambda)\neq 0$ then $\left\langle\mathcal{T}_{a,\mu}\psi_\lambda,\psi_\lambda\right\rangle=0$. It follows that $\left\langle\mathcal{T}_{a,\mu}\psi_\lambda,\psi_\lambda\right\rangle\neq 0$ implies that $\Re(\lambda)=0$.

\end{enumerate}

\end{proof}
 
\begin{lemma}\label{l:b2}
For $a$ and $\mu$ sufficiently small, the following statements hold.

\begin{enumerate}
\item The operator $\Pi_{a,\mu}$ is skew-Hamiltonian, i.e., it satisfies the condition

\begin{equation}
\mathcal{X}_\mu\Pi_{a,\mu}=\Pi^\ast_{a,\mu}\mathcal{X}_\mu,\quad \text{for}~\mu\neq0,\quad\mathcal{X}_0\Pi_{a,0}=\Pi^\ast_{a,0}\mathcal{X}_0,\quad\text{in}~L^2_0(\T).
\end{equation}

\no Moreover, it preserves reversibility, 
         
\begin{equation}
\overline{\mathcal{P}}\Pi_{a,\mu}=\Pi_{a,\mu}\overline{\mathcal{P}}.
\end{equation}

\no Additionally, $\Pi_{a,0}[1]=1$ and $\Pi_{a,0}\mathcal{A}_0=\mathcal{A}_0\Pi_{a,0}$, which implies that $\Pi_{a,0}$ leaves $L^2_0(\T)$ invariant.

\item The operators $\mathcal{U}_{a,\mu}$ are symplectic, i.e., they satisfy

\beq
\mathcal{U}^\ast_{a,\mu} \mathcal{X}_\mu \mathcal{U}_{a,\mu} = \mathcal{X}_\mu \quad \text{for}~\mu \neq 0, \quad \mathcal{U}^\ast_{a,0} \mathcal{X}_0 \mathcal{U}_{a,0} = \mathcal{X}_0 \quad \text{in}~L^2_0(\T),
\eeq

\no and reversibility preserving. Moreover, $\mathcal{U}_{a,0}[1] = 1$ and $\mathcal{U}_{a,0} \mathcal{A}_0 = \mathcal{A}_0 (\mathcal{U}^{\ast}_{a,0})^{-1}$, implying that $\mathcal{U}_{a,0}$ also leaves $L^2_0(\T)$ invariant.\\

\item Both $\Pi_{a,0}$ and $\mathcal{U}_{a,0}$ are real operators, i.e., $\overline{\Pi}_{a,0} = \Pi_{a,0}$ and $\overline{\mathcal{U}}_{a,0} = \mathcal{U}_{a,0}$.\\

\item The derivatives of the operators $\Pi_{a,0}$ and $\mathcal{U}_{a,0}$ with respect to $\mu$, evaluated at $\mu = 0$, 

\beq
\dot{\Pi}_{a,0} = \left.\frac{\partial \Pi_{a,\mu}}{\partial \mu}\right|_{\mu=0}, \quad \text{and} \quad \dot{\mathcal{U}}_{a,0} = \left.\frac{\partial \mathcal{U}_{a,\mu}}{\partial \mu}\right|_{\mu=0},
\eeq

\no are imaginary.

\end{enumerate}
\end{lemma}

\begin{proof}
We refer to \cite[Lemma~3.4]{MasperoRadakovic2024} for a detailed proof in a similar setting.  
\end{proof}

\section{Proof of the expansion of the basis $\Sigma$}\label{a:3}

This entire appendix is devoted to the proof of Lemma \ref{l:b1}. The lengthy proof is split up into a sequence of lemmas.

\begin{lemma}\label{l:pd1}
    \begin{align}
    \partial^l_a\Pi_{0,0}&=0, \quad\quad l<N-1,\\
    \partial_a^{N-1}\Pi_{0,0}&=\dfrac{1}{2\pi i}\oint_{\partial B(0; R/3)} (\lambda-\mathcal{Q}_{0,0})^{-1}(\partial^{N-1}_a \mathcal{Q}_{0,0})(\lambda-\mathcal{Q}_{0,0})^{-1} d\lambda.
    \end{align}
\end{lemma}
\begin{proof}

\[
\Pi_{a,\mu} = \frac{1}{2\pi i} \oint_{\partial B(0; R/3)} \mathcal{R}(\lambda)\, d\lambda,
\]
where $ \mathcal{R}_{a,\mu}(\lambda) = (\lambda - \mathcal{Q}_{a,\mu})^{-1} $ denotes the resolvent of $ \mathcal{Q}_{a,\mu} $, and the contour encloses the part of the spectrum onto whose eigenspace we wish to project. To compute the $n$th derivative of $\Pi_{a,\mu}$ with respect to $a$, we use

\beq
\partial_a^n \Pi_{a,\mu} = \frac{1}{2\pi i} \oint_{\partial B(0; R/3)} \frac{\partial^n \mathcal{R}_{a,\mu}(\lambda)}{\partial a^n}\, d\lambda.
\eeq


The $n$th derivative of the resolvent $\mathcal{R}_{a,\mu}(\lambda)$ can be expressed through Fa\`a di Bruno-type expansions involving compositions of operator derivatives, see for instance \cite{cons96}. This requires a minor generalization of the standard multivariate derivatives of composite functions to an operator-valued analogue obtained by iterating the resolvent identity and applying the Leibniz rule.

\beq
 \partial_a^n \mathcal{R}_{a,\mu}(\lambda) = \sum_{(j_1, \dots, j_k) \in \mathcal{C}(n)} \frac{n!}{j_1! \cdots j_k!} \mathcal{R}_{a,\mu} \mathcal{Q}_{a,\mu}^{(j_1)} \mathcal{R}_{a,\mu} \cdots \mathcal{Q}_{a,\mu}^{(j_k)} \mathcal{R}_{a,\mu},
\eeq
where the superindices denote derivatives with respect to $a$, and $\mathcal{C}(n)$ denotes the set of all decompositions of the integer $n$ into $k \geq 1$ strictly positive integers, {i.e.}, $j_1 + \dots + j_k = n$. It follows that

\beq\label{pia}
\partial_a^n \Pi_{a,\mu} = \frac{1}{2\pi i} \oint_{\partial B(0; R/3)} \sum_{(j_1, \dots, j_k) \in \mathcal{C}(n)} \frac{n!}{j_1! \cdots j_k!} \mathcal{R}_{a,\mu} \mathcal{Q}_{a,\mu}^{(j_1)} \mathcal{R}_{a,\mu} \cdots \mathcal{Q}_{a,\mu}^{(j_k)} \mathcal{R}_{a,\mu} \, d\lambda.
\eeq

\no Note that for $ \ell < N-1 $, $ \partial_a^\ell \mathcal{Q}_{0,0} = 0 $. Consequently, for any $ 
\ell < N-1 $,
\[
\partial_a^{\ell} \Pi_{0,0} = 0, 
\]

\no where here and below, $\partial_a^\ell \Pi_{0,0}=\partial_a^\ell \Pi_{a,\mu}|_{(a,\mu)=(0,0)}$.

Next, for $n=N-1$, in the Faà di Bruno-type expansion, an $N-1$-th order derivative can only appear if $k=1$. All decompositions with $k>1$ have terms that are identically zero. 
Then

\beq
\partial_a^{N-1} \Pi_{0,0} = \frac{1}{2\pi i} \oint_{\partial B(0; R/3)} (\lambda - \mathcal{Q}_{0,0})^{-1} \, \partial_a^{N-1} \mathcal{Q}_{0,0} \, (\lambda - \mathcal{Q}_{0,0})^{-1} \, d\lambda.
\eeq

\end{proof}

\begin{lemma}\label{l:ud2}
   \begin{align}
    \partial^l_a\mathcal{U}_{0,0}&=0, \quad\quad l<N-1,\\
    (\partial^{N-1}_a \mathcal{U}_{0,0})\Pi_{0,0}&=(\partial^{N-1}_a \Pi_{0,0})\Pi_{0,0}.
    \end{align}
\end{lemma}

\begin{proof}
   Recall from \eqref{e:u},
   
\begin{equation}
\mathcal{U}_{a,\mu} = \left( \mathcal{I} - (\Pi_{a,\mu} - \Pi_{0,0})^2 \right)^{-1/2} 
\left[ \Pi_{a,\mu} \Pi_{0,0} + (\mathcal{I} - \Pi_{a,\mu})(\mathcal{I} - \Pi_{0,0}) \right].
\end{equation}

\no Let $ \delta = \Pi_{a,\mu} - \Pi_{0,0} $, and write $ \mathcal{U}_{a,\mu} = \mathcal{U}^\mathrm{I}_{a,\mu} \mathcal{U}^\mathrm{II}_{a,\mu} $, where

\begin{equation}
\mathcal{U}^\mathrm{I}_{a,\mu} = (\mathcal{I} - \delta^2)^{-1/2}, \qquad 
\mathcal{U}^\mathrm{II}_{a,\mu} = \Pi_{a,\mu} \Pi_{0,0} + (\mathcal{I} - \Pi_{a,\mu})(\mathcal{I} - \Pi_{0,0}).
\end{equation}

By Lemma \ref{l:pd1}, we have $\partial_a^l\Pi_{0,0}=0$, for $l<N-1$. Hence, 

\[
\delta=\mathcal{O}(a^{N-1}), \quad\text{as}\quad a\to 0,
\]

\no and therefore

\[
\delta^2=\mathcal{O}(a^{2N-2}), \quad\text{as}\quad a\to 0.
\]

\no First, we write $\mathcal{U}^\mathrm{II}_{a,\mu}$ in terms of $\delta$. A direct calculation gives

\begin{align}\label{eq:secondone}
    \mathcal{U}^\mathrm{II}_{a,\mu}&=\Pi_{a,\mu}\Pi_{0,0}+(\mathcal{I}-\Pi_{a,\mu})(\mathcal{I}-\Pi_{0,0})\\
    &=2\Pi_{a,\mu}\Pi_{0,0}+\mathcal{I}-\Pi_{a,\mu}-\Pi_{0,0}\nonumber\\
    &=\mathcal{I}+\delta(2\Pi_{0,0}-\mathcal{I}).\nonumber
\end{align}

Next using the binomial series around $\delta^2=0$,

\begin{equation}\label{eq:binomial}
\mathcal{U}^\mathrm{I}_{a,\mu}=\mathcal{I}+\sum_{k=1}^\infty c_k(\delta^2)^k,
\end{equation}

\no with $c_k=(-1)^k \binom{k}{-1/2}$. In particular 

\begin{equation}
\mathcal{U}^\mathrm{I}_{a,\mu}=\mathcal{I}+\mathcal{O}(\delta^2)=\mathcal{I}+\mathcal{O}(a^{2N-2}).
\end{equation}

\no Therefore,

\begin{equation}
\mathcal{U}_{a,\mu}=\mathcal{U}^\mathrm{I}_{a,\mu}\mathcal{U}^\mathrm{II}_{a,\mu}=(\mathcal{I}+\mathcal{O}(\delta^2))(\mathcal{I}+\delta(2\Pi_{0,0}-\mathcal{I}))=\mathcal{I}+\delta(2\Pi_{0,0}-\mathcal{I})+\mathcal{O}(\delta^2).
\end{equation}

\no With $\delta=\mathcal{O}(a^{N-1})$, this shows that every coefficient of $a^l$, $l<N-1$, in the Taylor series of $\mathcal{U}_{a,\mu}$ is zero. Concretely, 

\beq
\partial_a^l\mathcal{U}_{0,0}=0,~~~~l<N-1,
\eeq

\no which proves the first part of the lemma. 

We move on to the second part. 
Multiplying $\mathcal{U}_{a,\mu}$ by $\Pi_{0,0}$,

\beq
\mathcal{U}_{a,\mu}\Pi_{0,0}=\mathcal{U}^\mathrm{I}_{a,\mu}\mathcal{U}^\mathrm{II}_{a,\mu}\Pi_{0,0}.
\eeq

\no With $(2\Pi_{0,0}-\mathcal{I})\Pi_{0,0}=\Pi_{0,0}$ (since $\Pi^2_{0,0}=\Pi_{0,0}$), 

\begin{align}
\mathcal{U}^\mathrm{II}_{a,\mu}\Pi_{0,0}=(\mathcal{I}+\delta(2\Pi_{0,0}-\mathcal{I}))\Pi_{0,0}=\Pi_{0,0}+\delta\Pi_{0,0}.
\end{align}
Thus 

\beq
\mathcal{U}_{a,\mu}\Pi_{0,0}=\mathcal{U}^\mathrm{I}_{a,\mu}(\Pi_{0,0}+\delta\Pi_{0,0})=\Pi_{0,0}+\delta\Pi_{0,0}+\mathcal{O}(\delta^2).
\eeq

\no Since $\delta^2=\mathcal{O}(a^{2N-2})$, the $(N-1)$-th derivative at $a=0$ of the $\mathcal{O}(\delta^2)$ term vanishes. Therefore

\beq
\partial^{N-1}_a(\mathcal{U}_{a,\mu}\Pi_{0,0})\big|_{a=0}=\partial^{N-1}_a(\delta\Pi_{0,0})|_{a=0}=(\partial^{N-1}_a \delta)|_{a=0}\Pi_{0,0}.
\eeq

\no Hence, we have

\beq
(\partial^{N-1}_a \mathcal{U}_{0,0})\Pi_{0,0}=(\partial^{N-1}_a \Pi_{0,0})\Pi_{0,0}, 
\eeq

\no which proves the second part of the lemma.
\end{proof}

\begin{lemma}\label{l:l1}
    For $\phi\in\{\psi_1^+(0,\mu)=\cos z,\psi_1^-(0,\mu)=\sin z,\psi_0(0,\mu)=1/\sqrt{2}\}$, 
    
    \beq
    (\partial_\mu^{N-1}\Pi_{0,0})\phi = 0, \quad N\geq2.
    \eeq

\end{lemma}

\begin{proof}
In order to compute the action of $(\mathcal{Q}_{0,0}-\lambda)^{-1}\partial_\mu^{N-1}\mathcal{Q}_{0,0}$ on $\mathscr{F}_{0,\mu}$, it is essential to first determine the action of  $\partial_\mu^{N-1}\mathcal{Q}_{0,0}$ on $\mathscr{F}_{0,\mu}$. Since $\phi$ is $\mu$-independent, 

\beq
(\partial_\mu^{N-1}\Pi_{0,0})\phi=\partial_\mu^{N-1}(\Pi_{0,0}\phi).
\eeq

To differentiate $\mathcal{Q}_{a,\mu}$ with respect to $\mu$, we use the Leibniz rule for operator-valued functions:

\beq
\partial_\mu^{N-1}\mathcal{Q}_{a,\mu}=\partial_\mu^{N-1}(\partial_z+i\mu)\mathcal{T}_{a,\mu}=\sum_{k=0}^{N-1}\binom{N-1}{k}(\partial_\mu^{N-1-k}(\partial_z+i\mu))(\partial_\mu^k\mathcal{T}_{a,\mu}).
\eeq

\no All terms but those with $N-1-k=0$ and $N-1-k=1$ vanish, so that

\beq
\partial_\mu^{N-1}\mathcal{Q}_{a,\mu}=(\partial_z+i\mu)\partial_\mu^{N-1}\mathcal{T}_{a,\mu}+i(N-1)\partial_\mu^{N-2}\mathcal{T}_{a,\mu}.
\eeq

\no First, we equate $a=0$ (since $\mathcal{T}_{a,\mu}$ depends analytically on $a$ and $\mu$, this does not impact the differentiation with respect to $\mu$) and apply the above to $e^{inz}$. The result is evaluated at $\mu=0$, giving

\begin{align}
\partial_\mu^{N-1}\mathcal{Q}_{0,0}e^{inz}=\begin{cases}(-in\rho j'(\rho n)+i(j(\rho)-j(\rho n))e^{inz},\quad &\text{if}\quad N=2,\\
-(in\rho^{N-1}j^{(N-1)}(\rho n)+i(N-1)\rho^{N-2}j^{(N-2)}(\rho n))e^{inz},\quad &\text{if}\quad N>2.
\end{cases}
\end{align}

Using the explicit eigenbasis
\[
\psi_1^\pm(0,\mu)=\frac{e^{iz}\pm e^{-iz}}{2}, 
\qquad 
\psi_0(0,\mu)=\frac{1}{\sqrt{2}},
\]

\no we obtain the following formulas for the action of
$\partial_\mu^{N-1}\mathcal{Q}_{0,0}$:

\begin{align}
\partial^{N-1}_\mu\mathcal{Q}_{0,0}\psi_1^+&=
\begin{cases}
-i\rho j'(\rho)\psi_1^+,\quad &\text{if}\quad N=2,\\
-i((N-1)\rho^{N-2}\jmath^{(N-2)}(\rho)+\rho^{N-1}\jmath^{(N-1)}(\rho))\psi_1^+,\quad &\text{if}\quad N>2, \mbox{even},\\((N-1)\rho^{N-2}\jmath^{(N-2)}(\rho)+\rho^{N-1}\jmath^{(N-1)}(\rho))\psi_1^-,\quad &\text{if}\quad N>2, \mbox{odd},
\end{cases}\label{e:1muq}\\
\partial^{N-1}_\mu\mathcal{Q}_{0,0}\psi_1^-&=
\begin{cases}
-i\rho j'(\rho)\psi_1^-,\quad &\text{if}\quad N=2,\\
-i((N-1)\rho^{N-2}\jmath^{(N-2)}(\rho)+\rho^{N-1}\jmath^{(N-1)}(\rho))\psi_1^-,\quad &\text{if}\quad N>2, \mbox{even},\\
-((N-1)\rho^{N-2}\jmath^{(N-2)}(\rho)+\rho^{N-1}\jmath^{(N-1)}(\rho))\psi_1^+,\quad &\text{if}\quad N>2, \mbox{odd},
\end{cases}\label{e:2muq}\\
\partial^{N-1}_\mu\mathcal{Q}_{0,0}\psi_0&=
\begin{cases}
i\dfrac{\jmath(\rho)-1}{\sqrt{2}},\quad &\text{if}\quad N=2,\\-i\dfrac{(N-1)\rho^{N-2}j^{(N-2)}(0)}{\sqrt{2}},\quad &\text{if}\quad N>2.\label{e:3muq}
\end{cases}
\end{align}

Similar to \eqref{pia}, the $n$th-derivative of $\Pi_{a,\mu}$ with respect to $\mu$ can be written as

\beq\label{pimu}
\partial_\mu^n \Pi_{a,\mu} = \frac{1}{2\pi i} \oint_{\partial B(0; R/3)} \sum_{k=1}^n
\sum_{\substack{j_1+\dots+j_k=n\\ j_i\ge1}} \frac{n!}{j_1! \cdots j_k!} \mathcal{R}_{a,\mu} \mathcal{Q}_{a,\mu}^{(j_1)} \mathcal{R}_{a,\mu} \cdots \mathcal{Q}_{a,\mu}^{(j_k)} \mathcal{R}_{a,\mu} \, d\lambda,
\eeq 

\no where the superindices denote derivatives with respect to $\mu$. For $(a,\mu)=(0,0)$, the operator $\mathcal Q_{0,0}$ vanishes on $\mathscr F_{0,0}$ and for every $\lambda$ on the integration contour we have the simple resolvent action

\beq
\mathcal R_{a,\mu}(\lambda)\phi=(\lambda-\mathcal Q_{0,0})^{-1}\phi=\frac{1}{\lambda}\phi.
\eeq

\no Using this identity repeatedly, 




\beq
\mathcal R_{a,\mu}\,\mathcal Q^{(j_1)}_{0,0}\mathcal R_{a,\mu}\cdots\mathcal Q^{(j_k)}_{0,0}\mathcal R_{a,\mu}\,\phi
= \frac{1}{\lambda^{\,k+1}}\;\bigl(\mathcal Q^{(j_1)}_{0,0}\cdots\mathcal Q^{(j_k)}_{0,0}\phi\bigr).
\eeq

\no Therefore, for $(a,\mu)=(0,0)$, 

\beq
\bigl(\partial_\mu^n\Pi_{0,0}\bigr)\phi
=\frac{1}{2\pi i}\sum_{k=1}^n
\sum_{\substack{j_1+\dots+j_k=n\\ j_i\ge1}}
\frac{n!}{j_1!\cdots j_k!}\;
\left(\oint_{\partial B(0;R/3)} \frac{d\lambda}{\lambda^{\,k+1}}\right)
\; \bigl(\mathcal Q^{(j_1)}_{0,0}\cdots\mathcal Q^{(j_k)}_{0,0}\phi\bigr).
\eeq

\no Since $k\geq 1$ in the Faà di Bruno sum, the lemma follows.





\end{proof}

\begin{lemma}\label{l:ud3}
    For $\phi\in\{\psi_1^+(0,\mu)=\cos z,\psi_1^-(0,\mu)=\sin z,\psi_0(0,\mu)=1/\sqrt{2}\}$, 
    
    \beq
    \partial_\mu^{N-1}\mathcal U_{0,0}\phi = 0, \quad \text{for all}\quad N\geq2.
    \eeq
\end{lemma}

\begin{proof}

Recall 

\beq
\delta_{a,\mu}=\Pi_{a,\mu}-\Pi_{0,0}, \qquad \mathcal{U}_{a,\mu}^{\mathrm{I}}=\mathcal I + \sum_{k=1}^\infty c_k \,\delta_{a,\mu}^{2k}, \qquad \mathcal U^{\mathrm{II}}_{a,\mu} = \mathcal I + \delta_{a,\mu}(2\Pi_{0,0}-\mathcal I), 
\eeq

\no see \eqref{eq:binomial} and \eqref{eq:secondone}. 
We show by direct differentiation that

\beq
\partial_\mu^m(\delta_{a,\mu})^p \phi\Big|_{a=0,\mu=0}=0,
\qquad p\geq 1, m\geq 0.
\eeq

\no From this, the lemma follows immediately because all terms in the series for $\mathcal U^{\mathrm I}_{0,\mu}\phi$ (and the single correction term in $\mathcal U^{\mathrm{II}}_{0,\mu}\phi$) are of the form $(\delta_{a,\mu})^{p}\phi$ or a finite linear combination of such terms, hence all derivatives at $0$ vanish.

With $a=0$, we fix $p\geq1$ and compute the $m$-th derivative of the product $(\delta_{0,\mu})^p\phi$, recalling that $\phi$ does not depend on $\mu$. 
Applying the Leibniz rule (or the operator form of Faà di Bruno), the $m$-th derivative of the product is a finite sum of terms of the form

\beq
(\partial_\mu^{r_1}\delta_{0,\mu})\,
(\partial_\mu^{r_2}\delta_{0,\mu})\cdots
(\partial_\mu^{r_p}\delta_{0,\mu})\,\phi,
\eeq

\no where $r_1+r_2+\cdots+r_p = m$.  
Evaluating each factor at $\mu=0$ gives

\beq
\partial_\mu^{r_j}\delta_{0,0} = \partial_\mu^{r_j}\bigl(\Pi_{0,\mu}-\Pi_{0,0}\bigr)\big|_{\mu=0}
= \partial_\mu^{r_j}\Pi_{0,0}.
\eeq

\no Using Lemma \ref{l:l1}, $\partial_\mu^{r_j}\Pi_{0,0}\,\phi=0$ for \(r_j\geq1\). Consequently, in the product above
for each index $j$, the factor $\partial_\mu^{r_j}\delta_{0,0}$ annihilates every function in the range of $\Pi_{0,0}$. In particular, it annihilates the function obtained by the action of the factors to the right. It follows that the entire chain applied to $\phi$ is zero. Because this holds for every choice of the partition $r_1+\cdots+r_p=m$, every summand in the Leibniz expansion vanishes at $\mu=0$. Therefore,

\beq
\partial_\mu^m((\delta_{0,\mu})^p\phi)\Big|_{\mu=0}=0,
\eeq

\no as claimed.

\end{proof}

Recall \eqref{e:teig}:

 \beq\label{eq:teig2}
\psi_j^\sigma(a,\mu)=\mathcal U_{a,\mu}\psi_j^\sigma(0,\mu).
 \eeq
 
 Since $\mathcal U_{a,\mu}$ is analytic in $(a,\mu)$ near $(0,0)$, we may expand $\mathcal U_{a,\mu}$ in a multivariate Taylor series about $(0,0)$. Using multi-index notation $\beta=(\beta_1,\beta_2)$ with $\beta_1$ the exponent for powers of $a$ and $\beta_2$ for powers of $\mu$, with $|\beta|=\beta_1+\beta_2$, we have the formal expansion

 \beq
\mathcal U_{a,\mu}
= \mathcal I + \sum_{1\le |\beta|\le M-1}\frac{a^{\beta_1}\mu^{\beta_2}}{\beta!}\,\partial^\beta\mathcal U_{0,0}
\;+\; \mathcal R_M(a,\mu),
 \eeq

\no where $\partial^\beta=\partial_a^{\beta_1}\partial_\mu^{\beta_2}$, $\beta!=\beta_1!\,\beta_2!$, and $\mathcal R_M$ is the remainder of order $M$. Substitution in \eqref{eq:teig2} gives

\beq\label{e:mv}
\psi_j^\sigma(a,\mu)=\psi_j^\sigma(0,\mu)+ \sum_{1\le |\beta|\le M-1}\frac{a^{\beta_1}\mu^{\beta_2}}{\beta!}\,(\partial^\beta\mathcal U_{0,0})\psi_j^\sigma(0,\mu)\;+\; \mathcal R_M(a,\mu)\psi_j^\sigma(0,\mu).
\eeq

From Lemma \ref{l:pd1}, \ref{l:ud2}, \ref{l:l1}, \ref{l:ud3}, every $a$ or $\mu$ derivative of order $<N-1$ vanishes, and more generally every mixed derivative $\partial^\beta\mathcal U_{0,0}$ with $|\beta|<N-1$ also vanishes (because it contains either $\partial_a^{r}$ or $\partial_\mu^{s}$ with $r<N-1$ or $s<N-1$ and those partials are zero). Hence all terms in the sum \eqref{e:mv} with $|\beta|<N-1$ vanish. With $M=N$, we obtain the expansion truncated at the first possibly nonzero order:  

\begin{align}\nonumber
\psi_j^\sigma(a,\mu)=&\psi_j^\sigma(0,\mu)+\frac{a^{\,N-1}}{(N-1)!}\bigl(\partial_a^{\,N-1}\mathcal U_{0,0}\bigr)\psi_j^\sigma(0,\mu)
+ \frac{\mu^{\,N-1}}{(N-1)!}\bigl(\partial_\mu^{\,N-1}\mathcal U_{0,0}\bigr)\psi_j^\sigma(0,\mu)\\\nonumber&+\sum_{|\beta|\ge N}\frac{a^{\beta_1}\mu^{\beta_2}}{\beta!}\bigl(\partial^\beta\mathcal U_{0,0}\bigr)\psi_j^\sigma(0,\mu)
+ \mathcal R_N(a,\mu)\psi_j^\sigma(0,\mu)\\\label{e:psiop1}
=&\psi_j^\sigma(0,\mu)+\frac{a^{\,N-1}}{(N-1)!}\bigl(\partial_a^{\,N-1}\mathcal U_{0,0}\bigr)\psi_j^\sigma(0,\mu)
+\frac{a^{\,N-1}\mu}{(N-1)!}\bigl(\partial_a^{\,N-1}\partial_\mu\mathcal U_{0,0}\bigr)\psi_j^\sigma(0,\mu) \nonumber\\&
+ \mathcal \mathcal{O}((|a|+|\mu|)^N)\psi_j^\sigma(0,\mu), 
\end{align}

\no where Lemma \ref{l:ud3} was used for the last equality. Since $\psi_j^\sigma(0,\mu)$ lies in the range of $\Pi_{0,0}$, we have $\psi_j^\sigma(0,\mu)=\Pi_{0,0}\psi_j^\sigma(0,\mu)$. Using Lemma \ref{l:ud2}, expansion \eqref{e:psiop1} becomes

\begin{align}\label{e:psiop2}
\psi_j^\sigma(a,\mu)=&\psi_j^\sigma(0,\mu)+\frac{a^{\,N-1}}{(N-1)!}\bigl(\partial_a^{\,N-1}\Pi_{0,0}\bigr)\psi_j^\sigma(0,\mu)
+ \mathcal{O}(|a|^{N-1}(1+|\mu|))\psi_i^\sigma(0,\mu).
\end{align}
 


In order to compute $(\partial_a^{N-1} \Pi_{0,0})\psi_n^\sigma(0,\mu)$ , it is useful to know the action of $(\mathcal{Q}_{0,0}-\lambda)^{-1}$ on 

\beq
\psi_n^+=\cos{(n z)},\quad \psi_n^-=\sin{(n z)},\quad n\in \N.
\eeq

\begin{lemma}\label{l:d5}
    The space $H^s(\T)$ decomposes as $H^s(\T)=\mathscr{F}_{0,0}\oplus\mathcal{M}_{H^s}$ with $\mathcal{M}_{H^s}=\overline{\oplus_{n=2}^\infty\mathcal{M}_n}$ where the subspaces $\mathscr{F}_{0,0}$ and $\mathcal{M}_n$, defined below, are invariant under $\mathcal{Q}_{0,0}$ and the following properties hold:
    
    \begin{enumerate}[(a)]
        \item $\mathscr{F}_{0,0}=\operatorname{span}\{\psi_1^+,\psi_1^-,\psi_0\}$ is the kernel of $\mathcal{Q}_{0,0}$. For any $\lambda\neq0$, the operator $\mathcal{Q}_{0,0}-\lambda:\mathscr{F}_{0,0}\to \mathscr{F}_{0,0}$ is invertible and
        
        \beq
        (\lambda-\mathcal{Q}_{0,0})^{-1}\phi=\dfrac{1}{\lambda}\phi,\qquad \phi\in \{\psi_1^+=\cos z, \psi_1^-=\sin z, \psi_0=1/\sqrt{2})\}.
        \eeq
        
        \item Each subspace $\mathcal{M}_n=\operatorname{span}\{\psi_n^+,\psi_n^-\}$ is invariant under $\mathcal{Q}_{0,0}$. Let $\mathcal{M}_{L^2}=\overline{\oplus_{n=2}^\infty\mathcal{M}_n}$. For any $|\lambda|$ sufficiently small, the operator $\mathcal{Q}_{0,0}-\lambda:\mathcal{M}_{H^s}\to\mathcal{M}_{L^2}$ is invertible. In particular, with $\sigma\in \{+,-\}$, 
        \beq\label{e:nq}
        (\lambda-\mathcal{Q}_{0,0})^{-1}\psi_n^\sigma=\dfrac{\lambda \psi_n^{\sigma}-\sigma n(\jmath(\rho)-\jmath(\rho n))\psi_n^{-\sigma}}{\lambda^2+(n(\jmath(\rho)-\jmath(\rho n)))^2}+\lambda^2\varphi_n^\sigma(\lambda,z),
        \eeq

        \no where $\lambda \mapsto \varphi_n^\sigma(\lambda,.)$ are analytic functions.
    \end{enumerate}
\end{lemma}

\begin{proof}
    The proof follows the approach outlined in \cite[Lemma~A.2]{Berti2021FullWater}.
\end{proof}

\begin{lemma}

\begin{align}
    \partial_a^{N-1}\Pi_{0,0}\psi_1^+(0,\mu)&=\begin{cases} N! \mathcal{D}_\rho^{-1}\left(\cos^N{z}-\dfrac{1}{2}\langle\cos^N{z},1\rangle\right),& N ~\text{even},\\
\alpha N!\mathcal{D}_\rho^{-1}\left(\cos^N{z}-\langle\cos^N{z},\cos{z}\rangle\right),& N~\text{odd}.
    \end{cases}\\
\partial_a^{N-1}\Pi_{0,0}\psi_1^-(0,\mu)&=\begin{cases} N!\mathcal{D}_\rho^{-1}\left(\cos^{N-1}{z}\sin{z}\right), & N~\text{even},\\
\alpha N!\mathcal{D}_\rho^{-1}\left(\cos^{N-1}{z}\sin{z}-\langle\cos^{N-1}{z}\sin{z},\sin{z}\rangle\right), & N~\text{odd}.
\end{cases}
\\
\partial_a^{N-1}\Pi_{0,0}\psi_0(0,\mu)&=\begin{cases} N!\mathcal{D}_\rho^{-1}\left(\cos^{N-1}{z}-\langle\cos^{N-1}{z},\cos{z}\rangle\right),& N~\text{even},\\
\alpha N!\mathcal{D}_\rho^{-1}\left(\cos^{N-1}{z}-\dfrac{1}{2}\langle\cos^{N-1}{z},1\rangle\right),& N~\text{odd}.
    \end{cases}
\end{align}
\end{lemma}

\begin{proof}
To compute the action of $\partial_a^{N-1}\Pi_{0,0}$ on $\mathscr F_{0,0}$, we need to determine how $\partial^{N-1}_a \mathcal{Q}_{0,0}$ acts on $\mathscr F_{0,0}$. Using Theorem~\ref{t:1}, 

\begin{equation}\label{e:expand311}
 \partial^{N-1}_a \mathcal{Q}_{0,0}=
\begin{cases}
- N!\partial_z\cos^{N-1}{z},\quad &N~\text{even},\\
\alpha(N-1)!\partial_z\left(\langle\cos^N{z} ,\cos{z}\rangle-N \cos^{N-1}{z}\right),&N~\text{odd}.
\end{cases}
\end{equation}

Using De Moivre's Theorem, $\cos^N z$ can be written as

\begin{align}\label{excos}
\cos^N z&=\dfrac{1}{2^N}\sum_{k=0}^N\binom{N}{k}\cos{(N-2k)z}\\\nonumber
&=\begin{cases}
   \displaystyle \frac{1}{2^N}\binom{N}{N/2}+\frac{1}{2^{N-1}}\sum_{k=0}^{N/2-1} \binom{N}{k}\cos(N-2k)z, &N~\text{even},\\
   \displaystyle \frac{1}{2^{N-1}}\sum_{k=0}^{(N-1)/2} \binom{N}{k}\cos(N-2k)z, &N~\text{odd}.
\end{cases}
\end{align}

\no Note that $\langle\cos^N{z} ,\cos{z}\rangle$ is the Fourier coefficient of $\cos{z}$ in the expansion of $\cos^Nz$, occurring when $k=(N-1)/2$ in the above series. 

Using \eqref{excos} in \eqref{e:expand311}, 

\begin{align}
    \partial_a^{N-1}\mathcal{Q}_{0,0}\psi_1^+(0,\mu)=\begin{cases}
- N!\partial_z(\cos^Nz),\quad &\quad N~\text{even},\\
\alpha(N-1)!\left(-\dfrac{1}{2^{N-1}}\binom{N}{(N-1)/2}\sin{z}-N \partial_z(\cos^Nz)\right),\quad &\quad N~\text{odd},
\end{cases}
\end{align}

\begin{align}
    \partial_a^{N-1}\mathcal{Q}_{0,0}\psi_1^-(0,\mu)=\begin{cases}
- N!\partial_z(\cos^{N-1}z\sin{z}),&N~\text{even},\\
\alpha(N-1)!\left(\dfrac{1}{2^{N-1}}\binom{N}{(N-1)/2}\cos{z}-N \partial_z(\cos^{N-1}z\sin{z})\right),\!\!\!\!\!&N~\text{odd},
\end{cases}
\end{align}

\begin{align}
    \partial_a^{N-1}\mathcal{Q}_{0,0}\psi_0(0,\mu)\!=\!\begin{cases}
-\dfrac{1}{\sqrt{2}} N!\partial_z\left(\cos^{N-1}z\right), &N~\text{even},\\
\dfrac{1}{\sqrt{2}}\alpha(N-1)!\left(\dfrac{1}{2^{N-1}}\binom{N}{(N-1)/2}\cos{z}-N \partial_z(\cos^{N-1}z\sin{z})\right),\!\!\!&N~\text{odd}.
\end{cases}
\end{align}

\no where, using \eqref{excos}, 

\beq
\partial_z(\cos^Nz)=-\dfrac{1}{2^N}\sum_{k=0}^N\binom{N}{k}(N-2k)\sin{(N-2k)z}, 
\eeq

\no and 
 
\begin{align}\nonumber
\partial_z(\cos^{N-1}z\sin{z})=&\dfrac{1}{2^N}\sum_{k=0}^{N-1}\binom{N-1}{k}\left((N-2k)\cos{(N-2k)z}-\right.\\
&\left.(N-2k-2)\cos{(N-2k-2)z}\right).
\end{align}

From Lemma \ref{l:pd1}, 

\begin{align}
\partial_a^{N-1}\Pi_{0,0}\psi_1^+(0,\mu)&=\dfrac{1}{2\pi i}\oint_{\partial B(0; R/3)}(\lambda-\mathcal{Q}_{0,0})^{-1}(\partial_a^{N-1}\mathcal{Q}_{0,0})(\lambda-\mathcal{Q}_{0,0})^{-1}\psi_1^+d\lambda\nonumber\\&=\dfrac{1}{2\pi i}\oint_{\partial B(0; R/3)}\dfrac{1}{\lambda}(\lambda-\mathcal{Q}_{0,0})^{-1}(\partial_a^{N-1}\mathcal{Q}_{0,0})\psi_1^+d\lambda,
\end{align}

\no using Lemma~\ref{l:d5} in the last step. 

For even $N$,

\begin{align}
    \partial_a^{N-1}\Pi_{0,0}\psi_1^+(0,\mu)=\dfrac{ N!}{2^N}\dfrac{1}{2\pi i}\oint_{\partial B(0; R/3)}\dfrac{1}{\lambda}(\lambda-\mathcal{Q}_{0,0})^{-1}\sum_{k=0}^N\binom{N}{k}(N-2k)\sin{(N-2k)z}d\lambda.
\end{align}

Using the results above and, again, Lemma \ref{l:d5}, 

\begin{align}\label{e:oppi}
    \partial_a^{N-1}\Pi_{0,0}\psi_1^+(0,\mu)=\dfrac{ N!}{2^N}\dfrac{1}{2\pi i}\oint_{\partial B(0; R/3)}\dfrac{1}{\lambda}A(\lambda,z)d\lambda,
\end{align}

\no where

\begin{align}
A(\lambda,z)&=\sum_{k=0}^N\binom{N}{k}(N-2k)\dfrac{\lambda \psi_{N-2k}^-+(N-2k)(j(\rho)-j(\rho(N-2k))\psi_{N-2k}^+}{\lambda^2+((N-2k)(j(\rho)-j(\rho(N-2k))))^2}.
\end{align}

\no In \eqref{e:oppi}, the term with $N-2k=0$ vanishes. The residue calculation gives 


%

\begin{align}
    \partial_a^{N-1}\Pi_{0,0}\psi_1^+(0,\mu)=\dfrac{ N!}{2^N}\left(\sum_{k=0}^N\binom{N}{k}\dfrac{1}{j(\rho)-j(\rho(N-2k))}\cos{((N-2k)z)}-\binom{N}{N/2}\dfrac{1}{j(\rho)-1}\right),
\end{align}

\no where we have separated the $k=N/2$ term. Using \eqref{doperation}, this is re-written as

\begin{align}
    \partial_a^{N-1}\Pi_{0,0}\psi_1^+(0,\mu)= N!\mathcal{D}_\rho^{-1}\left(\cos^Nz-\dfrac{1}{2}\langle\cos^Nz,1\rangle\right).
\end{align}

For odd $N$, and for the action on $\psi_1^-(0,\mu)$ and $\psi_0(0,\mu)$, the calculations are similar. 














\end{proof}
This completes the proof of Lemma \ref{l:b1}.

\section{Proof of Proposition \ref{p:1}}\label{a:4}

For every value of $ \mu$, the basis $\{\psi_1^\pm(0,\mu), \psi_0(0,\mu)\} $ is both $ \mu $-symplectic and reversible. According to Lemma~\ref{l:b2}, the Kato basis $ \{\psi_1^\pm(a,\mu), \psi_0(a,\mu)\} $ retains the $ \mu $-symplectic property for all pairs $(a,\mu)$ in 
neighborhoods of $ a_0 $ and $ \mu_0 $. The maps $(a,\mu)\mapsto \psi_1^\pm(a,\mu)$ and $(a,\mu)\mapsto \psi_0(a,\mu)$ are smooth and continuously differentiable. The matrix $ T_{a,\mu} = \big(\mathcal{T}_{a,\mu}\psi^\sigma_n(a,\mu), \psi^{\sigma^\prime}_{n'}(a,\mu)\big) $ is sufficiently smooth, with regularity of class $ C^r $, where $ r \geq 3 $. From \eqref{eq:tdef}, we have

\beq\label{eq:tdef1}
\mathcal{T}_{a,\mu}=c-e^{-i\mu z}\mathcal{J}_\rho e^{i\mu z}-N\alpha \eta^{N-1}.
\eeq

Using \eqref{E:w_ansatz} together with \eqref{e:c}, this expression becomes

\beq
\mathcal{T}_{a,\mu}=\jmath(\rho)+a^\tau c_\tau+\mathcal{O}(a^{\tau+1})-e^{-i\mu z}\mathcal{J}_\rho e^{i\mu z}-N\alpha(a\cos{z}+a^N\eta_N+\mathcal{O}(a^{N+1}))^{N-1}.
\eeq

\no We decompose $\mathcal{T}_{a,\mu}$ into a $\mu$-dependent part and an $a$-dependent part as

\begin{equation}\label{e:dc}
\mathcal{T}_{a,\mu}
= \tilde{\mathcal{T}}_{\mu} + \check{\mathcal{T}}_{a},
\end{equation}

\no where

\begin{align}
\tilde{\mathcal{T}}_{\mu}
&= \jmath(\rho) - e^{-i\mu z}\mathcal{J}_\rho e^{i\mu z}, \\\label{eqta2}
\check{\mathcal{T}}_a
&= c(a)
 - N\alpha \eta^{N-1}.
\end{align}

\no To find the matrix representation, we use 

 \beq
\tilde{\mathcal{T}}_{\mu}(\cos{nz})=\tilde{\mathcal{T}}_{\mu}\left(\dfrac{e^{inz}+e^{-inz}}{2}\right)\quad\text{and}\quad\tilde{\mathcal{T}}_{\mu}(\sin{nz})=\tilde{\mathcal{T}}_{\mu}\left(\dfrac{e^{inz}-e^{-inz}}{2i}\right).
 \eeq
 
\no Consequently,

\begin{align}\nonumber
    \tilde{\mathcal{T}}_{\mu}(\cos{nz})=&j(\rho)\cos{nz}-\dfrac{1}{2}\left(j(\rho(n+\mu))e^{inz}+j(\rho(n-\mu))e^{-inz}\right)\\\nonumber
    =&(\jmath(\rho)-\jmath(\rho n))\cos{n z}-\\
    &~~\sum_{m=1}^\infty \left(\frac{\rho^{2m}\mu^{2m}}{(2m)!}\jmath^{(2m)}(\rho n)\cos{nz}+ i \frac{\rho^{2m-1}\mu^{2m-1}}{(2m-1)!}\jmath^{(2m-1)}(\rho n)\sin{nz}\right), 
\end{align}

\no and
\begin{align}\nonumber
    \tilde{\mathcal{T}}_{\mu}(\sin{nz})=&j(\rho)\sin{nz}-\dfrac{1}{2i}\left(j(\rho(n+\mu))e^{inz}-j(\rho(n-\mu))e^{-inz}\right)\\\nonumber
    =&(\jmath(\rho)-\jmath(\rho n))\sin{n z}+ \\
    &~~\sum_{m=1}^\infty \left(-\frac{\rho^{2m}\mu^{2m}}{(2m)!}\jmath^{(2m)}(\rho n)\sin{nz}+ i \frac{\rho^{2m-1}\mu^{2m-1}}{(2m-1)!}\jmath^{(2m-1)}(\rho n)\cos{nz}\right).
\end{align}

\no Using Lemma \ref{l:b1}, the action of $\tilde{\mathcal{T}}_{\mu}$ on the basis functions $\psi_1^+(a,\mu)$, $\psi_1^-(a,\mu)$ and $\psi_0(a,\mu)$ is given by

\begin{align}\nonumber
\tilde{\mathcal{T}}_{\mu}\psi_1^+(a,\mu)=&-\sum_{m=1}^\infty \left(\frac{\rho^{2m}\mu^{2m}}{(2m)!}\jmath^{(2m)}(\rho)\cos{z}+ i \frac{\rho^{2m-1}\mu^{2m-1}}{(2m-1)!}\jmath^{(2m-1)}(\rho)\sin{z}\right)+\\
&\quad \quad \quad \alpha N a^{N-1}T_1(N)+\mathcal{O}(a^N),\\\nonumber
\tilde{\mathcal{T}}_{\mu}\psi_1^-(a,\mu)=&\sum_{m=1}^\infty \left(-\frac{\rho^{2m}\mu^{2m}}{(2m)!}\jmath^{(2m)}(\rho)\sin{z}+ i \frac{\rho^{2m-1}\mu^{2m-1}}{(2m-1)!}\jmath^{(2m-1)}(\rho)\cos{z}\right)+\\
&\quad \quad \quad \alpha N a^{N-1}T_2(N)+\mathcal{O}(a^N),\\\nonumber
\tilde{\mathcal{T}}_{\mu}\psi_0(a,\mu)=&\frac{\jmath(\rho)-1}{\sqrt{2}}-\frac{1}{\sqrt{2}}\sum_{m=1}^\infty \left(\frac{\rho^{2m}\mu^{2m}}{(2m)!}\jmath^{(2m)}(0)\right)+\\
&\quad \quad \quad \frac{1}{\sqrt{2}}\alpha N a^{N-1}T_3(N)+\mathcal{O}(a^N),
\end{align}

\no where

\begin{align}
T_1(N)=&\cos^N z-\psi^{(1)}_{N},\\
T_2(N)=& \cos^{N-1} z\sin{z}-\psi^{(2)}_{N},\\
T_3(N)=&\cos^{N-1} z-\psi^{(3)}_{N}.
\end{align}
\no The values $\psi^{(1)}_{N}$, $\psi^{(2)}_{N}$, and $\psi^{(3)}_{N}$ are provided in \eqref{e:expand31}, \eqref{e:expand32} and \eqref{e:expand4} respectively. Using these we have,

\begin{align}\nonumber
\left<\tilde{\mathcal{T}}_{\mu}\psi_1^+(a,\mu),\psi_1^+(a,\mu)\right>=&-\sum_{m=1}^\infty \left(\frac{\rho^{2m}\mu^{2m}}{(2m)!}\jmath^{(2m)}(\rho)\right)+\\
& \quad \quad N^2 a^{2N-2}\left<T_1(N),\mathcal{D}_\rho^{-1} T_1(N)\right>+\mathcal{O}(a^{2N-1}),\\
\left<\tilde{\mathcal{T}}_{\mu}\psi_1^+(a,\mu),\psi_1^-(a,\mu)\right>=&-\sum_{m=1}^\infty \left( i \frac{\rho^{2m-1}\mu^{2m-1}}{(2m-1)!}\jmath^{(2m-1)}(\rho)\right)+\mathcal{O}(a^{2N-1}),\\
\left<\tilde{\mathcal{T}}_{\mu}\psi_1^+(a,\mu),\psi_0(a,\mu)\right>=&\frac{1}{\sqrt{2}}\alpha Na^{N-1}\left<T_1(N),1\right>+\nonumber\\
& \quad \quad \frac{1}{\sqrt{2}}N^2a^{2N-2}\left<T_1(N),\mathcal{D}_\rho^{-1}T_3(N)\right>+\mathcal{O}(a^{2N-1}),\\
\left<\tilde{\mathcal{T}}_{\mu}\psi_1^-(a,\mu),\psi_1^+(a,\mu)\right>=&\sum_{m=1}^\infty \left( i \frac{\rho^{2m-1}\mu^{2m-1}}{(2m-1)!}\jmath^{(2m-1)}(\rho)\right)+\mathcal{O}(a^{2N-1}),\\
\left<\tilde{\mathcal{T}}_{\mu}\psi_1^-(a,\mu),\psi_1^-(a,\mu)\right>=&-\sum_{m=1}^\infty \left(  \frac{\rho^{2m}\mu^{2m}}{(2m)!}\jmath^{(2m)}(\rho)\right)+\nonumber\\
& \quad \quad N^2 a^{2N-2}\left<T_2(N),\mathcal{D}_\rho^{-1} T_2(N)\right>+\mathcal{O}(a^{2N-1}),\\
\left<\tilde{\mathcal{T}}_{\mu}\psi_1^-(a,\mu),\psi_0(a,\mu)\right>=& \mathcal{O}(a^{2N-1}),\\
\left<\tilde{\mathcal{T}}_{\mu}\psi_0(a,\mu),\psi_1^+(a,\mu)\right>=&\frac{1}{\sqrt{2}}N^2a^{2N-2}\left<T_3(N),\mathcal{D}_\rho^{-1}T_1(N)\right>+\mathcal{O}(a^{2N-1}),\\
\left<\tilde{\mathcal{T}}_{\mu}\psi_0(a,\mu),\psi_1^-(a,\mu)\right>=& \mathcal{O}(a^{2N-1}),\\
\left<\tilde{\mathcal{T}}_{\mu}\psi_0(a,\mu),\psi_0(a,\mu)\right>=& \jmath(\rho)-1-\sum_{m=1}^\infty \left(\frac{\rho^{2m}\mu^{2m}}{(2m)!}\jmath^{(2m)}(0)\right)+\nonumber \\
& \quad \quad \frac{1}{2}N^2a^{2N-2}\left<T_3(N),\mathcal{D}_\rho^{-1}T_3(N)\right>+\mathcal{O}(a^{2N-1}).
\end{align}

\no Next, the expansion of $\check{\mathcal{T}}_a$ using ~\eqref{eqta2} is written out for even and odd $N$ separately. 



\subsection{Even values of $\mathbf{N}$} We have

\begin{align}
\check{\mathcal{T}}_a=c_{2N-2}a^{2N-2}-Na^{N-1}\cos^{N-1}{z}-N(N-1)a^{2N-2}(\cos^{N-2}{z})\eta_N+\mathcal{O}(a^{2N-1}).
\end{align}

\no The action of $\check{\mathcal{T}}_a$ on the basis functions $\psi_1^+(a,\mu)$, $\psi_1^-(a,\mu)$, and $\psi_0(a,\mu)$ is given by

\begin{align}\nonumber
\check{\mathcal{T}}_a\psi_1^+(a,\mu)=&c_{2N-2}\, a^{2N-2}\cos{z}-N a^{N-1}\cos^N{z}-N(N-1)a^{2N-2}\eta_N \cos^{N-1}{z}-\\&\quad \quad N^2a^{2N-2}\cos^{N-1}{z}\mathcal{D}_\rho^{-1}(T_1(N))_{\mathrm{even}}+\mathcal{O}(a^{2N-1}),\\\nonumber
\check{\mathcal{T}}_a\psi_1^-(a,\mu)=&c_{2N-2}\, a^{2N-2}\sin{z}-N a^{N-1}\cos^{N-1}{z}\sin{z}-\nonumber\\&\quad N(N-1)a^{2N-2}\eta_N \cos^{N-2}{z}\sin{z}-N^2a^{2N-2}\cos^{N-1}{z}\mathcal{D}_\rho^{-1}(T_2(N))_{\mathrm{even}}+\nonumber\\ &\quad \mathcal{O}(a^{2N-1}),\\\nonumber
\check{\mathcal{T}}_a\psi_0(a,\mu)=&\frac{1}{\sqrt{2}}c_{2N-2}\, a^{2N-2}-\frac{1}{\sqrt{2}}N a^{N-1}\cos^{N-1}{z}-\frac{1}{\sqrt{2}}N(N-1)a^{2N-2}\eta_N \cos^{N-2}{z}-\\&\quad \frac{1}{\sqrt{2}}N^2a^{2N-2}\cos^{N-1}{z}\mathcal{D}_\rho^{-1}(T_3(N))_{\mathrm{even}}+\mathcal{O}(a^{2N-1}),
\end{align}

\no where $(T_i(N))_{\mathrm{even}}$ is the value of $T_i(N)$ for even $N$:

\begin{align}
(T_1(N))_{\mathrm{even}}=&\cos^N z-\dfrac{1}{2}\langle \cos^Nz,1\rangle,\\
(T_2(N))_{\mathrm{even}}=& \cos^{N-1} z\sin{z},\\
(T_3(N))_{\mathrm{even}}=&\cos^{N-1} z-\cos{z}\langle \cos^{N-1}z,\cos{z}\rangle.
\end{align}

\no Using these, 
 
\begin{align*}
\left<\check{\mathcal{T}}_a\psi_1^+(a,\mu),\psi_1^+(a,\mu)\right>=&c_{2N-2} a^{2N-2}-Na^{N-1}\left<\cos^N{z},\cos{z}\right>-\\
&N(N-1)a^{2N-2}\left<\eta_N\cos^{N-1}{z},\cos{z}\right>-\nonumber\\&\quad 2N^2a^{2N-2}\left<\cos^N{z},\mathcal{D}_\rho^{-1}(T_1(N))_{\mathrm{even}}\right>+\mathcal{O}(a^{2N-1})\nonumber,\\
\left<\check{\mathcal{T}}_a\psi_1^+(a,\mu),\psi_1^-(a,\mu)\right>=& \mathcal{O}(a^{2N-1}),\nonumber\\
\left<\check{\mathcal{T}}_a\psi_1^+(a,\mu),\psi_0(a,\mu)\right>=&-\frac{1}{\sqrt{2}}Na^{N-1}\left<\cos^N{z},1\right>+\mathcal{O}(a^{2N-1}),\nonumber\\
\left<\check{\mathcal{T}}_a\psi_1^-(a,\mu),\psi_1^+(a,\mu)\right>=& \mathcal{O}(a^{2N-1}),\nonumber\\
\left<\check{\mathcal{T}}_a\psi_1^-(a,\mu),\psi_1^-(a,\mu)\right>=&c_{2N-2}a^{2N-2}-N(N-1)a^{2N-2}\left<\eta_N\cos^{N-2}{z}\sin{z},\sin{z}\right>-\nonumber\\&\quad 2N^2a^{2N-2}\left<\cos^{N-1}{z}\mathcal{D}_\rho^{-1}(T_2(N))_{\mathrm{even}},\sin{z}\right>+\mathcal{O}(a^{2N-1})\nonumber,\nonumber\\
\left<\check{\mathcal{T}}_a\psi_1^-(a,\mu),\psi_0(a,\mu)\right>=&\mathcal{O}(a^{2N-1}),\nonumber\\
\left<\check{\mathcal{T}}_a\psi_0(a,\mu),\psi_1^+(a,\mu)\right>=& -\frac{1}{\sqrt{2}}Na^{N-1}\left<\cos^N{z},1\right>+\mathcal{O}(a^{2N-1}),\nonumber\\
\left<\check{\mathcal{T}}_a\psi_0(a,\mu),\psi_1^-(a,\mu)\right>=& \mathcal{O}(a^{2N-1}),\nonumber\\
\left<\check{\mathcal{T}}_a\psi_0(a,\mu),\psi_0(a,\mu)\right>=&c_{2N-2}a^{2N-2}-\frac{1}{2}N(N-1)a^{2N-2}\left<\eta_N\cos^{N-2}{z},1\right>-\nonumber\\&\quad N^2a^{2N-2}\left<\cos^{N-1}{z}\mathcal{D}_\rho^{-1}(T_3(N))_{\mathrm{even}},1\right>+\mathcal{O}(a^{2N-1}).
\end{align*}

Using \eqref{e:dc}, we have

\beq\label{e:dcm}
\left<\mathcal{T}_{a,\mu}\phi_1,\phi_2\right>=\left<\tilde{\mathcal{T}}_{\mu}\phi_1,\phi_2\right> +\left<\check{\mathcal{T}}_a\phi_1,\phi_2\right>,
\eeq

\no where $\phi_1,\phi_2\in\{\psi_1^+(a,\mu),\psi_1^-(a,\mu),\psi_0(a,\mu)\}$.
Applying \eqref{e:dcm},

\begin{align}\nonumber
\langle\mathcal{T}_{a,\mu}\psi_1^+(a,\mu),\psi_1^+(a,\mu)\rangle=&-\sum_{m=1}^\infty \left(\frac{\rho^{2m}\mu^{2m}}{(2m)!}\jmath^{(2m)}(\rho)\right)+\\\nonumber
&N^2 a^{2N-2}\left<(T_1(N))_{\mathrm{even}},\mathcal{D}_\rho^{-1} (T_1(N))_{\mathrm{even}}\right>+\\\nonumber
&c_{2N-2} a^{2N-2}-Na^{N-1}\left<\cos^N{z},\cos{z}\right>-\\\nonumber
&N(N-1)a^{2N-2}\left<\eta_N\cos^{N-1}{z},\cos{z}\right>-\\
&2N^2a^{2N-2}\left<\cos^N{z},\mathcal{D}_\rho^{-1}(T_1(N))_{\mathrm{even}}\right>+\mathcal{O}(a^{2N-1})
\end{align}

Since $N$ is even,

\beq\label{e:e39}
\langle\cos^N{z},\cos{z}\rangle=0.
\eeq 

Next, note that

\begin{align}
\langle\cos^Nz,\mathcal{D}_\rho^{-1}(T_1(N))_{\mathrm{even}}\rangle&=\left\langle\cos^N{z},\mathcal{D}_\rho^{-1}(\cos^N{z})-\dfrac{\langle \cos^Nz,1\rangle}{2(\jmath(\rho)-1)}\right\rangle,\label{e:e40}\\
\langle(T_1(N))_{\mathrm{even}},\mathcal{D}_\rho^{-1} (T_1(N))_{\mathrm{even}}\rangle&=\left\langle\cos^N{z},\mathcal{D}_\rho^{-1}(\cos^N{z})-\dfrac{\langle \cos^Nz,1\rangle}{2(\jmath(\rho)-1)}\right\rangle,\label{e:e41}
\end{align}

\no where we have used that

\[
\left\langle\dfrac{1}{2}\langle\cos^Nz,1\rangle,\mathcal{D}_\rho^{-1}(\cos^N{z})-\dfrac{\langle \cos^Nz,1\rangle}{2(\jmath(\rho)-1)}\right\rangle=0.
\]



Using \eqref{e:e39}, \eqref{e:e40} and \eqref{e:e41}, we obtain

\begin{align}
  \langle\mathcal{T}_{a,\mu}\psi_1^+(a,\mu),\psi_1^+(a,\mu)\rangle=&-\sum_{m=1}^{N-1} \left(\frac{\rho^{2m}\mu^{2m}}{(2m)!}\jmath^{(2m)}(\rho)\right)+2N(1-N)\left\langle\cos^Nz,\mathcal{D}_\rho^{-1}(\cos^Nz)\right\rangle +\nonumber\\
  &\quad N^2\left\langle\cos^Nz,\dfrac{\langle\cos^Nz,1\rangle}{2(\jmath(\rho)-1)}\right\rangle+\mathcal{O}(\mu^{2N-1}+a^{2N-1}).
\end{align}

Since $\langle(T_1(N))_{\mathrm{even}},1\rangle=0$ and $\langle(T_1(N))_{\mathrm{even}},\mathcal{D}_\rho^{-1}(T_3(N))_{\mathrm{even}}\rangle=0$, 

\begin{align}
\left\langle\mathcal{T}_{a,\mu}\psi_1^+(a,\mu),\psi_0(a,\mu)\right\rangle=\left\langle\mathcal{T}_{a,\mu}\psi_0(a,\mu),\psi_1^+(a,\mu)\right\rangle=-\frac{1}{\sqrt{2}}Na^{N-1}\left<\cos^N{z},1\right>+\mathcal{O}(a^{2N-1}).
\end{align}

\begin{align}
\langle\mathcal{T}_{a,\mu}\psi_1^-(a,\mu),\psi_1^-(a,\mu)\rangle&=-\sum_{m=1}^\infty \left(  \frac{\rho^{2m}\mu^{2m}}{(2m)!}\jmath^{(2m)}(\rho)\right)+ N^2 a^{2N-2}\left<T_2(N),\mathcal{D}_\rho^{-1} T_2(N)\right>+\nonumber\\&c_{2N-2}a^{2N-2}-N(N-1)a^{2N-2}\left<\eta_N\cos^{N-2}{z}\sin{z},\sin{z}\right>-\nonumber\\& 2N^2a^{2N-2}\left<\cos^{N-1}{z}\mathcal{D}_\rho^{-1}(T_2(N))_{\mathrm{even}},\sin{z}\right>+\mathcal{O}(a^{2N-1}).
\end{align}

\no Note that

\begin{align}
    \langle\cos^{N-2}z\eta_N\sin{z},\sin{z}\rangle&=\langle\eta_N\cos^{N-2}z,1-\cos^2z\rangle=\langle\eta_N\cos^{N-2}z,1\rangle-\langle\eta_N\cos^Nz,1\rangle,\label{e:e45}
\end{align}

\no and 

\begin{align}
    \langle\cos^{N-1}{z}\mathcal{D}_\rho^{-1}(T_2(N))_{\mathrm{even}},\sin{z}\rangle&=\langle\cos^{N-1}\sin{z},\mathcal{D}_\rho^{-1}(\cos^{N-1}\sin{z})\rangle\nonumber\\&=\dfrac{1}{N^2}\langle\partial_z(\cos^Nz),\mathcal{D}_\rho^{-1}(\partial_z(\cos^Nz))\rangle\nonumber\\&=-\dfrac{1}{N^2}\langle\partial_z^2(\cos^Nz),\mathcal{D}_\rho^{-1}(\cos^Nz)\rangle\nonumber\\&=-\dfrac{N-1}{N}\langle\cos^{N-2}z,\mathcal{D}_\rho^{-1}(\cos^Nz)\rangle+\langle\cos^Nz,\mathcal{D}_\rho^{-1}(\cos^Nz)\rangle.\label{e:e46}
\end{align}

Using \eqref{e:e45} and \eqref{e:e46},

\begin{align}
    \langle\mathcal{T}_{a,\mu}\psi_1^-(a,\mu),\psi_1^-(a,\mu)\rangle=-\sum_{m=1}^{N-1} \left(  \frac{\rho^{2m}\mu^{2m}}{(2m)!}\jmath^{(2m)}(\rho)\right)+\mathcal{O}(\mu^{2N-1}+a^{2N-1}).
\end{align}

\begin{align}\nonumber
\langle\mathcal{T}_{a,\mu}\psi_0(a,\mu),\psi_0(a,\mu)\rangle =
& \jmath(\rho)-1-\sum_{m=1}^{N-1} \left(\frac{\rho^{2m}\mu^{2m}}{(2m)!}\jmath^{(2m)}(0)\right)+N\langle \cos^Nz,\mathcal{D}_\rho^{-1}(\cos^Nz)\rangle-\\\nonumber
& \dfrac{1}{2}N(N-1)\langle\cos^{N-2}z,\mathcal{D}_\rho^{-1}(\cos^Nz)\rangle-\\
&\dfrac{1}{2}N^2\langle\cos^{N-1}z,\mathcal{D}_\rho^{-1}(\cos^{N-1}z-\cos{z}\langle\cos^{N-1}{z},\cos{z}\rangle)\rangle.
\end{align}

\subsection{Odd values of $N$} We have  

\begin{align}\label{e:ta}
\check{\mathcal{T}}_a=c_{N-1} a^{N-1}-N\alpha a^{N-1}\cos^{N-1}{z}+\mathcal{O}(a^{N}).
\end{align}

\no The action of $\check{\mathcal{T}}_a$ on the basis functions $\psi_1^+(a,\mu)$, $\psi_1^-(a,\mu)$, and $\psi_0(a,\mu)$ is given by

\begin{align}
\check{\mathcal{T}}_a\psi_1^+(a,\mu)=&c_{N-1} a^{N-1}\cos{z}-N\alpha a^{N-1}\cos^N{z}+\mathcal{O}(a^{N}),\\
\check{\mathcal{T}}_a\psi_1^-(a,\mu)=&c_{N-1} a^{N-1}\sin{z}-N\alpha a^{N-1}\cos^{N-1}{z}\sin{z}+\mathcal{O}(a^{N}),\\
\check{\mathcal{T}}_a\psi_0(a,\mu)=&\frac{1}{\sqrt{2}}c_{N-1} a^{N-1}-\frac{1}{\sqrt{2}}N\alpha a^{N-1}\cos^{N-1}{z}+\mathcal{O}(a^{N}).
\end{align}

\no Using these, 

\begin{align*}
\left<\check{\mathcal{T}}_a\psi_1^+(a,\mu),\psi_1^+(a,\mu)\right>=&c_{N-1}a^{N-1}-N\alpha a^{N-1}\left<\cos^N{z},\cos{z}\right>+\mathcal{O}(a^{N}),\\
\left<\check{\mathcal{T}}_a\psi_1^+(a,\mu),\psi_1^-(a,\mu)\right>=& \mathcal{O}(a^{N}),\\
\left<\check{\mathcal{T}}_a\psi_1^+(a,\mu),\psi_0(a,\mu)\right>=& \mathcal{O}(a^{N}),\\
\left<\check{\mathcal{T}}_a\psi_1^-(a,\mu),\psi_1^+(a,\mu)\right>=& \mathcal{O}(a^{N}),\\
\left<\check{\mathcal{T}}_a\psi_1^-(a,\mu),\psi_1^-(a,\mu)\right>=&c_{N-1}a^{N-1}-N\alpha a^{N-1}\left<\cos^{N-1}{z}\sin{z},\sin{z}\right>+\mathcal{O}(a^{N}),\\
\left<\check{\mathcal{T}}_a\psi_1^-(a,\mu),\psi_0(a,\mu)\right>=&\mathcal{O}(a^{N}),\\
\left<\check{\mathcal{T}}_a\psi_0(a,\mu),\psi_1^+(a,\mu)\right>=&\mathcal{O}(a^{N}),\\
\left<\check{\mathcal{T}}_a\psi_0(a,\mu),\psi_1^-(a,\mu)\right>=&\mathcal{O}(a^{N}),\\
\left<\check{\mathcal{T}}_a\psi_0(a,\mu),\psi_0(a,\mu)\right>=&c_{N-1}a^{N-1}-\frac{1}{2}N \alpha a^{N-1}\left<\cos^{N-1}{z},1\right>+\mathcal{O}(a^{N}).
\end{align*}

Combining the action of $\tilde{\mathcal{T}}_{\mu}$ and $\check{\mathcal{T}}_a$, we obtain

\begin{align}\nonumber
    \langle\mathcal{T}_{a,\mu}\psi_1^+(a,\mu),\psi_1^+(a,\mu)\rangle=&-\sum_{m=1}^{(N-1)/2} \left(\frac{\rho^{2m}\mu^{2m}}{(2m)!}\jmath^{(2m)}(\rho)\right)\\
    &-\alpha(N-1)a^{N-1}\langle\cos^{N}z,\cos{z}\rangle+\mathcal{O}(\mu^N+a^N),\\
    \langle\mathcal{T}_{a,\mu}\psi_1^-(a,\mu),\psi_1^-(a,\mu)\rangle=&-\sum_{m=1}^{N/2} \left(  \frac{\rho^{2m}\mu^{2m}}{(2m)!}\jmath^{(2m)}(\rho)\right)+\mathcal{O}(\mu^N+a^N),\\\nonumber
    \langle\mathcal{T}_{a,\mu}\psi_0(a,\mu),\psi_0(a,\mu)\rangle=&\jmath(\rho)-1-\sum_{m=1}^{(N-1)/2} \left(\frac{\rho^{2m}\mu^{2m}}{(2m)!}\jmath^{(2m)}(0)\right)\\
    &+a^{N-1}\alpha\langle\cos^Nz,\cos{z}\rangle-\dfrac{1}{2}a^{N-1}\alpha N\langle\cos^{N-1}z,1\rangle+\mathcal{O}(\mu^N+a^N).
\end{align}
This completes the proof of Proposition \ref{p:1}.

\section{Proof of the Block-Diagonalization}\label{a:5}
Our next goal is to conjugate the matrix $ Q_{a,\mu} $ into block-diagonal form near $ \mu = 0 $, in order to isolate a reduced $ 2 \times 2 $ system that governs the leading-order spectral behavior. This reduction plays a central role in revealing the bifurcation mechanism that leads to the figure-eight pattern in the spectrum.

At $\mu=0$ the Poisson block $A_0$ has rank \emph{two}. Thus
$Q_{a,0}=A_0T_{a,0}$ is singular.  
To factor out this singularity and to obtain a matrix whose entries have
regular Taylor expansions in $\mu$, we introduce a diagonal rescaling
$L(\mu)$ and conjugate
$Q_{a,\mu}$ accordingly.  The following lemma shows one can
separate out a global factor $\mu$, while maintaining the Hamiltonian structure.

\begin{lemma}[Rescaling for singular Hamiltonian structure]
\label{lem:AQT-rescaling}
Define a reversibility-preserving rescaling
matrix 

\[
L = \operatorname{diag}(\mu^{1/2},\; \mu^{-1/2},\; \mu^{1/2}),
\qquad \mu > 0.
\]

\no Then the rescaled matrix
\[
Q^{(1)}_{a,\mu} = L^{-1} Q_{a,\mu} L = \mu A^{(1)}_\mu \cdot T^{(1)}_{a,\mu}\qquad
A^{(1)}_\mu = L^{-1} A_\mu L^{-\ast}, \quad T^{(1)}_{a,\mu} = \dfrac{L^\ast T_{a,\mu} L}{\mu},
\]
is Hamiltonian and reversible, with the following properties.

\begin{enumerate}

\item 
$A^{(1)}_\mu$ is a skew-adjoint and reversible matrix.

\item $T^{(1)}_{a,\mu}$ is a self-adjoint and reversibility-preserving matrix analytic in $(a,\mu)$.
\
\item The matrix $Q^{(1)}_{a,\mu}$ is analytic in $(a,\mu)$ near $(0,0)$.

\item The entries of $Q^{(1)}_{a,\mu}$ can be expanded as 

\[
Q^{(1)}_{a,\mu} = Q_0 + \mu Q_1(a) + \mathcal{O}(\mu^2),
\]

\no where $Q_0 = A^{(1)}_0 T^{(1)}_{a,0}$, and $Q_1(a)$ is bounded and analytic in $a$.

\end{enumerate}
\end{lemma}

\begin{proof}
\textit{\bf Step 1: Motivation for rescaling.} 
The unscaled matrix $ A_\mu $ becomes rank-deficient as $ \mu \to 0 $, with
\[
A_0 = \begin{pmatrix}
0 & 1 & 0 \\
-1 & 0 & 0 \\
0 & 0 & 0
\end{pmatrix}, \quad \text{rank } 2.
\]
As a result, $ Q_{a,\mu} = A_\mu T_{a,\mu} $ is degenerate at $ \mu = 0 $, and its entries scale unevenly in $\mu$. This non-uniform scaling obstructs an analytic expansion in $\mu$, and particularly affects the solvability of the homological equation used for block diagonalization.

\smallskip
\textit{\bf Step 2: Choosing scaling exponents.} 
We introduce a diagonal scaling matrix:

\[
L = \operatorname{diag}(\mu^{\alpha_1}, \mu^{\alpha_2}, \mu^{\alpha_3}).
\]

\no Conjugating the matrix $Q_{a,\mu}$ with $L$, we obtain   

\beq
Q^{(1)}_{a,\mu}
= L^{-1} A_\mu T_{a,\mu} L
= A_\mu^{(1)} \cdot T_{a,\mu}^\ddagger,
\eeq

\no where  

\beq
A_\mu^{(1)} = L^{-1} A_\mu L^{-*},
\qquad
T_{a,\mu}^\ddagger = L^* T_{a,\mu} L,
\eeq

\no and  

\beq
A_\mu^{(1)}
=\begin{pmatrix}
i\mu^{1-2\alpha_1} & \mu^{-\alpha_1 - \alpha_2} & 0 \\
- \mu^{-\alpha_1 - \alpha_2} & i\mu^{1-2\alpha_2} & 0 \\
0 & 0 & i\mu^{1-2\alpha_3}
\end{pmatrix}.
\eeq

\no \no Since $A_0$ has rank two, the rescaling must normalize the Hamiltonian structure so that the leading-order operator obtained from $A_\mu^{(1)}$ is non-degenerate and of uniform size as $\mu\to 0$. In particular, the dominant entries should neither blow up nor vanish in the limit, but remain of order one.

From the expression of $A_\mu^{(1)}$, this requires the balancing conditions
\[
-\alpha_1-\alpha_2 = 0, \qquad
1-2\alpha_1 = 0, \qquad
1-2\alpha_3 = 0,
\]
\no which ensures that the Hamiltonian coupling and the diagonal terms remain finite and nonzero at leading order. These conditions yield
\[
\alpha_1=\dfrac12,\quad \alpha_2=-\dfrac12,\quad \alpha_3=\dfrac12.
\]

\no This choice yields the explicit scaling matrix
\[
L = 
\begin{pmatrix}
\mu^{1/2} & 0 & 0 \\
0 & \mu^{-1/2} & 0 \\
0 & 0 & \mu^{1/2}
\end{pmatrix},
\]
which conjugates the original matrix $ A_\mu $ into
\beq\label{e:amu}
A_\mu^{(1)}=
L^{-1} A_\mu L^{-\ast} =
\begin{pmatrix}
    i & 1 & 0 \\
    -1 & i\mu^2 & 0 \\
    0 & 0 & i
\end{pmatrix}.
\eeq

\no The rescaled matrix $T_{a,\mu}^\ddagger$ is given by

\beq 
T_{a,\mu}^\ddagger=L^{\ast}T_{a,\mu}L=\mu\begin{pmatrix}
\begin{array}{c|c}
        \Lambda_{1}^{(\mathrm{I})} & \Lambda_{2}^{(\mathrm{I})}\\
        \hline
        \Lambda_{2}^{(\mathrm{I})^\dagger} & \Lambda_{3}^{(\mathrm{I})}
        \end{array}
    \end{pmatrix}=\mu T_{a,\mu}^{(1)},
\eeq

\no where the last equality defines $T_{a,\mu}^{(1)}$. The matrix $\Lambda_{1}^{(\mathrm{I})}$ is a $2\times2$ matrix given by

\beq\label{e:l11}
\Lambda_{1}^{(\mathrm{I})}=
\begin{pmatrix}
    \alpha \Lambda_{11}^a(N)a^{2N-2}-\mu^2\Lambda_{11}^\mu(N,\mu)&i\Lambda_{12}(N,\mu)\\-i\Lambda_{12}(N,\mu) & -\Lambda_{11}^\mu(N,\mu)
\end{pmatrix},
\eeq

\no while 

\beq\label{e:l22}
\Lambda_{2}^{(\mathrm{I})}=\begin{pmatrix}
    \Lambda_{13}(N)a^{N-1}\\0
    \end{pmatrix}
\quad\text{and}\quad \Lambda_{3}^{(\mathrm{I})}=\Lambda_{3}.
\eeq

\no The explicit form of the functions $\Lambda_{ij}'s$ is provided in Proposition \ref{p:1}. 

The rescaled structure constructed above is uniformly bounded in $\mu$, and together with the analyticity of $T_{a,\mu}$ ensures that the matrix $Q^{(1)}_{a,\mu} =\mu A^{(1)}_\mu T^{(1)}_{a,\mu}$ admits an analytic expansion in $\mu$ near zero, with a well-defined and non-degenerate leading-order term.

\end{proof}

Next, we block-diagonalize the matrix $Q^{(1)}_{a,\mu}$.

\begin{lemma}\label{l:f2}
    Given $a$ and $\mu$ sufficiently small, there exists a $C^{r-2}$-smooth map 
    \beq
\Upsilon: (a,\mu) \longmapsto \mathbb{C}^2,
\eeq
defined by

\beq\label{e:up}
\Upsilon(a,\mu) 
= \frac{-\Lambda_{13}a^{N-1}}{\widetilde{\Lambda}_d^2} 
\begin{pmatrix}
-i\widetilde{\Lambda}_d+\mathcal{O}(a^2,\mu^2)\\
\widetilde{\Lambda}_b+\mathcal{O}(a,\mu)
\end{pmatrix},
\eeq
where $\widetilde{\Lambda}_d=\rho \jmath'(\rho) + \jmath(\rho) - 1$ and $\widetilde{\Lambda}_b=-\rho\jmath^\prime(\rho)-\frac{1}{2}\rho^2\jmath^{\prime\prime}(\rho)$,
such that the conjugation of the Hamiltonian and reversible matrix $Q^{(1)}_{a,\mu}$ with the symplectic and reversibility-preserving matrix $\exp(K)$ with 

\beq\label{e:k}
K=A_\mu^{(1)}\begin{pmatrix}
\begin{array}{c|c}
\mathbf{0} & \Upsilon(a,\mu) \\
\hline
\Upsilon(a,\mu)^\dagger & 0
\end{array}
\end{pmatrix}=A_\mu^{(1)}\,B,
\eeq

\no yields the Hamiltonian and reversible matrix

\beq
Q_{a,\mu}^{(2)}=\exp{
(K)}\,Q^{(1)}_{a,\mu}\,\exp{(K)}^{-1}=\mu\, A^{(1)}_\mu\, T^{(2)}_{a,\mu},
\eeq

\no where 

\beq
T_{a,\mu}^{(2)}=\exp{(K)}^{-*}\, T^{(1)}_{a,\mu}\,\exp{(K)}^{-1}=\begin{pmatrix}
\begin{array}{c|c}
\Lambda_1^{\mathrm{(II)}} & \Lambda_2^{\mathrm{(II)}} \\
\hline
\Lambda_2^{\mathrm{(II)}^\dagger} & \Lambda_3^{\mathrm{(II)}}
\end{array}
\end{pmatrix},
\eeq

\no with the $2\times 2$ symmetric and reversibility preserving matrix $\Lambda_1^{\mathrm{(II)}}$, the $2\times 1$ vector $\Lambda_2^{\mathrm{(II)}}$ and the scalar $\Lambda_3^{\mathrm{(II)}}$ given by

\beq
\Lambda_1^{\mathrm{(II)}}=
\begin{pmatrix}
    \alpha\Lambda_{11}^a a^{2N-2}-\dfrac{\Lambda^2_{13}a^{2N-2}}{\widetilde{\Lambda}_d}-\mu^2\Lambda^\mu_{11}& i\Lambda_{12}+\dfrac{i}{2}\dfrac{\Lambda^2_{13}a^{2N-2}}{\widetilde{\Lambda}^2_d}\widetilde{\Lambda}_b\\-i\Lambda_{12}-\dfrac{i}{2}\dfrac{\Lambda^2_{13}a^{2N-2}}{\widetilde{\Lambda}^2_d}\widetilde{\Lambda}_b & -\Lambda^\mu_{11}
\end{pmatrix},
\eeq

\beq
\Lambda_{2}^{\mathrm{II}}=\begin{pmatrix}
    \mathcal{O}(a^{3N-3})\\0,
\end{pmatrix}\quad \Lambda_3^{\mathrm{(II)}}=\Lambda_3^{\mathrm{(I)}}+\dfrac{\Lambda_{13}^2a^{2N-2}}{\widetilde{\Lambda}_d^2}(\widetilde{\Lambda}_d+\widetilde{\Lambda}_b).
\eeq
\end{lemma}

\begin{proof}
    For the matrix $B$ to be both self-adjoint and reversible, it must satisfy the condition (\ref{d:3})

\beq
B \circ \mathcal{P} = -\mathcal{P} \circ B,
\eeq

\no where

\beq
\mathcal{P} =
\begin{pmatrix}
\varrho & 0 & 0 \\
0 & -\varrho & 0 \\
0 & 0 & \varrho
\end{pmatrix},
\eeq

\no and $\varrho : z \mapsto \bar{z}$ denotes complex conjugation. A straightforward computation shows that this condition is equivalent to requiring that $\Upsilon(a,\mu)$ must be of the form

\beq
\Upsilon(a,\mu) = 
\begin{pmatrix}
i \, \Upsilon_1(a,\mu) \\
\Upsilon_2(a,\mu)
\end{pmatrix},
\eeq

\no for some real-valued functions $\Upsilon_1$ and $\Upsilon_2$. 
Consequently, by Lemma~\ref{l:Asym}, the matrix $\exp(K)$ is $A_\mu^{(1)}$-symplectic and preserves reversibility for all $\mu \neq 0$. Hence, by Definition~\ref{d:3} for $A_\mu^{(1)} $-symplecticity, we obtain

\[
\exp(K) \, A_\mu^{(1)} \, \exp(K)^* = A_\mu^{(1)}.
\]

We begin by decomposing $ Q_{a,\mu}^{(1)} $ into its block-diagonal and off-diagonal components:

\[
Q_{a,\mu}^{(1)} = \mu \left( D^{(1)} + N^{(1)} \right),
\]

\no where

\beq\label{e:d1}
D^{(1)} = 
\begin{pmatrix}
\begin{array}{c|c}
\widetilde{A}_\mu \Lambda_1^{(\mathrm{I})} & 0 \\ \hline
0^\dagger & i \Lambda_3^{(\mathrm{I})}
\end{array}
\end{pmatrix},
\qquad
N^{(1)} = 
\begin{pmatrix}
\begin{array}{c|c}
\mathbf{0} & \widetilde{A}_\mu \Lambda_2^{(\mathrm{I})} \\ \hline
i \Lambda_2^{(\mathrm{I})^\dagger} & 0
\end{array}
\end{pmatrix}.
\eeq

To simplify notation, we introduce the matrix

\beq\label{e:a1}
A_\mu^{(1)} = 
\begin{pmatrix}
\begin{array}{c|c}
\widetilde{A}_\mu & 0 \\ \hline
0^\dagger & i
\end{array}
\end{pmatrix},
\qquad
\text{with} \quad 
\widetilde{A}_\mu = 
\begin{pmatrix}
i & 1 \\ -1 & i\mu^2
\end{pmatrix}.
\eeq

\no We aim to block-diagonalize $ Q_{a,\mu}^{(1)} $ using a symplectic and reversibility-preserving transformation. To compute the conjugated matrix 

\[
Q^{(2)}_{a,\mu}=\exp(K)\,Q^{(1)}_{a,\mu}\,\exp(-K),
\]

\no we apply the Lie series expansion, formally derived from the Baker–Campbell–Hausdorff formula \cite{hall2015lie}. This yields a series of the form \footnote{The operator $\mathrm{ad}_A(B)$ stands for the adjoint action of 
$A$ on $B$ defined by $\mathrm{ad}_A(B)=[A,B]=AB-BA$. Terms like $\mathrm{ad}_K^2(N^{(1)})$ or $\mathrm{ad}_K^3(D^{(1)})$ denote nested commutators, i.e., \[\mathrm{ad}_K^2(N)=[K,[K,R]], \quad \mathrm{ad}_K^3(D)=[K,[K,[K,D]]].\]}:

\begin{align}\nonumber
Q^{(2)}_{a,\mu} 
=& \mu \Big( D^{(1)} + N^{(1)} + [K, D^{(1)}] + [K, N^{(1)}] + \frac{1}{2}[K,[K, D^{(1)}]] \Big) 
+\\\label{e:expbc}
&\tfrac{1}{2} \int_0^1 (1 - \tau)^2 e^{\tau K} \, \mathrm{ad}^3_K(D^{(1)}) \, e^{-\tau K} \, d\tau 
+ \int_0^1 (1 - \tau) e^{\tau K} \, \mathrm{ad}^2_K(N^{(1)}) \, e^{-\tau K} \, d\tau. 
\end{align}

\no This approach enables the construction of a near-block-diagonal form by canceling off-diagonal terms using the homological equation

\begin{equation}\label{eq:homo}
N^{(1)} + [K, D^{(1)}] = 0. 
\end{equation}

\no With the explicit block forms of $D^{(1)}$  and $N^{(1)}$ in \eqref{e:d1}, \eqref{eq:homo} becomes

\beq
\begin{pmatrix}
\begin{array}{c|c}
0 & (i\Lambda_{3}^{(\mathrm{I})} - \widetilde{A}_\mu \Lambda_1^{(\mathrm{I})}) \widetilde{A}_\mu \Upsilon + \widetilde{A}_\mu \Lambda_{2}^{(\mathrm{I})} \\ \hline
i \Upsilon^\dagger (\widetilde{A}_\mu \Lambda_1^{(\mathrm{I})} - i\Lambda_{3}^{(\mathrm{I})}) + i \Lambda_{2}^{(\mathrm{I})\dagger} & 0
\end{array}
\end{pmatrix} = 0.
\eeq

\no This matrix equation reduces to a vector equation,

\begin{equation}\label{eq:homo2}
(\widetilde{A}_\mu \Lambda_1^{(\mathrm{I})} - i\Lambda_{3}^{(\mathrm{I})})\widetilde{A}_\mu \Upsilon  = \widetilde{A}_\mu \Lambda_{2}^{(\mathrm{I})}. 
\end{equation}

\no Using \eqref{e:l11} and \eqref{e:l22}, 

\begin{align}\nonumber
&\widetilde{A}_\mu \Lambda_1^{(\mathrm{I})} - i\Lambda_{3}^{(\mathrm{I})} =\\ 
&\begin{pmatrix}
-i \Lambda_d\! +\!i\alpha\Lambda_{11}^aa^{2N-2}\!-\!i\mu^2\Lambda_{11}^\mu\!-\!i\alpha\Lambda_{33}^a\!+\!i\Lambda_{33}^\mu & \Lambda_b\\
-\alpha\Lambda_{11}^aa^{2N-2}-\Lambda_b\mu^2& \!-\!i \Lambda_d \!-\!i\mu^2\Lambda_{11}^\mu\!-\!i\alpha\Lambda_{33}^aa^{2N-2}\!+\!i\Lambda_{33}^\mu
\end{pmatrix}, 
\end{align}

\no where $\Lambda_d$ and $\Lambda_b$ are defined by

\beq
\Lambda_d=\Lambda_{12}+\Lambda_{33}, \quad \Lambda_b=-\Lambda_{12}-\Lambda_{11}^\mu.
\eeq

The determinant of this matrix is

\begin{equation}
\det(\widetilde{A}_\mu \Lambda_1^{(\mathrm{I})} - i\Lambda_{3}^{(\mathrm{I})})
= -(\rho \jmath'(\rho) + \jmath(\rho) - 1)^2 + \mathcal{O}(a^2, a\mu, \mu^2)
= -\widetilde{\Lambda}_d^2 + \mathcal{O}(a^2, a\mu, \mu^2),
\end{equation}

\no which is not zero for sufficiently small $a$ and $\mu$. Solving the homological equation \eqref{eq:homo2} for $\Upsilon$, we obtain

\beq
\Upsilon=\widetilde{A}_\mu ^{-1}\,(\widetilde{A}_\mu \Lambda_1^{(\mathrm{I})} - i\Lambda_{3}^{(\mathrm{I})})^{-1}\,\widetilde{A}_\mu \Lambda_{2}^{(\mathrm{I})}, 
\eeq

\no which leads to the expression stated in \eqref{e:up}.

Since the matrix $K$ satisfies the homological equation \eqref{eq:homo}, the Lie expansion of $Q^{(2)}_{a,\mu}$ simplifies to

\begin{equation}\label{e:d28}
Q^{(2)}_{a,\mu} = \mu\left( D^{(1)} + \tfrac{1}{2} [K, N^{(1)}] + \frac{1}{2} \int_0^1 (1 - \tau)^2 e^{\tau K} \, \mathrm{ad}_K^2(N^{(1)}) \, e^{-\tau K} \, d\tau \right). 
\end{equation}

\no To extract the block-diagonal structure, we focus on the leading-order term $\frac{1}{2}[K, N^{(1)}]$, which is both Hamiltonian and reversibility preserving. This correction takes the form

\begin{equation}
\frac{1}{2} A_\mu^{(1)}
\begin{pmatrix}
\begin{array}{c|c}
i\, \Upsilon \Lambda_2^{(\mathrm{I})\dagger} - i\, \Lambda_2^{(\mathrm{I})} \Upsilon^\dagger & 0 \\\hline 0 &
\Upsilon^\dagger \widetilde{A}_\mu \Lambda_2^{(\mathrm{I})} - \Lambda_2^{(\mathrm{I})\dagger} \widetilde{A}_\mu\Upsilon
\end{array}
\end{pmatrix}
=
A_\mu^{(1)} 
\begin{pmatrix}
\begin{array}{c|c}
\delta \Lambda_1^{(\mathrm{I})} & 0 \\ \hline
0 & \delta \Lambda_3^{(\mathrm{I})}
\end{array}
\end{pmatrix}, 
\end{equation}

\no where $\delta \Lambda_1^{(\mathrm{I})}$  is a $2 \times 2$ self-adjoint and reversibility-preserving matrix, and $\delta \Lambda_3^{(\mathrm{I})}$ is a real scalar.

These quantities expand as

\begin{align}
\delta \Lambda_1^{(\mathrm{I})} = \begin{pmatrix}
    -\dfrac{\Lambda^2_{13}}{\widetilde{\Lambda}_d}& +\dfrac{i}{2}\dfrac{\widetilde{\Lambda}^2_{13}}{\widetilde{\Lambda}^2_d}\widetilde{\Lambda}_b\\-\dfrac{i}{2}\dfrac{\Lambda^2_{13}}{\widetilde{\Lambda}^2_d}\widetilde{\Lambda}_b & 0
\end{pmatrix},\quad \delta\Lambda_3^{(\mathrm{I})}=\dfrac{\Lambda_{13}^2}{\widetilde{\Lambda}_d^2}(\widetilde{\Lambda}_d+\widetilde{\Lambda}_b), 
\end{align}

\no and the block-diagonal matrix $D^{(1)}+\frac{1}{2}[K,N^{(1)}]$ is of the form

\beq
A_\mu^{(1)}\begin{pmatrix}
    \begin{array}{c|c}
      \Lambda_1^{(\mathrm{I})}+  \delta\Lambda_1^{(\mathrm{I})} & 0 \\\hline
       0  & \Lambda_3^{(\mathrm{I})}+  \delta\Lambda_3^{(\mathrm{I})}
    \end{array}
\end{pmatrix}=A_\mu^{(1)}\begin{pmatrix}
    \begin{array}{c|c}
      \Lambda_1^{(\mathrm{II})} & 0 \\\hline
       0  & \Lambda_3^{(\mathrm{II})}
    \end{array}
\end{pmatrix},
\eeq

\no with $\Lambda_1^{(\mathrm{II})}$ and $\Lambda_3^{(\mathrm{II})}$ as given in Lemma \ref{l:f2}. Now, we estimate the remainder term arising in the Lie expansion of $Q^{(2)}_{a,\mu}$ in \eqref{e:d28}. Recall that 
\beq
N^{(1)} = \mathcal{O}(a^{N-1}), 
\qquad 
K = \mathcal{O}(a^{N-1}).
\eeq
Hence,
\beq
\mathrm{ad}_K(N^{(1)}) = [K,N^{(1)}] = \mathcal{O}(a^{2N-2}),
\eeq
and iterating once more,
\beq
\mathrm{ad}_K^2(N^{(1)}) = [K,[K,N^{(1)}]] = \mathcal{O}(a^{3N-3}).
\eeq
Using the identity
\beq
e^{\tau K} X e^{-\tau K}
= X + \tau [K,X] + \frac{\tau^2}{2}[K,[K,X]] + \cdots,
\eeq
we obtain
\beq
e^{\tau K}\,\mathrm{ad}_K^2(N^{(1)})\,e^{-\tau K}
= \mathcal{O}(a^{3N-3}),
\eeq
uniformly for $\tau \in [0,1]$. Therefore,
\beq
\frac{1}{2} \int_0^1 (1-\tau)^2 e^{\tau K} \,\mathrm{ad}_K^2(N^{(1)})\, e^{-\tau K} \, d\tau
= \mathcal{O}(a^{3N-3}).
\eeq

This completes the proof of Lemma \ref{l:f2}.

\end{proof}

\bibliographystyle{abbrv}
\bibliography{gkdv.bib}
\end{document}